\documentclass[11.5pt]{amsart}
\usepackage[letterpaper]{geometry}
\usepackage{amsmath,amsfonts,amsthm,amscd,amssymb,mathrsfs,amssymb, multirow, pst-node, comment}
\usepackage{enumerate}
\usepackage{stmaryrd}
\usepackage{graphicx}
\usepackage{verbatim}
\usepackage{mathtools}
\usepackage{tikz-cd} 
\newtheorem{theorem}{Theorem}

\newtheorem{proposition}[theorem]{Proposition}
\newtheorem{prop-def}[theorem]{Proposition-Definition}
\newtheorem{lemma}[theorem]{Lemma}

\newtheorem{corollary}[theorem]{Corollary}

\numberwithin{figure}{section}
\usepackage[colorlinks, linkcolor=blue, anchorcolor=black, citecolor=red]{hyperref}
\theoremstyle{definition}
\newtheorem{definition}[theorem]{Definition}
\newtheorem{remark}[theorem]{Remark}

\newcommand{\mcH}{\mathcal{H}}
\newcommand{\mcX}{\mathcal{X}}

\newcommand{\mcY}{\mathcal{Y}}
\newcommand{\mcZ}{\mathcal{Z}}

\newcommand{\mcL}{\mathcal{L}}

\newcommand{\mcJ}{\mathcal{J}}

\newcommand{\mcE}{\mathcal{E}}
\newcommand{\mcM}{\mathcal{M}}
\newcommand{\mcN}{\mathcal{N}}

\newcommand{\mcT}{\mathcal{T}}
\newcommand{\mcO}{\mathcal{O}}

\renewcommand{\AA}{\mathbb{A}}
\newcommand{\GG}{\mathbb{G}}
\newcommand{\PP}{\mathbb{P}}
\newcommand{\CC}{\mathbb{C}}
\newcommand{\NN}{\mathbb{N}}
\newcommand{\RR}{\mathbb{R}}
\newcommand{\ZZ}{\mathbb{Z}}
\newcommand{\QQ}{\mathbb{Q}}

\newcommand{\VV}{\mathbb{V}}
\newcommand{\TT}{\mathbb{T}}

\newcommand{\bA}{\mathbb{A}}

\newcommand{\NA}{{\rm NA}}
\newcommand{\MA}{{\rm MA}}

\newcommand{\bfE}{{\mathbf E}}
\newcommand{\bfJ}{{\mathbf J}}

\newcommand{\bfM}{{\mathbf M}}
\newcommand{\bfI}{{\mathbf I}}
\newcommand{\bfH}{{\mathbf H}}
\newcommand{\bfF}{{\mathbf F}}
\newcommand{\bfR}{{\mathbf R}}
\newcommand{\Fut}{{\mathbf {Fut}}}

\newcommand{\ddc}{{\rm dd^c}}
\renewcommand{\div}{{\rm div}}

\newcommand{\bvk}{{[\vec{k}]}}
\renewcommand{\max}{{\rm max}}
\newcommand{\FS}{{\rm FS}}

\newcommand{\an}{{\rm an}}
\newcommand{\gf}{\mathrm{gf}}
\newcommand{\IN}{\mathrm{IN}}
\newcommand{\Ric}{\mathrm{Ric}}
\newcommand{\rmd}{\mathrm{d}}

\newcommand{\cc}{\mathcal{C}}

\newcommand{\oo}{\mathcal{O}}

\newcommand*{\sheafhom}{\mathscr{H}\kern -.5pt om}
\newcommand*{\sheafext}{\mathscr{E}\kern -.5pt xt}

\newcommand{\Aut}{{\rm Aut}}

\newcommand{\ord}{{\rm ord}}
\newcommand{\bG}{\mathbb{G}}
\newcommand{\triv}{{\rm triv}}
\newcommand{\mcF}{\mathcal{F}}
\newcommand{\cF}{\mathcal{F}}
\newcommand{\vol}{{\rm vol}}

\newcommand{\mcD}{\mathcal{D}}

\newcommand{\PSH}{{\rm PSH}}
\newcommand{\PL}{{\rm PL}}

\newcommand{\fa}{\mathfrak{a}}

\newcommand{\fb}{\mathfrak{b}}

\newcommand{\rv}{\mathrm v}
\newcommand{\rw}{\mathrm w}
\newcommand{\ext}{\mathrm{ext}}
\newcommand{\aut}{\mathfrak{aut}}

\renewcommand{\tilde}{\widetilde}

\allowdisplaybreaks[4]

\numberwithin{equation}{section}
\numberwithin{theorem}{section}
\numberwithin{table}{section}
\numberwithin{table}{section}

\begin{document}
\bibliographystyle{amsalpha}
\title{$\mathbb{G}$-uniform weighted K-stability for models on klt varieties}

\author{Jiyuan Han and Yaxiong Liu}

\address{Jiyuan Han, Institute for Theoretical Sciences,
Westlake University,
No.600 Dunyu Road, 
Hangzhou, 310030, China}
\email{hanjiyuan@westlake.edu.cn}

\address{Yaxiong Liu, Department of Mathematics,
University of Maryland,
4176 Campus Dr,
College Park, MD 20742, USA}
\email{yxliu238@umd.edu}

\thanks{}
\date{}

\subjclass[2020]{}
\begin{abstract}

In this paper, we make a generalization of the results in \cite{Li22a} to the singular and weighted setting. 
In particular, under the assumption of the envelope conjecture, we show that on a polarized projective klt variety, the $\mathbb{G}$-uniform weighted K-stability for models implies the $\mathbb{G}$-coercivity of the weighted Mabuchi functional. In the toric case (the envelope conjecture is known to be true), we further show that the $(\mathbb{C}^{\times})^n$-uniform $(\mathrm{v},\mathrm{w}\cdot\ell_{\mathrm{ext}})$-weighted K-stability is preserved when perturbing the polarization on the resolution, which implies the existence of the weighted extremal metric(s) on the resolution if the weight function $\mathrm{v}$ is log-concave.
\end{abstract}
\maketitle
\setcounter{tocdepth}{1}
\tableofcontents
\section{Introduction}
\label{intro}

This is a sequential paper of the work \cite{HL25}.
In this manuscript, we study the algebraic side of the (generalized) Yau--Tian--Donaldson conjecture, which predicts an equivalence between the existence of weighted cscK metric(s) and the $\GG$-uniform weighted K-stability condition. In \cite{Li22a}, the author shows that on a polarized smooth vareity $(X,L)$, $\GG$-uniform K-stability for models implies the existence of cscK metric(s). In \cite{BJ23}, the authors furthermore show that uniform K-stability for models is equivalent to divisorial stability. In this work, we make a generalization of the previous works to the singular and weighted setting. 
In particular, under the assumption of the envelope conjecture, we show that on a polarized projective variety with at most klt singularities, $\GG$-uniform weighted K-stability for models implies that the weighted Mabuchi functional is $\GG$-coercive.

Our setting is stated as follows.
Let $X$ be a projective variety with at most klt singularities. Let $L$ be an ample $\QQ$-line bundle on $X$. Let $\omega$ be a smooth K\"ahler form  (as a restriction of a smooth K\"ahler form defined on an ambient projective space)  in the class $c_1(L)$.
Let $T$ be an $r$-dimensional real torus generated by Hamiltonian vector fields of $(X, \omega)$. 
The complexified torus $T_\CC$ acts holomorphically on $(X,\omega)$. 
The Lie algebra $\mathfrak{t} := \mathrm{Lie} (T)$ is identified with $\underline{N}_\RR \simeq \mathbb{R}^r$.
Similarly, we identify the dual $\mathfrak{t}^*$ with $\underline{M}_\RR \simeq \mathbb{R}^r$ with the corresponding dual real coordinates, which will be denoted by 
$\underline{y}:=(\underline{y}_1,\cdots,\underline{y}_{r})$. 

Let $\{\xi^{\alpha}\}$ be a standard basis of $\mathfrak{t}$, and let 
$m^{\xi^\alpha}_\varphi\in C^{\infty}(X)$ be the associated Hamiltonian function for a K\"ahler form $\omega_{\varphi}\in [\omega]$. 
We have
\begin{equation*}
	\iota_{\xi^{\alpha}}\omega_\varphi
	=\dfrac{\sqrt{-1}}{2\pi}\bar{\partial} m^{\xi^\alpha}_\varphi.
\end{equation*}
By \cite{Ati82}, the image of 
 $m_{\varphi}:X\rightarrow \mathfrak{t}^*=\mathbb{R}^r$
is a convex polytope $P$, which is called the \textit{associated moment polytope}. It is explicitly given by
\begin{equation*}
	m_{\varphi}(x)
	:=(m_\varphi^{\xi^1}(x),\cdots,m_\varphi^{\xi^r}(x)), \quad
	P:=m_{\varphi}(X).
\end{equation*}

\begin{definition} 
    We define the Lie algebra
    \begin{align*}
        {{\aut}}_T(X) = \{\xi\in {{\aut}(X)}: [\xi, c]=0, \forall c\in \mathfrak{t}_\CC\}.
    \end{align*}
    Let ${\Aut}_T(X)$ be the connected subgroup of $\Aut(X)$ that is generated by ${\aut}_T(X)$. 
    Let $\GG = K_\CC \subset {\Aut}_T(X)$ be a reductive Lie group, and $K$ contains a maximal torus of ${\Aut}_T(X)$.
    Denote $\TT$ as the center of $\GG$. Note that, since $T$ is a subgroup of the center of ${\Aut}_T(X)$ and $K$ contains a maximal torus, we have $T\subset K$.
\end{definition}

The following is our main theorem.

\begin{theorem}
\label{theorem_A}
    Let $(X,L)$ be a polarized projective variety with at most klt singularities. Assume $(X,L)$ satisfies the envelope property $($see \textup{Definition~\ref{def:envelope_property}}$)$. Let $\rv\in C^\infty(P,\RR_{>0})$ and $\rw\in C^\infty(P,\RR)$. If $(X,L)$ is $\GG$-uniformly weighted K-stable for models, then the weighted Mabuchi functional $\bfM_{\rm v,w}$ is $\GG$-coercive.
\end{theorem}
\begin{remark}
   
        In this paper, the envelope property (see Definition~\ref{def:envelope_property}) plays an important role in reducing arguments on $X$ to arguments on its resolution. When $X$ is smooth, the envelope property holds. See \cite{BFJ16, BJ22, BJ24a}. In Section \ref{toric}, we show that the envelope property is satisfied if $X$ is toric.

\end{remark}

The $\GG$-coercivity of the weighted Mabuchi functional has a strong connection with the existence of the weighted cscK metric(s).
When $X$ is smooth and $\rv$ is log-concave, as shown in \cite{BDL20}, \cite{CC21b}, \cite{AJL23}, \cite{HL25}, \cite{DNJL24}, the $\bG$-coercivity of weighted Mabuchi functional is equivalent to the existence of weighted cscK metric(s).
When $X$ is a projective klt variety that admits a \textit{resolution of Fano type} and ${\rm v}$ is log-concave, in \cite{PT24} (which is based on \cite{BJT24}), the authors showed that the $\GG$-coercivity of $(\rv,\rw\cdot\ell_\ext)$-weighted Mabuchi functional implies the existence of singular weighted extremal metrics. See also \cite{LTW21,Sze24,GS25,FGS25}.
After the post of the first version of this paper, we were informed of the breakthroughs made by Boucksom--Jonsson \cite{BJ25} and Darvas--Zhang \cite{DZ25}. They prove two different versions of (weighted)-cscK Yau--Tian--Donaldson conjecture for smooth projective varieties independently.

In the toric setting,  we further show that the $\GG$-uniform $(\rv,\rw\cdot\ell_\ext)$-weighted K-stability is an open condition when perturbing the polarization along the resolution. 
In particular, by combining Theorem~\ref{theorem_A} and \cite[Theorem~1.2]{HL25}, such openness implies the existence of toric weighted extremal metric(s) on the resolution of a toric variety.

\begin{theorem}
\label{theorem_B}
    Let $(X,L)$ be a polarized projective toric variety with at most klt singularities and $\rv,\rw\in C^\infty(P,\RR_{>0})$. Let $\pi:Y\rightarrow X$ be a resolution of $X$. Let $L_\epsilon =\pi^* L-\epsilon E$ be an ample $\RR$-line bundle on $Y$ for $0<\epsilon\ll 1$, where $E$ is an exceptional divisor. Assume $(X,L)$ is $(\CC^{\times})^n$-uniformly $(\rv,\rw\cdot\ell_\ext)$-weighted K-stable. Then there exists $\epsilon_0>0$, such that for $0<\epsilon<\epsilon_0$, $(Y,L_\epsilon)$ is $(\CC^{\times})^n$-uniformly $(\rv,\rw\cdot\ell_{\ext,\epsilon})$-weighted K-stable. In particular, if ${\rm v}$ is log-concave, there exists a weighted extremal metric on $(Y,L_\epsilon)$.
\end{theorem}

\subsection*{Organization}
\begin{itemize}
    \item In Section~\ref{preliminaries}, we review some background material from non-Archimedean pluripotential theory developed by Boucksom--Jonsson \cite{BJ22, BJ24}. We also recall weighted Archimedean pluripotential theory in \cite{HL23}.
    \item In Section~\ref{Weighted non-Archimedean Monge-Ampere measure, weighted non-Archimedean functionals and geodesic rays}, based on \cite{HL23}, we define weighted non-Archimedean Monge-Amp\`ere measure and functionals. We make a slight generalization of \cite{BJ22} and prove a weighted non-Archimedean Calabi--Yau theorem provided the envelope property on klt variety. 
    Finally, we study the maximal geodesic ray and prove the slope formula needed later.
    \item In Section~\ref{w-K-stability}, we provide a proof of Theorem \ref{theorem_A}.
    \item In Section~\ref{toric}, we study $\GG$-uniform weighted K-stability of toric varieties and prove Theorem \ref{theorem_B}.
\end{itemize}

\subsection*{Acknowledgement} 
The authors would like to thank Chi Li, Zhuo Liu for their helpful discussions. We are grateful to S\'ebastien Boucksom and Mattias  Jonsson for pointing out a gap in the proof of the envelope conjecture for klt varieties in the previous version of our draft and giving us numerous helpful suggestions. We are grateful to Simon Jubert for his valuable comments on the integral formula of the toric Mabuchi functional which helps us to refine our statement. We are grateful to Tam\'as Darvas for the helpful discussion on the definition of Ricci energy for semiample classes.
J. Han is supported by National Key R\&D Program of China 2023YFA1009900, NSFC-20701003A012401, XHD23A0101.

\section{Preliminaries}
\label{preliminaries}

\subsection{Notation}
Throughout this article, we use the following conventions:
\begin{itemize}
\item $\varphi$ denotes an Archimedean psh metric on the line bundle $L$;
   \item $\Phi$, $\Psi$ denote Archimedean psh rays or geodesics on $L$;
\item $\phi$, $\psi$  denote non-Archimdean (psh) functions on the Berkovich analytification $X^\an$;
\item$\varphi_\fa$, $\varphi_\mcL$ denote non-Archimedean piecewise linear (PL) functions arising from a flag ideal $\mathfrak a$ or a model $\mathcal L$;
\item$\phi_\mcL\coloneqq{\rm P}_L(\varphi_\mcL)$ denotes a non-Archimedean model psh metric associated to a model $\mcL$.
\end{itemize}

\subsection{Background on non-Archimedean pluripotential theory}
\label{NA-pre}

In this subsection, we review some basic notions and facts from non-Archimedean pluripotential theory, as developed in \cite{BJ22}, \cite{BE21}, \cite{BJ24}.

Let $X$ be a normal projective variety over an algebraically closed field $k$.

\subsubsection{Berkovich analytification and psh functions}
The \textit{Berkovich analytification} $X^\an$ of $X$ (with respect to the trivial absolute value on $k$) is a compact Hausdorff topological space, whose points are \textit{semivaluations} on $X$, i.e., valuations $v:k(W)^*\rightarrow\RR$ for a subvariety $W\subset X$ (see \cite{BJ22}, \cite{Berk90}).
A \textit{divisorial valuation} $v$ on $X$ is a valuation of the form $v=t\ord_E$, where $t\in\QQ_{\geq0}$ ($t=0$ corresponds to trivial valuation) and $E$ is a prime divisor on a smooth birational model $\pi:Y\rightarrow X$.
The set $X^\div$ of divisorial valuations is dense in $X^\an$ (see \cite[Theorem 2.14]{BJ22}).

\begin{definition}[\cite{Li22a}, \cite{BHJ17}]
    Let $L$ be a $\QQ$-line bundle on $X$.

    A \textit{model} of $(X,L)$ is a flat family of projective varieties $\mcX\rightarrow\AA^1$ together with the following data:
    \begin{enumerate}[(i)]
        \item a $\GG_m$-action on $\mcX$ lifting the canonical action on $\AA^1$;
        \item a $\GG_m$-linearized $\QQ$-line bundle $\mcL$ on $\mcX$;
        \item an isomorphism $(\mcX_1,\mcL_1)\simeq (X,L)$.
    \end{enumerate}
    
    The \textit{trivial model} of $(X,L)$ is given by $(X\times\AA^1,L\times\AA^1)=:(X_{\AA^1},L_{\AA^1})$.

    If we forget about the data $\mcL$ and $L$, then we say that $\mcX$ is a \textit{model} of $X$.

    A model $\mcX$ of $X$ is called \textit{dominating} if there exists a $\GG_m$-equivariant birational morphism $\rho:\mcX\rightarrow X_{\AA^1}$.

    A model $(\mcX,\mcL)$ of $(X,L)$ is called \textit{big} if $\bar\mcL$ is a big $\QQ$-line bundle over $\bar\mcX$ and the stable base ideal of $m\bar\mcL$ is the same as the $\pi$-base ideal of $m\mcL$ for $m\gg1$, where $(\bar\mcX,\bar\mcL)$ is the canonical $\GG_m$-equivariant compactification of $(\mcX,\mcL)$ over $\PP^1$.

    A model $(\mcX,\mcL)$ of $(X,L)$ is called a \textit{test configuration} if $\mcL$ is semiample over $\AA^1$.
\end{definition}
For any dominating model $(\mcX,\mcL)$ of $(X,L)$, one associates a continuous function
\begin{align*}
    \varphi_\mcL(v)\coloneqq \sigma(v)(\mcL-\rho^*L_{\AA^1})
\end{align*}
for any $v\in X^\an$, where $\sigma:X^\an\rightarrow(X\times\AA^1\setminus\{0\})^\an\subset\bar\mcX^\an$ is the \textit{Gauss extension} (see \cite[Section~1.3]{BJ22} for more details).

A \textit{flag ideal} $\fa$ is a $\GG_m$-invariant, coherent fractional ideal sheaf on $X\times \AA^1$ (see \cite{BHJ17}, \cite{Oda13}).
One defines a continuous function $\varphi_\fa:X^\an\rightarrow\RR$ by setting $\varphi_\fa(v)\coloneqq -\sigma(v)(\fa)$ (see \cite[Section 2.1]{BJ22}).
\begin{definition}
    The space of \textit{piecewise linear} (PL) functions on $X^\an$ is defined as the $\QQ$-linear subspace $\PL(X^\an)\subset C^0(X^\an)$ generated by all functions $\varphi_\fa$ associated to flag ideals $\fa$.
\end{definition}
By \cite[Theorem 2.7]{BJ22}, each PL function $\psi\in\PL(X^\an)$ is of the form $\psi=\varphi_\mcL$ for some dominating model $(\mcX,\mcL)\rightarrow\AA^1$ for $(X,L)$.
As illustrated by \cite[Theorem 2.2]{BJ22}, the space $\PL(X^\an)$ is dense in $C^0(X^\an)$ under the uniform convergence. 

\begin{definition}
    Given a subgroup $\Lambda\subset\RR$ and a $\QQ$-line bundle $L$ on $X$, 
a $\Lambda$-\textit{rational, generically finite Fubini--Study function for} $L$ is a function $\phi:X^\an\rightarrow\RR$ of the following form
\begin{align*}
    \phi\coloneqq m^{-1}\max_j\{\log|s_j|+\lambda_j \},
\end{align*}
where $m\in\ZZ_{>0}$ such that $mL$ is an honest line bundle, $(s_j)_j$ is a finite set of nonzero global sections of $mL$, and $\lambda_j\in\Lambda$.
Denote $\mcH^\gf_\Lambda(X^\an,L^\an)$ all such functions.

Such function is called a  $\Lambda$-\textit{rational Fubini--Study function for} $L$ if $\phi$ is further finite valued on $X^\an$, i.e., $(s_j)$ is a finite set of $H^0(X,mL)$ without common zeros.
Denote $\mcH_\Lambda(X^\an,L^\an)$ all $\Lambda$-rational Fubini--Study functions.
For simplicity, one denotes $\mcH(X^\an,L^\an)=\mcH_\QQ(X^\an,L^\an)$.
\end{definition}
It is easy to check that
\begin{align*}
    \mcH^\gf_\RR(X^\an,L^\an)\neq\varnothing
    \Leftrightarrow \mcH^\gf_0(X^\an,L^\an)\neq\varnothing
    \Leftrightarrow L \text{ effective}; \\
    \mcH_\RR(X^\an,L^\an)\neq\varnothing
    \Leftrightarrow \mcH_0(X^\an,L^\an)\neq\varnothing
    \Leftrightarrow L \text{ semiample}.
\end{align*}

Fubini--Study functions and PL functions have the following relation.
\begin{proposition}[\textup{\cite[Proposition 2.25, Theorem 2.31, Corollary 2.32]{BJ22}}]
\label{equiv_Fubibi-Study_func}
    Given a semiample $\QQ$-line bundle $L$ on $X$, $\phi\in\mcH(X^\an,L^\an)$ iff 
    $\phi=m^{-1}\varphi_\fa$ for a flag ideal $\fa$ and $m\in\ZZ_{>0}$ such that $mL$ is an honest line bundle and $mL_{\AA^1}\otimes\fa$ is globally generated on $X_{\AA^1}$,
    iff
    $\phi=\varphi_\mcL$ for a semiample test configuration $(\mcX,\mcL)$ of $(X,L)$.

    Furthermore, if $L$ is ample, $\phi\in\mcH(X^\an,L^\an)$ iff  $\phi=\varphi_\mcL$ for an ample test configuration $(\mcX,\mcL)$ of $(X,L)$.
\end{proposition}

\begin{definition}
    Given a $\QQ$ line bundle $L$ on $X$, a $L$-\textit{psh function} $\phi:X^\an\rightarrow\RR\cup\{-\infty\}$ is a generically finite (that is, $\phi\neq-\infty$), upper semicontinuous function that can be written as the pointwise limit of a decreasing sequence $\phi_i\in\mcH^\gf_\RR(X^\an,L_i^\an)$ with $L_i\in{\rm Pic}(X)_\QQ$ such that $\lim_i c_1(L_i)=c_1(L)$ in $N^1(X)$.
    One denotes by $\PSH(X^\an,L^\an)$ the set of $L$-psh functions.
\end{definition}
Note that $\PSH(X^\an,L^\an)$ is nonempty only if $L$ is psef, whereas $\PSH(X^\an,L^\an)$ contains the constant functions iff $L$ is nef (see \cite[Section 4.1]{BJ22}).

When $L$ is ample, $\phi_i$ can be chosen in $\mcH(X^\an,L^\an)$ (see \cite[Theorem 4.15]{BJ22}).

\begin{definition}
\label{def:envelope_P}
    Let $L$ be a $\QQ$-line bundle on $X$.
    For any function $\psi: X^\an \rightarrow \RR\cup \{\pm \infty\}$, the $L$-\textit{psh envelope} of $\psi$ is defined as
    \begin{align*}
        {\rm P}_L(\psi) = \sup \{ \phi\in \PSH(X^\an, L^\an): \; \phi \leq \psi \}.
    \end{align*}
    The \textit{Fubini--Study envelope} of $\psi$ is defined as
    \begin{align*}
       {\rm Q}_L(\psi) = \sup \{ \phi\in \mcH^{\gf}(X^\an, L^\an): \; \phi \leq \psi \}. 
    \end{align*}
    When $L$ is ample,
    \begin{align*}
       {\rm Q}_L(\psi) = \sup \{ \phi\in \mcH(X^\an, L^\an): \; \phi \leq \psi \}. 
    \end{align*}
\end{definition}
\begin{definition}[\cite{BJ22}]
\label{def:envelope_property}
    We say $(X^\an, L^\an)$ has the \textit{envelope property} if for any increasing sequence $\varphi_j \in \PSH(X^\an, L^\an)$ that is uniformly bounded from above, the u.s.c regularization of the limit $(\lim_{j\to\infty} \varphi_j)^* $ is in $\PSH(X^\an, L^\an)$.
\end{definition}

\subsubsection{Norms and Fubini--Study operator}
\label{subsec:Norms_and_FS}
We follow the notations of \cite{BJ24}.

Let $L$ be a semiample $\QQ$-line bundle on $X$.
For any $d\in\NN$ such that $dL$ is an honest line bundle, one writes $R_d\coloneqq H^0(X,dL)$ and
\begin{align*}
    R^{(d)}=R(X,dL)=\bigoplus_{m\in\NN}R_{md},
\end{align*}
which is finite generated for all $d$ sufficiently divisible, since $L$ is semiample.
One writes $\mcN_\RR(R^{(d)})$ for the set of norms $\chi:R^{(d)}\rightarrow\RR$ that are
\begin{itemize}
    \item \textit{superadditive}, i.e., $\chi(fg)\geq\chi(f)+\chi(g)$ for $f,g\in R^{(d)}$;
    \item $k^*$-\textit{invariant}, i.e., $\chi(a\cdot f)=\chi(f)$ for $a\in k^*$ and $f\in R^{(d)}$, this is equivalent to that $\chi(\sum_ms_m)=\min_m\chi(s_m)$ where $s_m\in R^{(d)}_m$;
    \item \textit{linearly bounded}, i.e., there exists $C>0$ such that $|\chi|\leq Cm$ on $R^{(d)}_m\setminus\{0\}$ for all $m\geq1$.
\end{itemize}
Norms in $\mcN_\RR(R^{(d)})$ are in one-to-one correspondence with graded, linear bounded filtrations of $R^{(d)}$ as in \cite{BC11}.
When $d$ divides $d'$, there is a restriction map $\mcN_\RR(R^{(d)})\rightarrow\mcN_\RR(R^{(d')})$.
We set
\begin{align*}
    \mcN_\RR(X,L)\coloneqq \varinjlim_d\mcN_\RR(R^{(d)}).
\end{align*}
An element $\chi\in\mcN_\RR(X,L)$ is represented by a norm on some $R^{(d)}$.
For all $m$ sufficiently divisible, 
to define  $\chi|_{R_m}$, one needs to choose a representative of $\chi$ as a norm on some $R^{ (d)}$. But any other choice leads to the same norms $ \chi|_{R_m}\in \mcN_\RR(R_m)$.

A norm $\chi\in\mcN_\RR(X,L)$ is \textit{of finite type} if it is represented by a norm on some $R^{(d)}$ whose associated graded algebra ${\rm gr}_\chi R^{(d)}$ is of finite type. Denote by $\mcT_\RR(X,L)$ the set of finite type norms.
The element in $\mcT_\RR(X,L)$ can be interpreted as $\RR$-test configuration by the Rees construction (see \cite[Section 1.4]{BJ24}, \cite{HL24}).

As constructed in \cite{WN12}, any model $(\mcX,\mcL)$ induces an integral norm $\chi_{(\mcX,\mcL)}\in\mcN_\ZZ$, defined on $R_m=H^0(X,mL)$ for any $m\in\NN$ such that $m\mcL$ is a line bundle, whose corresponding filtration is defined as
\begin{align}
\label{eq:model_norm}
    \cF^\lambda R_m\coloneqq\{s\in R_m:t^{-\lambda}s\in H^0(\mcX,m\mcL)\}.
\end{align}

The \textit{canonical approximation} of $\chi\in\mcN_\RR(X,L)$ is a sequence $\chi_d\in\mcT_\RR(X,L)$ for $d\in\ZZ_{\geq1}$ sufficiently divisble, defined by letting $\chi_d$ be the (class of the) norm on $R^{(d)}$ generated in degree $1$ by $\chi_d$, i.e., $\chi_d|_{R^{(d)}_m}$ is the quotient norm of $S^m(\chi|_{R_d})$ under the canonical surjective map $S^m(R_d)\rightarrow R_{md}$.
If $d$ divides $d'$, then $\chi_d\leq\chi_{d'}\leq \chi$,

For any $\chi\in\mcN_\RR(X,L)$, one associates
\begin{align*}
    \FS_m(\chi)\coloneqq m^{-1}\FS_{mL}(\chi|_{R_m})
    =m^{-1}\sup_{s\in R_m\setminus\{0\}} \{\log|s|+\chi(s)\}\in\mcH_\RR(X^\an,L^\an)
\end{align*}
for $m$ sufficiently divisible.
By \cite[Lemma 2.12]{BJ24}, $\FS_m(\chi)$ is an increasing function of $m$ with respect to divisibility, and is further uniformly bounded, by linear boundedness of $\chi$.
The \textit{Fubini-Study} operator $\FS:\mcN_\RR(X,L)\rightarrow L^\infty(X^\an)$ is defined as $\FS(\chi)\coloneqq \lim_m\FS_m(\chi)=\sup_m\FS_m(\chi)$.

Conversely, one can define an operator $\IN:L^\infty(X^\an)\rightarrow\mcN_\RR(X,L)$.
Given a bounded function $\psi$, for any $m\in \NN$ such that $mL$ is a line bundle and $s\in R_m$, one defines
\begin{align*}
     {\rm IN}(\psi)(s) 
     = \inf_{v\in X^{\an}}\{v(s)+m\psi(v)\}
     =\inf_{X^\an}\{m\psi-\log|s|\}.
\end{align*}
In ``multiplicative" notation (\cite{BE21}), it is the usual supnorm $\|s\|_\psi=\sup_{X^\an}|s|e^{-m\psi}$.
\begin{proposition}[\textup{\cite[Theorem 2.29]{BJ24}, \cite[Theorem 7.26]{BE21}}]
\label{FS-IN=Q_L}
    For any $\psi\in L^\infty(X^\an)$, then $\FS(\IN(\psi))={\rm Q}_L(\psi)$, i.e.,
    ${\rm Q}_L(\psi)
        =\lim_m \FS_{m}(\IN(\psi)).$
\end{proposition}
Given any non-pluripolar set $\Sigma\subset X^\an$ and any $\psi\in L^\infty(\Sigma)$, as in \cite[Definition 6.1]{BJ24}, one defines
\begin{align*}
    {\rm IN}_\Sigma(\psi)(s) &= \inf_{v\in \Sigma}\{v(s)+m\psi(v)\},
\end{align*}
for any $s\in R_m$.
We have $ {\rm IN}_\Sigma(\psi) \geq {\rm IN}(\psi)$. %and $\phi|_{\Sigma} = {\rm FS}\big({\rm IN}_\Sigma(\phi)\big)|_{\Sigma}$. 
One denotes by $\mcN_\RR^\Sigma$ the set of norms $\IN_\Sigma(\phi)$, with $\phi$ ranging over bounded functions on $\Sigma$
\begin{lemma}[\textup{\cite[Proposition~6.5~(i)]{BJ24}}]
\label{IN_Sigma(FS)}
    Let $L$ be an ample line bundle on $X$. Then each $\chi\in\mcN_\RR^\Sigma$ satisfies $\chi=\IN_\Sigma(\phi)$ with $\phi=\FS(\chi)|_\Sigma$. Moreover, if $\chi=\IN_\Sigma(\psi)$, then $\FS(\chi)|_\Sigma\leq\psi$.
\end{lemma}
The set of \textit{divisorial norms} $\mcN_\RR^\div$ is defined as
\begin{align*}
    \mcN_\RR^\div\coloneqq \bigcup_{\Sigma\subset X^\div\text{ finite}}\mcN_\RR^\Sigma.
\end{align*}
Concretely, a divisorial norm is of the form $\chi=\max_i\{\chi_{v_i}+c_i\}$ for a finite set of divisorial valuations $(v _i)$ and $c_i\in\RR$, where $\chi_{v_i}(s)\coloneqq v_i(s)$ for $s\in R_m$.

\begin{proposition}
\label{env_cor}
Let $(X,L)$ be a polarized variety satisfying the envelope property. Let $\pi: Y \rightarrow X$ be a resolution of $X$. Then
    \begin{enumerate}[$(1)$]
        \item $\pi^*\PSH(X^\an,L^\an) = \PSH(Y^\an,\pi^*L^\an)$. 
        \item For any normal model $(\mcY,\mcL')$ of $(Y,\pi^*L)$, ${\rm Q}_{\pi^*L}(\varphi_{\mcL'}) = {\rm P}_{\pi^*L}(\varphi_{\mcL'}) \in \PSH(Y^\an,\pi^*L^\an) \cap C^0(Y^\an)$.
        \item For any normal model $(\mcY,\mcL')$ of $(Y,\pi^*L)$, there exists a model $(\mcX,\mcL)$ of $(X,L)$, such that ${\rm P}_{\pi^*L}(\varphi_{\mcL'}) = \pi^* {\rm P}_{L}(\varphi_{\mcL})$. 
        \item For any big model $(\mcY,\mcL')$ of $(Y,\pi^*L)$, let $\fa'_m$ be the base ideal of $|m\mcL'|$. Then $\phi_m = \varphi_{\mcL'} + \frac{1}{m}\varphi_{\fa'_m}$ converges to ${\rm P}_{\pi^*L}(\varphi_{\mcL'})$ uniformly.
        \item For any normal model $(\mcX,\mcL)$, we have $\pi_*\MA^\NA(\pi^*{\rm P}_L(\varphi_\mcL)) = \MA^\NA({\rm P}_L(\varphi_\mcL))$. More generally, this also applies for mixed Monge-Amp\`ere measures, 
        \begin{align*}
            \pi_*\MA^\NA(\phi_\triv^{[i]}, \pi^*{\rm P}_L(\varphi_\mcL)^{[n-i]}) = \MA^\NA(\phi_\triv^{[i]}, {\rm P}_L(\varphi_\mcL)^{[n-i]}).
        \end{align*}
    \end{enumerate}
\end{proposition}
\begin{proof}
    For (1), the statement is a direct consequence of  \cite[Lemma 5.13]{BJ22}.

    For (2), denote $R = \oplus_{m\geq 0} R_m$, $R_m = H^0(X,mL) = H^0(Y,m\pi^*L)$. Let $\chi_{(\mcY,\mcL')} $ be the norm on $R$ induced by the model $(\mcY,\mcL')$. 
    By \cite[Proposition~A.3]{BJ24} and \cite[Lemma~5.1]{BJ23},  $\IN(\varphi_{\mcL'})=\chi_{(\mcY,\mcL')}^{\rm hom}\in \mcN_\QQ^\div(Y,\pi^*L)=\mcN_\QQ^\div(X,L)$ (see \cite[Section~2]{BJ24} for the definition of $\chi^{\rm hom}$). Thus there exists a model $(\mcX,\mcL)$ such that $\chi_{(\mcX,\mcL)}^{\rm hom}=\chi_{(\mcY,\mcL')}^{\rm hom}$. It follows that $\FS(\IN(\varphi_{\mcL'}))=\pi^*\FS(\IN(\varphi_\mcL))$. 
    By Proposition~\ref{FS-IN=Q_L},  ${\rm Q}_{\pi^* L}(\varphi_{\mcL'}) = \pi^* {\rm Q}_L(\varphi_\mcL)$. Since the envelope property holds, ${\rm Q}_L(\varphi_\mcL) \in \PSH(X^\an,L^\an)$. Then ${\rm Q}_{\pi^*L}(\varphi_{\mcL'}) \in \PSH(Y^\an, \pi^*L^\an)$. Since $(Y^\an,\pi^*L^\an)$ satisfies the envelope property (see \cite[Theorem~A]{BJ24a}), then ${\rm P}_{\pi^*L}(\varphi_{\mcL'}) \in \PSH(Y^\an,\pi^*L^\an)$. By \cite[Lemma~1.5]{BJ22}, the psh functions ${\rm P}_{\pi^*L}(\varphi_{\mcL'})$, ${\rm Q}_{\pi^*L}(\varphi_{\mcL'})$ are equal on a Zariski open set $(Y\setminus \mathbb{B}_+(\pi^*L))^\an$. Then they are equal on $Y^\an$.

    For (3),  as above in the proof of (2), there exists a model $(\mcX,\mcL)$ such that ${\rm Q}_{\pi^* L}(\varphi_{\mcL'}) = \pi^* {\rm Q}_L(\varphi_\mcL)$.
By (2) and \cite[Lemma~5.19]{BJ22}, then ${\rm P}_{\pi^*L}(\varphi_{\mcL'})={\rm Q}_{\pi^*L}(\varphi_{\mcL'}) = \pi^* {\rm Q}_L(\varphi_\mcL)=\pi^* {\rm P}_L(\varphi_\mcL)$.
    
%the normal model $(\mcY,\mcL')$ induces a norm $\chi$ on $\oplus_{m\geq0}H^0(Y,m\pi^*L)$. From the construction, we know that $\chi$ is of the form $\inf\{\chi_{v_i}(s)+m\cdot \lambda_i\}$ for $s\in R_m$. Then ${\rm Q}_L(\chi)$ can be induced from a model $(\mcX,\mcL)$. Therefore, the psh function induced by $(\mcY,\mcL')$ is equal to $\pi^*{\rm Q}_L(\phi_\mcL) = \pi^*{\rm P}_L(\phi_\mcL)$.

    For (4), since the model is big, $|m\mcL'|\neq \emptyset$ for $m$ large. By \cite[Lemma 1.6]{BJ24a}, we have $\lim_{m\to\infty}\nearrow \phi_m = {\rm Q}_{\pi^*L}(\varphi_{\mcL'})$. The uniform convergence then follows from Dini's lemma.

    For (5), by (4), we have $\phi_m$ converges to ${\rm P}_L(\varphi_\mcL)$ uniformly. Then $\MA^\NA(\phi_m)$ converges to $\MA^\NA({\rm P}_L(\varphi_\mcL))$ weakly as Radon measures. Then there exists a model $(\mcX_m,\mcL_m)$ with relatively semi-ample $\QQ$-line bundle $\mcL_m$, such that $\phi_m = \varphi_{\mcL_m}$. Let $\pi: (\mcY_m,\pi^*\mcL_m) \rightarrow (\mcX_m,\mcL_m)$ be the birational morphism induced by the pullback diagram 
    \begin{equation}
        \begin{tikzcd}
        \mcY_m \arrow[r, "\rho"] \arrow[d, "\pi"] 
        & Y_{\bA^1} \arrow[d, "\pi"] \\
        \mcX_m \arrow[r, "\rho"] & X_{\bA^1} 
        \end{tikzcd}.
    \end{equation}

    For any piecewise-linear function $\psi$ on $X^\an$, we can find a model $\mu: (\mcX', \mcL')\rightarrow \mcX_m$, such that $\psi = \varphi_{\mcL'}$, and 
    \begin{align*}
        \int_{X^\an} \psi \; \MA^\NA(\phi_m) &= \mcL'\cdot (\mu^*\mcL_m)^n \\
            &= \pi^*\mcL'\cdot (\pi^*\mu^*\mcL_m)^n \\
            &= \int_{Y^\an} \pi^*\psi \; \MA^\NA(\pi^*\phi_m).
    \end{align*}
    By taking $m\to\infty$, we have
    \begin{align*}
        \int_{X^\an} \psi \; \MA^\NA({\rm P}_L(\varphi_\mcL)) = \int_{Y^\an} \pi^*\psi \; \MA^\NA(\pi^*{\rm P}_L(\varphi_\mcL)).
    \end{align*}
    Since the set of piecewise linear functions is dense in $C^0(X^\an)$, the proof is finished.
\end{proof}

\subsection{Weighted Archimedean functionals and geodesic}
\label{w-fun-pre}
In this subsection, we work on the complex field $\CC$ and follow notations in \cite{HL23}.

Let $L$ be an ample $\QQ$-line bundle on $X$.
Fix a smooth K\"ahler metric $\omega\in c_1(L)$ with $\omega=\ddc \varphi_0$, the space of $K$-invariant psh metrics on $L$ is defined as
\begin{align*}
    \PSH_K(X,L)
    \coloneqq \{\varphi=\varphi_0+u:u\in L^1_{\rm loc},\text{ and locally } \varphi\text{ is a }K\text{-invariant psh function} \}.
\end{align*}
Let $\rv$ be a smooth positive function on the moment polytope $P$. In \cite{HL23}, \cite{BWN14}, the authors defined the weighted Monge--Amp\`ere measure $\MA_\rv(\varphi)=\rv(m_\varphi)\frac{(\ddc \varphi)^n}{n!}$ for any $\varphi\in\PSH_K(X,L)$.
The weighted total volume is a constant
\begin{align}
    \VV_\rv := \int_X \rv(m_{\varphi_0})\; \frac{(\ddc\varphi_0)^n}{n!}.
\end{align}

One defines the following full mass space and finite energy space
\begin{align*}
    \mcE_K(X,L)
    \coloneqq &\left\{\varphi\in\PSH_K(X,L):\int_X\MA_\rv(\varphi)=\VV_\rv \right\}, \\
    \mcE_K^1(X,L)\coloneqq &\left\{\varphi\in\mcE_K(X,L):\int_X(\varphi-\varphi_0)\MA_\rv(\varphi)>-\infty \right\}.
\end{align*}
For $\varphi\in \mcE_K^1(X,L)$, $\varphi=\varphi_0+u$, set $\varphi_t=\varphi_0+tu$ and define
\begin{align*}
    \bfE_\rv(\varphi) 
    =& \frac{1}{\mathbb{V}_\rv} \int_0^1 \int_X (\varphi-\varphi_0) \MA_\rv(\varphi_t) d t, \\
    \bfJ_\rv(\varphi) 
        =& \frac{1}{\mathbb{V}_\rv} \int_X (\varphi-\varphi_0) \MA_\rv(\varphi_0) - \bfE_\rv(\varphi),  \\
         \bfJ_{\rv,\TT}(\varphi) 
         = &\inf_{\sigma\in \TT} \bfJ_{\rv}(\sigma^* \varphi), \\
         \bfJ_{\rv,\varphi'}(\varphi)
         =&\bfE_\rv(\varphi')-\bfE_\rv(\varphi)+\frac{1}{\VV_\rv}\int_X(\varphi-\varphi')\MA_\rv(\varphi'), \\
         \bfI_\rv(\varphi,\varphi')
         =&\frac{1}{\VV_\rv}\int_X(\varphi-\varphi')(\MA_\rv(\varphi')-\MA_\rv(\varphi)), \\
         \bfE_\rv^\eta(\varphi) 
        =& \frac{1}{\mathbb{V}_\rv} \int_0^1 \int_X (\varphi-\varphi_0) \Big( \rv(m_{\varphi_t}) \eta\wedge \frac{(\ddc \varphi_t)^{n-1}}{(n-1)!} + \langle  d\rv(m_{\varphi_t}),m_\eta \rangle \frac{(\ddc \varphi_t)^n}{n!} \Big)d t,
\end{align*}
where $\eta$ is a smooth $(1,1)$-form, which can be locally expressed as $\eta=\ddc  g$. And $m_\eta(\xi) \coloneqq  -J\xi(g)$, which is globally well-defined. And $\langle d\rv(m_\varphi),\cdot\rangle=d\rv(m_\varphi)^{\sharp} = \sum_\alpha \rv_{,\alpha} \xi^\alpha$ as a $\mathfrak{t}$-valued function on $X$, where $(\cdot)^{\sharp}$ is with respect to the K\"ahler form. Therefore, $\langle d\rv(m_\varphi),m_\eta\rangle=\sum_\alpha\rv_{,\alpha}\cdot \xi^\alpha(m_\eta)=-\sum_\alpha \rv_{,\alpha} J\xi^\alpha(g)$.

\begin{lemma}
\label{arch_J_v-relation}
    For all $\varphi_1,\varphi_2\in\mcE_K^1(X,L)$, one has
    \begin{align}
    \label{uniform_Lip_J_v}
        |\bfJ_\rv(\varphi_1)-\bfJ_\rv(\varphi_2)|
        \leq C_n \bfI_\rv(\varphi_1,\varphi_2)^{\frac{1}{2}}\max\{\bfJ_\rv(\varphi_1),\bfJ_\rv(\varphi_2)\}^{\frac{1}{2}}
    \end{align}
\end{lemma}
\begin{proof}
For the reader's convenience, we sketch the proof.
    For any $\varphi_1,\varphi_2\in\mcH_K(X,L)$, by definition,
    \begin{align*}
        |\bfJ_\rv(\varphi_1)-\bfJ_\rv(\varphi_2)|
        \leq&|\bfE_\rv(\varphi_1)-\bfE_\rv(\varphi_2)| +\frac{1}{\VV_\rv}\int_X|\varphi_1-\varphi_2|\,\MA_\rv(\varphi_0)  \\
        \leq&C\bfI_\rv(\varphi_1,\varphi_2)
        +C\int_X|\varphi_1-\varphi_2|\,\omega^n ,
    \end{align*}
    where we have used \cite[Lemma~6.9]{AJL23} for the second inequality.
    By \cite[Corollary~5.8]{Dar15} or \cite[Lemma~3.13]{BBGZ13}, one has
    \begin{align*}
        \int_X|\varphi_1-\varphi_2|\,\omega^n
        \leq& C\bfI(\varphi_1,\varphi_2)^{\frac{1}{2}}\max\{\bfJ(\varphi_1),\bfJ(\varphi_2)\}^{\frac{1}{2}} \\
        \leq& C_\rv \bfI_\rv(\varphi_1,\varphi_2)^{\frac{1}{2}}\max\{\bfJ_\rv(\varphi_1),\bfJ_\rv(\varphi_2)\}^{\frac{1}{2}}.
    \end{align*}
    Combining two inequalities above, we obtain \eqref{uniform_Lip_J_v}. By approximation, the inequality \eqref{uniform_Lip_J_v} also holds on $\mcE^1_K(X,L)$.

\end{proof}
As in \cite{Lah19}, the $\rv$-scalar curvature is defined as
\begin{align*}
    S_\rv(\varphi)
:=\rv(m_\varphi)S(\varphi)-2\rv_{,\alpha}(m_\varphi)\Delta_\varphi(m_\varphi^{\xi^\alpha})-\rv_{,\alpha\beta}(m_\varphi)\langle\xi^\alpha,\xi^\beta\rangle_\varphi.
\end{align*}
Let $\rw\in C^\infty(P,\RR)$ be another weight function on $P$. We say that $\omega_\varphi$ is a \textit{weighted-cscK metric} if $S_\rv(\omega_\varphi)=c_{\rv,\rw}(L)\rw(m_\varphi)$, where
\begin{align*}
        c_{(\rv,\rw)}(L)\coloneqq\left\{
        \begin{aligned}
            &\frac{\int_XS_\rv(\omega)\frac{\omega^n}{n!}}{\int_X\rw(m_\omega)\frac{\omega^n}{n!} }, &\text{ if } \int_X\rw(m_\omega)\frac{\omega^n}{n!}\neq0 \\
            &1, &\text{ if } \int_X\rw(m_\omega)\frac{\omega^n}{n!}=0,
        \end{aligned}\right.
    \end{align*}
    which is independent of the choice $\omega\in c_1(L)$ (see \cite[Definition~3, Lemma~2]{Lah19}).

For any $\varphi\in\mcE^1_K(X,L)$, the weighted Mabuchi functional is defined as
\begin{align*}
        \bfM_{\rv,\rw}(\varphi) 
        = \bfH_\rv(\varphi) -\bfE_\rv^{{\rm Ric}(\omega)}(\varphi) 
         +c_{(\rv,\rw)}(L) \bfE_{\rw}(\varphi) - C_0,
    \end{align*}
    where 
    \begin{align*}
        \bfH_\rv(\varphi)
        \coloneqq \int_X \log\left(\frac{\MA_\rv(\varphi)}{\omega^n}\right)  \MA_\rv(\varphi), \;  
        C_0 \coloneqq \int_X \log(\rv(m_{\omega}))\MA_\rv(\varphi_0).
    \end{align*}

    Suppose $\rw\in C^\infty(P,\RR_{>0})$. Let $\ell_\ext$ be the weighted extremal function on $(X,L)$. We say that $\omega_\varphi$ is a \textit{weighted-extremal metric} if $S_\rv(\omega_\varphi)=\rw(m_\varphi)\cdot \ell_{\rm ext}(m_\varphi)$.
    The weighted relative Mabuchi functional is defined as
    \begin{align*}
        \bfM_{\rv,\rw\cdot\ell_\ext}
        = \bfH_\rv -\bfE_\rv^{{\rm Ric}(\omega)} 
         + \bfE_{\rw\cdot\ell_\ext} - C_0.
    \end{align*}
\begin{definition}
    The weighted Mabuchi functional $\bfM_{\rv,\rw}$ is called $\bG$-\textit{coercive} over $\mcE^1_{K}$ if there exist $\gamma>0$ and $C>0$ such that for any $\varphi\in \mcE^1_{K}$
    \begin{align*}
        \bfM_{\rv,\rw}(\varphi)
        \geq \gamma\cdot\bfJ_{\rv,\TT}(\varphi)-C.
    \end{align*}
\end{definition}

Given $\varphi_1, \varphi_2\in\mcE^1(X,L)$, a map $\Phi=(\Phi(s)):(\cdot,s)\in(0,1)\rightarrow\mcE_K^1(X,L)$ is a \textit{psh subgeodesic segment} if $\limsup_{s\rightarrow0^+}\Phi(s)\leq \varphi_1$ and $\limsup_{s\rightarrow1^-}\Phi(s)\leq \varphi_2$, and if $\Phi(x,-\log|\tau|)$ is a psh metric on $p_1^*L$ over $X\times {\mathbb D}_{(0,1)}$, where ${\mathbb D}_{(0,1)}\coloneqq  \{\tau\in\CC^*:-\log|\tau|\in(0,1) \}$.
The largest psh subgeodesic segment joining $\varphi_1, \varphi_2$ is called \textit{psh geodesic}, which exists by \cite{Dar17a}, \cite{DG18}.

By the same proof as in \cite[Section 4.5]{HL23}, we have
\begin{proposition}
    The weighted Mabuchi functional is convex along the psh geodesic in $\mcE^1_K(X,L)$.
\end{proposition}

A continuous map $\Phi=(\Phi(s)):(\cdot,s)\in\RR_{>0}\rightarrow\mcE_K^1(X,L)$ is a \textit{psh ray} if $\Phi(x,-\log|\tau|)$ is a psh metric on $p_1^*L$ over $X\times {\mathbb D^*}$, where ${\mathbb D}^*\coloneqq  \{\tau\in\CC^*:-\log|\tau|\in(0,\infty) \}$.
A psh ray $\Phi=(\Phi(s))$ is said to \textit{linear growth} if
\begin{align*}
    \lambda_\max(\Phi)\coloneqq\lim_{s\rightarrow\infty}\frac{\sup_X(\Phi(s)-\varphi_0)}{s}
    <\infty.
\end{align*}

Let $\Phi=(\Phi(s))_{s\geq0}$ be a psh ray of linear growth and $\bfF$ be a functional on $\mcE_K^1(X,L)$. The \textit{slope at infinity} of $\bfF$ along $\Phi$ is defined as
\begin{align*}
    \bfF^{\prime\infty}(\Phi)
    \coloneqq \lim_{s\rightarrow\infty}\frac{\bfF(\Phi(s))}{s}
\end{align*}
if the limit exists.

A continuous map $\Phi=(\Phi(t)):(\cdot,t)\in\RR_{>0}\rightarrow\mcE_K^1(X,L)$ is a \textit{psh 
geodesic} if the restriction of $\Phi$ to each compact interval  $[a,b]\subset\RR_{>0}$ coincides (up to affine reparametrization) with the psh geodesic joining $\Phi(a)$ to $\Phi(b)$.

\section{Weighted non-Archimedean Monge-Amp\`ere measure, weighted non-Archimedean functionals and geodesic rays}
\label{Weighted non-Archimedean Monge-Ampere measure, weighted non-Archimedean functionals and geodesic rays}

In this section, we define the weighted non-Archimedean Monge--Amp\`ere measure and functional based on \cite{HL23}.
%As an immediate application of Theorem \ref{env_thm}, 
Under the assumption of the envelope property (Definition~\ref{def:envelope_property}), 
we solve the weighted non-Archimedean Monge--Amp\`ere equation by adapting the variational argument from \cite[Theorem 12.8]{BJ22}. We work on complex field $\CC$.
All the statements in this section will hold if we replace the group $T$ by the group $K$.

We will follow the notations defined in \cite{HL23}. Since the torus action $T$ acts on $X$, we have the associated fiber bundle
\begin{align*}
    (X^{[\vec{k}]}, L^{[\vec{k}]}) = (X,L)\times_{T} \mathbb{S}^{[\vec{k}]}.
\end{align*}
Then we have the diagram
\begin{equation}
\begin{tikzcd}[column sep = 2.1cm]
\label{projections}
\CC^* \arrow[r, ""] \arrow[d, ""] & L\times (\CC^{k_1}\setminus\{0\})\times \cdots\times (\CC^{k_r}\setminus\{0\}) \arrow[r, ""] \arrow[d, ""] & L^{\bvk} \arrow[d, ""]\\
\CC^* \arrow[r,""]  & X\times (\CC^{k_1}\setminus\{0\})\times \cdots\times (\CC^{k_r}\setminus\{0\})
\arrow[r, ""] &  X^{\bvk}.
\end{tikzcd}
\end{equation}

The smooth functions $\rv,\rw$ on polytope $P$ can be expressed as limits of polynomials
\begin{align}
    \rv = \lim\sum_{\vec{k}} a_{\vec{k}} \underline{y}^{\vec{k}}, \; \rw = \lim \sum_{\vec{k}} b_{\vec{k}} \underline{y}^{\vec{k}}.
\end{align}

\subsection{Weighted non-Archimedean Monge-Amp\`ere measure and functionals for the polynomial weight}
\label{w-NA-MA}

\begin{definition}
    Let $\phi\in \mcH_T(X^\an,L^\an)$. Let $\rv = \sum_{\vec{k}}a_{\vec{k}} \underline{y}^{\vec{k}}$ be a polynomial.
     We define the weighted non-Archimedean Monge-Amp\`ere measure
     $\MA^\NA_\rv$ as below. For any $T$-invariant model function $\phi_D$, we define
     \begin{align}
     \label{wMA}
    \int_{X^\an} \phi_D\; \MA^\NA_\rv(\phi) = \sum_{\vec{k}}a_{\vec{k}}\cdot \frac{k_1!\ldots k_r!}{(n+k)!} \cdot \phi^{[\vec{k}]}_D\cdot {(\ddc\phi^{[\vec{k}]})^{n+k}}.
\end{align}

Since any continuous function on $X^\an$ can be approximated by model functions, formula \eqref{wMA} defines $\MA^\NA_\rv(\phi)$ as a Radon measure on $X^\an$.

\end{definition}

For a $T_\CC$-equivariant test configuration $(\mcX,\mcL)$,
define 
\begin{align}
    (\bfE^{[\vec{k}]})^\NA (\phi_\mcL) = \frac{k_1!\ldots k_r!}{(n+k+1)!}(\mcL^{[\vec{k}]})^{n+k+1}.
\end{align}

Let $(\mcX, \mathcal{Q})$ be a $T_\CC$-equivariant model.
Define
\begin{align}
    \label{intersection_formula_bfE^Q}
    ((\bfE^Q)^{[\vec{k}]})^\NA (\phi_\mcL) = \frac{k_1!\ldots k_r!}{(n+k)!} (\mathcal{Q})^{[\vec{k}]} \cdot (\mcL^{[\vec{k}]})^{n+k} .
\end{align}

\begin{comment}
\begin{align}
    (\bfR^{[\vec{k}]})^\NA (\phi_\mcL) = \frac{n!}{(n+k)!} (K_{X_{\PP^1}/\PP^1})^{[\vec{k}]} \cdot (\mcL^{[\vec{k}]})^{n+k} .
\end{align}
\end{comment}

If $\rv = \sum_{\vec{k}}a_{\vec{k}} \underline{y}^{\vec{k}}$, then one defines
\begin{align}
    \bfE^\NA_\rv (\phi_\mcL) \coloneqq \sum_{\vec{k}} a_{\vec{k}} \cdot (\bfE^{[\vec{k}]})^\NA (\phi_\mcL),
\end{align}

\begin{align}
\label{intersection_formula_bfE^Q_polynomial}
    (\bfE^\mathcal{Q})^\NA_\rv (\phi_\mcL) \coloneqq \sum_{\vec{k}} a_{\vec{k}} \cdot ((\bfE^\mathcal{Q})^{[\vec{k}]})^\NA (\phi_\mcL).
\end{align}

We define
\begin{align}
    \bfR^\NA_\rv (\phi_\mcL) = (\bfE^{K^{\log}_{X_\PP^1/\PP^1}})_\rv^\NA(\phi_\mcL).
\end{align}

\begin{comment}
\begin{align}
    \bfR^\NA_\rv (\phi_\mcL) = \sum_{\vec{k}} a_{\vec{k}} \cdot (\bfR^{[\vec{k}]})^\NA (\phi_\mcL).
\end{align}
\end{comment}

\subsubsection{Weighted Okounkov body and weighted volume}

Let $x\in X$ be a regular point where the Lie algebra of $T_\CC$ has the maximal rank at $x$. Let $z=(z_1,\ldots,z_n)$ be a system of parameters centered at $x$.
This defines a rank-$n$ valuation:
\begin{align*}
    {\rm ord}_z = \{{\rm ord}_{z_1},\ldots,{\rm ord}_{z_n}\}: \mathcal{O}_{X,x}\setminus\{0\}\rightarrow \RR^n.
\end{align*}
For $i\in\{1,\ldots,r\}$, the valuation induced by the vector filed $\xi_i$ at $x$ can be expressed as ${\rm ord}_{\xi_i} = \sum_{1\leq j\leq n}a_{i,j}\cdot {\rm ord}_{z_j}$. Denote ${\rm ord}_\xi = ({\rm ord}_{\xi_1},\ldots, {\rm ord}_{\xi_r})$.

We have
\begin{align*}
    \Gamma_m &= \{({\rm ord}_{z}(s), m): \; s\in H^0(X,mL)\}, \\
    \underline{\Gamma}_m &= \{({\rm ord}_{\xi}(s), m): \; s\in H^0(X,mL)\}.
\end{align*}
Let $\Gamma = \{(y,m)\in \RR^{n+1}:\; (y,m)\in \Gamma_m\}$, $\underline{\Gamma} = \{(\underline{y},m)\in \RR^{r+1}:\; (\underline{y},m)\in \underline{\Gamma}_m\}$. Let $\Sigma, \underline{\Sigma}$ be closed convex cones generated by $\Gamma,\underline{\Gamma}$ respectively. 
Define
\begin{align}
    \mathbb{O} &= \{y: \; (y,1)\in \Sigma\}, \\
    P &= \{\underline{y}: \; (\underline{y},1)\in \underline{\Sigma}\}.
\end{align}
We have the map ${\rm Pr}: \mathbb{O}\rightarrow P$, by sending $(y_1,\ldots,y_n)\mapsto (\sum_{1\leq j\leq n}a_{1,j}y_j,\ldots, \sum_{1\leq j\leq n}a_{r,j}y_j)$.

For a filtration $\mcF$ on $R$, the concave transform function is defined as
$G_\mcF: \mathbb{O}\rightarrow \mathbb{R}$, $G_\mcF(y) = \sup\{t:\; y\in \mathbb{O}(R^t_\cdot)\}$.

Let $\varphi_1,\varphi_2 \in \mcH(X^\an,L^\an)$, which correspond to filtrations $\mcF_1,\mcF_2$ respectively. Then
\begin{align}
\label{d_1_v_d_1}
    \rmd_{\rv,1}(\varphi_1,\varphi_2) = \frac{1}{\vol_\rv(\mathbb{O})n!}\int_P \Big( \int_{\mathbb{O}_{{\underline{y}}}} |G_{\mcF_1}-G_{\mcF_2}| \; dy \; \Big)\rv(\underline{y}) \; d\underline{y},
\end{align}
where $\vol_\rv(\mathbb{O}) = \int_{\mathbb{O}} \rv(y)\; dy$.

Let $\chi\in\mcN_\RR(X,L)$ and $\{e_j\}_{j=1}^{N_{m,\eta}}$ be an orthogonal basis of $R_{m,\eta}$ with respect to $\chi$.
The \textit{successive minima} of $\chi|_{R_{m,\eta}}$ is the sequence $\lambda^{(m,\eta)}_j\coloneqq \chi(e_j)$. Following \cite[Section~5.6]{HL23}, the \textit{weighted volume} of $\chi$ is defined as
\begin{align*}
    \vol_\rv(\chi)
    \coloneqq \lim_{m\rightarrow\infty}\frac{1}{m^{n+1}}\sum_{\eta}\rv(\frac{\eta}{m})\sum_j\lambda^{(m,\eta)}_j.
\end{align*}
We define a pseudo-metric $\rmd_{\rv,1}$ on $\mcN_\RR$ as follows. For any two norms $\chi,\chi'\in\mcN_\RR(X,L)$,  one defines
\begin{align*}
    \rmd_{\rv,1}(\chi,\chi')
    \coloneqq \vol_\rv(\chi)+\vol_\rv(\chi')-2\vol_\rv(\chi\wedge\chi'),
\end{align*}
where $\chi\wedge\chi'(s)\coloneqq\min\{\chi(s),\chi'(s)\}$.
One easily checks $\rmd_{\rv,1}(\chi+c,\chi'+c)=\rmd_{\rv,1}(\chi,\chi')$ for any $c\in\RR$. Similar as in \cite[Section~3.5]{BJ24}, the metric $\rmd_{\rv,1}$ induces a quotient metric $\underline{\rmd}_{\rv,1}$ on $\mcN_\RR/\RR$ by
\begin{align*}
    \underline{\rmd}_{\rv,1}(\chi,\chi')
    \coloneqq\inf_{c\in\RR}\rmd_{\rv,1}(\chi,\chi'+c).
\end{align*}

Set $\cF_\chi$ as the associated filtration of $\chi$, then
\begin{align}
\label{eq:vol_v(chi)=integral}
    \vol_\rv(\chi)
=\frac{1}{\vol_\rv(\mathbb{O})n!}\int_P\left(\int_{\mathbb{O}_{\underline{y}}}G_{\mcF_\chi}\, dy\right)\rv(\underline{y})\,d \underline{y}.
\end{align}
By integration by parts, 
\begin{align*}
    \vol_\rv(\chi)
    =\int_\RR\lambda\, {\bf DH}_\rv(\chi)
    =a+\int_a^{+\infty}\vol_\rv(\chi\geq\lambda)\,d\lambda,
\end{align*}
where $\vol_\rv(\chi\geq\lambda)\coloneqq \vol_\rv(\mcF_\chi^{(\lambda)})$ and $a\leq \lambda_{\min}(\chi)$.
\begin{lemma}
    \label{vol_converge_chi_d}
    For any $\chi\in\mcN_\RR(X,L)$, then $\vol_\rv(\chi_d)\rightarrow\vol_\rv(\chi)$.
\end{lemma}
\begin{proof}
    By definition of canonical approximation and $G_{\cF_\chi}$, it is easy to check that $G_{\cF_{\chi_d}}\leq G_{\cF_{\chi}}$ on $\mathbb{O}$. Then, by \cite[Theorem 3.18]{BJ24},
    \begin{align*}
        |\vol_\rv(\chi)-\vol_\rv(\chi_d)|
        =&\frac{1}{\vol_\rv(\mathbb{O})n!}\int_P\left(\int_{\mathbb{O}_{\underline{y}}}(G_{\mcF_\chi}-G_{\mcF_{\chi_d}})\,dy\right)\rv(\underline{y})\,d \underline{y} \\
        \leq& C_\rv\frac{1}{\vol_\rv(\mathbb{O})n!}\int_{\mathbb O}(G_{\mcF_\chi}-G_{\mcF_{\chi_d}})\, dy
        =C_\rv(\vol(\chi)-\vol(\chi_d))\rightarrow0.
    \end{align*}
\end{proof}
\begin{lemma}
    \label{vol_E_v}
    Let $\rv= \sum_{\vec{k}}a_{\vec{k}} \underline{y}^{\vec{k}}$ be a polynomial.
    Let $\chi\in \mathcal{N}_\RR$ be a norm of finite type. Then
    \begin{align}
        \vol_\rv(\chi) &= \bfE^\NA_\rv({\rm FS}(\chi)), 
        \label{vol_v()=E_v^NA(FS)} %\\
        %\vol_\rv({\rm IN}(\phi)) &= \bfE^\NA_\rv(Q(\phi)),
    \end{align}
    %where $\bfE^\NA_\rv(\cdot)$ denotes the extended weighted Monge--Amp\`ere energy of an arbitrary function $\phi:X^\an\rightarrow\RR$, defined by
    %\begin{align} 
    %\label{eq:extended_E_v^NA}
    %    \bfE_\rv^\NA(\phi)
    %    \coloneqq \sup\{\bfE_\rv^\NA(\phi):\phi\in\mcH_T(X^\an,L^\an), \phi\leq \phi\}.
    %\end{align}
\end{lemma}
\begin{proof}
     Let $(\mcX,\mcL)$ be the $T_\CC$-equivariant corresponding test configuration of $\chi$ under the Rees construction. Without the loss of generality, we may assume $\lambda_{\min}(\mcL)=0$.

    Recall that, $\pi_{\vec{k}}: X^{[\vec{k}]}\rightarrow \PP^{[\vec{k}]}$, and similarly, $\pi_{[\vec{k}]}:(\mcX^{[\vec{k}]},m\mcL^{[\vec{k}]})\rightarrow \PP^{[\vec{k}]}$.
    As in the proof of \cite[Theorem 3.1]{BHJ17}, 
    \begin{align}
        (\pi_{[\vec{k}]})_*\oo_{\mcX^{[\vec{k}]}}(m\mcL^{[\vec{k}]}) = \bigoplus_{\eta_1,\ldots,\eta_r\in\ZZ} H^0(\mcX,m\mcL)_{\eta_1,\ldots,\eta_r}\otimes\oo_{\PP^{k_1}}(\eta_1)\otimes\ldots\otimes\oo_{\PP^{k_r}}(\eta_r).
    \end{align}
    Consider the basis $\{\tilde{s}_{i,j_1,\ldots,j_r}^{[\vec{k}]}\}$ of $H^0(\mcX^{[\vec{k}]}, \mcL^{[\vec{k}]})$:
    \begin{align*}
        H^0(\mcX^{[\vec{k}]}, \mcL^{[\vec{k}]}) \ni \tilde{s}_{i,j_1,\ldots,j_r}^{[\vec{k}]} = s_{i,\eta_1,\ldots,\eta_r}\cdot\sigma_{j_1}\ldots\sigma_{j_r},
    \end{align*}
    where $s_{i,\eta_1,\ldots,\eta_r}\in H^0(\mcX,m\mcL)_{\eta_1,\ldots,\eta_r}$ and $\sigma_{j_i}\in H^0(\PP^{k_i},\oo(\eta_i))$.
    Set 
    \begin{align*}
        c^{[\vec{k}]}_{i,\eta_1,\ldots,\eta_r} 
        = \prod_{1\leq i\leq r}h^0(\PP^{k_i},\oo(\eta_i)) 
        = \binom{k_1+\eta_1}{\eta_1}\ldots\binom{k_r+\eta_r}{\eta_r}.
    \end{align*}   
    Then if $(\frac{\eta_1}{m},\ldots,\frac{\eta_r}{m})\xrightarrow{m\to\infty}(y_1,\ldots,y_r)\in \RR^r$, we have
    $\frac{c^{[\vec{k}]}_{i,\eta_1,\ldots,\eta_r}}{m^k}$ converges to $\frac{y_1^{k_1}\ldots y_r^{k_r}}{k_1!\ldots k_r!}$.
    Let
    \begin{align}
        S_\rv(m\mcL) &:= \sum_{\vec{k}} \frac{a_{\vec{k}}\cdot k_1!\cdots k_r!}{m^{n+k+1}}\cdot  \Big(\sum_{\eta_1,\ldots,\eta_r\in\ZZ} \sum_{i}  c^{[\vec{k}]}_{i,\eta_1,\ldots,\eta_r} \Big)
    \end{align}
    Since
    \begin{align*}
        \frac{(\mcL^{[\vec{k}]})^{n+k+1}}{(n+k+1)!} = \frac{1}{m^{n+k+1}}  \Big(\sum_{\eta_1,\ldots,\eta_r\in\ZZ} \sum_{i}  c^{[\vec{k}]}_{i,\eta_1,\ldots,\eta_r} \Big) + O(\frac{1}{m}),
    \end{align*}
    we have
    \begin{align}
            \bfE_\rv^\NA(\phi_\mcL) &=  S_\rv(m\mcL) + O(\frac{1}{m}) \nonumber \\
            &= \frac{1}{m^{n+1}}\sum_{\vec{k}} a_{\vec{k}}\cdot k_1!\cdots k_r!\cdot \sum_{\eta_1,\ldots,\eta_r\in\ZZ} \sum_{i}\Big(\frac{1}{m^k} \cdot c^{[\vec{k}]}_{i,\eta_1,\ldots,\eta_r} \Big) + O(\frac{1}{m}) \nonumber \\
            &= \frac{1}{m^{n+1}}\sum_{P\cap (\frac{1}{m}\ZZ)^r}\sum_i\sum_{\vec{k}} \Big(a_{\vec{k}} y_1^{k_1}\ldots y_r^{k_r}\Big) + O(\frac{1}{m}) \nonumber \\
            &= \frac{1}{m^{n+1}}\sum_{s_i\in H^0(\mcX,m\mcL)} \rv(y_1(s_i),\ldots,y_r(s_i)) + O(\frac{1}{m}) \nonumber \\
            &= \frac{1}{m^{n+1}} \sum_{P\cap (\frac{1}{m}\ZZ)^r} \sum_{s_i\in H^0(\mcX,m\mcL)_{\eta_1,\ldots,\eta_r}} \rv(y_1,\ldots,y_r) +O(\frac{1}{m}) \nonumber \\
            &= \frac{1}{\vol(\mathbb{O})n!}\int_P\Big( \int_{\mathbb{O}_{\underline{y}}} G_\mcF \; dy \Big) \; \rv(\underline{y})\; d \underline{y} + O(\frac{1}{m}). \nonumber\\
            &= \vol_\rv(\chi)+O(\frac{1}{m}).  \label{S_v=vol_v}
    \end{align}
    As $m\rightarrow\infty$, we conclude \eqref{vol_v()=E_v^NA(FS)}.
\end{proof}

\subsection{Weighted non-Archimedean functional for smooth weight}
For the smooth function $\rv$, we use Stone–Weierstrass approximation $\rv = \lim \rv_i$, where each $\rv_i$ is a polynomial. Then
\begin{lemma} 
\label{lem:E_v_i_convergence}
We can define $\MA_\rv^\NA(\phi_\mcL)$, $\bfE^\NA_\rv (\phi_\mcL)$ and $(\bfE^\mathcal{Q})^\NA_\rv (\phi_\mcL)$ by the following convergent sequences
    \begin{align}
        \label{MA_v_i_convergence}
          \MA_\rv^\NA(\phi_\mcL)&=\lim_i\MA^\NA_{\rv_i}(\phi_\mcL),
    \end{align}
    \begin{align}
        \label{E_v_i_convergence}
             \bfE^\NA_\rv (\phi_\mcL) &= \lim_i \bfE^\NA_{\rv_i} (\phi_\mcL), 
    \end{align}
    \begin{align}
         \label{EQ_v_i_convergence}
            (\bfE^\mathcal{Q})^\NA_\rv (\phi_\mcL) &= \lim_i (\bfE^\mathcal{Q})^\NA_{\rv_i} (\phi_\mcL).
    \end{align}

\end{lemma}
\begin{proof}
   Let $\{ {\rm v}_i \}_{i\in\NN}$ be a sequence of positive polynomials that converge to $\rv$ smoothly on $\bar{P}$.
    To prove \eqref{MA_v_i_convergence}, it suffices to show that for any model function $\phi_E\geq 0$,
    \begin{align*}
         \int_{X^\an} \phi_E\; \MA_\rv^\NA(\phi_\mcL)&= \lim_{i\to\infty} \int_{X^\an} \phi_E \; \MA^\NA_{\rv_i}(\phi_\mcL).
    \end{align*}
    For any $\epsilon>0$, there exists $N>0$ such that for any $i,j>N$, $|\rv_i-\rv_j|<\epsilon$.
    By line $3$ and line $7$ in formula \eqref{vol_v_upbd}, we have
    \begin{align*}
         \int_{X^\an} \phi_E \; \MA^\NA_{\rv_i}(\phi_\mcL) &= \lim_{m\to\infty}\left(  \frac{1}{m^{n}}\sum_{s_k\in H^0(E,m\mcL|_E)} \rv_i(y_1(s_k),\ldots,y_r(s_k)) \right).
    \end{align*}
    This implies that 
    \begin{align*}
        \int_{X^\an} \phi_E \; \MA^\NA_{\rv_j-\epsilon}(\phi_\mcL) \leq \int_{X^\an} \phi_E \; \MA^\NA_{\rv_i}(\phi_\mcL) \leq \int_{X^\an} \phi_E \; \MA^\NA_{\rv_j+\epsilon}(\phi_\mcL).
    \end{align*}
    Then by formula \eqref{wMA},
    \begin{align*}
         &\int_{X^\an} \phi_E \; \MA^\NA_{\rv_j}(\phi_\mcL)-\epsilon\int_{X^\an}\phi_E \; \MA^\NA(\phi_\mcL)  \\
        &\leq \int_{X^\an} \phi_E \; \MA^\NA_{\rv_i}(\phi_\mcL) \leq \int_{X^\an} \phi_E \; \MA^\NA_{\rv_j}(\phi_\mcL)+\epsilon\int_{X^\an}\phi_E\; \MA^\NA(\phi_\mcL).
    \end{align*}
    Therefore, $\int_{X^\an} \phi_E \; \MA^\NA_{\rv_i}(\phi_\mcL)$ is a Cauchy sequence, its limit defines $\int_{X^\an} \phi_E \; \MA^\NA_{\rv}(\phi_\mcL)$.
    
    Formula \eqref{E_v_i_convergence} follows \cite[Lemma 5.5]{HL23}. Alternatively, it can be proved by using Okounkov body.
    By line $4$ in the following \eqref{S_v=vol_v}, we have
    \begin{align*}
        \bfE_{\rv_i}^\NA(\phi_\mcL)
        = \lim_{m\to\infty}\left(  \frac{1}{m^{n}}\sum_{s_k\in H^0(E,m\mcL)} \rv_i(y_1(s_k),\ldots,y_r(s_k)) \right).
    \end{align*}
    Applying the above argument, we obtain that $\bfE_{\rv_i}^\NA(\phi_\mcL)$ is a Cauchy sequence.

    For the proof of formula \eqref{EQ_v_i_convergence},  when $\rv_i$ is a polynomial, by using \eqref{intersection_formula_bfE^Q} and \eqref{intersection_formula_bfE^Q_polynomial}, we may express $(\bfE^\mathcal{Q})^\NA_\rv (\phi_\mcL)$ as a finite sum of intersections. Therefore by applying the argument of \cite[Theorem 3.6]{BHJ19} to each intersection, we have the slope formula for $(\bfE^\mathcal{Q})^\NA_\rv (\phi_\mcL)$
    \begin{align*}
         (\bfE^\mathcal{Q})^\NA_{\rv_i} (\phi_\mcL) 
         &= \lim_{s\to\infty}\frac{\bfE^Q_{\rv_i} (\Phi_\mcL(s))}{s}\\
        &=\lim_{s\to\infty}\frac{1}{s} \frac{1}{\mathbb{V}_{\rv_i}} \int_0^1 \int_X (\varphi(s)-\varphi_0) \Big( \rv_i(m_{\varphi_t(s)}) \alpha_s\wedge \frac{(\ddc \varphi_t(s))^{n-1}}{(n-1)!} \\
        &\qquad\qquad\qquad\qquad\qquad\qquad\qquad\quad+ \langle  d\rv_i(m_{\varphi_t(s)}),m_{\alpha_s} \rangle \frac{(\ddc \varphi_t(s))^n}{n!} \Big)d t,
    \end{align*}
    where $\alpha_s$ is a smooth metric in the class of $c_1(Q_s)$.
    
    Let $Y$ be a resolution of $X$, $\mcY$ be a smooth simple normal crossing model, which is a resolution of $\mcX$.
    Over the model $\mcY$, $\rv_i(m_{\varphi_t(s)}) \alpha_s\wedge \frac{(\ddc \varphi_t(s))^{n-1}}{(n-1)!}$ and $\langle  d\rv_i(m_{\varphi_t(s)}),m_{\alpha_s} \rangle\frac{(\ddc \varphi_t(s))^n}{n!}$ converge weakly to signed measures on $\mcY_0$. Meanwhile, $\frac{1}{s}(\varphi(s)-\varphi_0)$  extends to a continuous function on $\mcY^{\bf hyb}$. Then by \cite[Theorem B]{PS22} (see also \cite[Proposition 5.15]{BJ25}), we have
    \begin{align}
        \label{bfE^Q_integral_over_central_fiber}
         & (\bfE^\mathcal{Q})^\NA_{\rv_i} (\phi_\mcL) = \lim_{s\to\infty}\frac{\bfE^Q_{\rv_i} (\Phi_\mcL(s))}{s}\\
        &= \frac{1}{\mathbb{V}_{\rv_i}} \int_0^1 \int_{\mcY_0} (\varphi-\varphi_0) \Big( \rv_i(m_{\varphi_t}) \alpha_s\wedge \frac{(\ddc \varphi_t)^{n-1}}{(n-1)!} + \langle  d\rv_i(m_{\varphi_t}),m_{\alpha_s} \rangle \frac{(\ddc \varphi_t)^n}{n!} \Big)d t \nonumber.
    \end{align}
    As $\rv_i$ converges to $\rv$, the right-hand-side of formula \eqref{bfE^Q_integral_over_central_fiber} converges. This implies the formula \eqref{EQ_v_i_convergence}.
\end{proof}
\begin{remark}
    We are grateful to S\'ebastien Boucksom and Mattias  Jonsson for pointing out a gap in the proof of \eqref{EQ_v_i_convergence} in our previous version.
\end{remark}

\begin{comment}
\begin{lemma} ({\color{red} to be deleted})
\label{J_v_J}
    There exists a constant $C>0$, such that for any $\phi\in\mcH_T(X^\an,L^\an)$,  we have
    \begin{align}
        \frac{1}{C} \cdot \bfJ^\NA(\phi) \leq \bfJ_\rv^\NA(\phi) \leq C\cdot \bfJ^\NA(\phi).
    \end{align}
\end{lemma}
\begin{proof}
    The psh function $\phi\in\mcH_T(X^\an,L^\an)$, corresponds to $T_\CC$-equivariant a test configuration $(\mcX,\mcL)$.
    Let $\Phi$ be a geodesic ray with $\Phi^{\NA} = \phi$. By Lemma \ref{lemma_slope_formula} $\bfJ^{\prime\infty}_\rv(\Phi) =\bfJ^\NA_\rv(\phi)$. By \cite[(2.35)]{HL20}, there exists a constant $C>0$, such that $\frac{1}{C}\bfJ(\Phi(s)) \leq \bfJ_\rv(\Phi(s)) \leq C\bfJ(\Phi(s))$. Then
    \begin{align}
        \frac{1}{C}\bfJ^\NA(\phi)\leq \bfJ^\NA_\rv(\phi)\leq C\bfJ^\NA(\phi).
    \end{align}
\end{proof}
\end{comment}

For any $\phi\in \PSH_T(X^\an,L^\an)$, there exists a decreasing sequence $\phi_j\in \mcH_T(X^\an,L^\an)$ that converges to $\phi$.
The weighted non-Archimedean Monge-Amp\`ere  measure $\MA^\NA_\rv(\phi)$ is defined as a weak limit
$\MA^\NA_\rv(\phi_j)$.
In addition, one defines
\begin{align*}
    \bfE_\rv^\NA(\phi)
    \coloneqq\inf\{\bfE_\rv^\NA(\psi):\psi\in\mcH_T(X^\an,L^\an) \text{ and }\psi\geq\phi\}.
\end{align*}
This endows $\bfE_\rv^\NA$ a unique u.s.c and monotone increasing extension over $\PSH_T(X^\an,L^\an)$.
\begin{definition}
    We define the weighted finite energy space as
    \begin{align}
        \mcE^{1}_\rv(X^\an,L^\an) = \{ \phi\in \PSH_T(X^\an,L^\an): \; \bfE^\NA_\rv(\phi)>-\infty \}.
    \end{align}
    A strong topology of $\mcE^{1}_\rv(X^\an,L^\an)$ is induced by $\bfE^\NA_\rv$. 
\end{definition}
%Since $\bfJ^\NA_\rv$ and $\bfJ^\NA$ are comparable, 
%we have the following corollary.

For $\phi\in \mcE_\rv^1(X^\an,L^\an)$, let $\{\phi_i\}$ be a decreasing sequence of Fubini-Study functions that converges to $\phi$. Since $\bfE^\NA_\rv$ is upper semi-continuous, it converges along a decreasing sequence. Using formula \eqref{bfE^Q_integral_over_central_fiber}, by fixing a positive constant $C$, for each polynomial $\rv_i$, $(\bfE^\mathcal{Q})^\NA_{\rv_i}+ C\cdot \bfE^\NA_{\rv_i}$ is also upper semi-continuous. Therefore we have the following lemma.
\begin{lemma}
    \begin{align}
    \bfE^\NA_\rv (\phi) &= \lim_i \bfE^\NA_\rv (\phi_i), \\
    (\bfE^\mathcal{Q})^\NA_\rv (\phi) &= \lim_i (\bfE^\mathcal{Q})^\NA_\rv (\phi_i).
\end{align}
\end{lemma}

\begin{corollary}
\label{cor:vol_v_E_v_generalcase}
    Let $\chi\in \mathcal{N}_\RR$, $\phi\in L^\infty(X^\an)$. Then
    \begin{align}
        \vol_\rv(\chi) &= \bfE^\NA_\rv({\rm FS}(\chi)), 
        \label{vol_v()=E_v^NA(FS)_any_chi} \\
        \vol_\rv({\rm IN}(\phi)) &= \bfE^\NA_\rv(Q_L(\phi)),
    \end{align}
    where $\bfE^\NA_\rv(\cdot)$ denotes the extended weighted Monge--Amp\`ere energy of an arbitrary function $\phi:X^\an\rightarrow\RR$, defined by
    \begin{align} 
    \label{eq:extended_E_v^NA}
        \bfE_\rv^\NA(\phi)
        \coloneqq \sup\{\bfE_\rv^\NA(\phi):\phi\in\mcH_T(X^\an,L^\an), \phi\leq \phi\}.
    \end{align}
\end{corollary}
\begin{proof}
By Proposition~\ref{FS-IN=Q_L}, we only need to show \eqref{vol_v()=E_v^NA(FS)_any_chi}.
    We follow the argument in the proof of \cite[Theorem 5.1]{BJ24}.

    We may assume that $\chi\in\mcN_\ZZ(X,L)$. Then the canonical approximation $\chi_d\in\mcT_\ZZ(X,L)$ and $\FS(\chi_d)=\FS_d(\chi)$ increases pointwise to $\FS(\chi)$.
    
    A Dini-type argument (\cite[Proposition~8.3~(v)]{BJ22}) yields that $\bfE_\rv^\NA(\cdot)$ is continuous along increasing sequences of bounded-below lsc functions.
    Hence $\bfE_\rv^\NA(\FS(\chi_d))\rightarrow\bfE_\rv^\NA(\FS(\chi))$. Together with Lemma~\ref{vol_converge_chi_d}, we reduce to the case $\chi\in\mcT_\ZZ(X,L)$.

    By Lemma~\ref{vol_E_v}, one has
    \begin{align*}
        \vol_{\rv_i}(\chi)=\bfE_{\rv_i}^\NA(\FS(\chi)).
    \end{align*}
    By \eqref{eq:vol_v(chi)=integral} and Lemma~\ref{lem:E_v_i_convergence}, taking $i\rightarrow$ implies \eqref{vol_v()=E_v^NA(FS)_any_chi}.
\end{proof}

\begin{comment}
{\color{blue} modify}

Let $(\mcX,\mcL)$ be a $T_\CC$-equivariant test configuration. Then
\begin{align*}
    \bfE^\NA_\rv (\phi_\mcL) &= \sum_{\vec{k}} a_{\vec{k}} \frac{n!}{(n+k+1)!} \big((\mcL^{[\vec{k}]})^{n+k+1} - (L_{\PP^1}^{[\vec{k}]})^{n+k+1} \big) \\
    &= \mcD \cdot \sum_{\vec{k}} a_{\vec{k}} \frac{n!}{(n+k)!}  \big( \sum_{i=0}^{n+k} (\mcL^{[\vec{k}]})^{n+k-i}\cdot (L^{[\vec{k}]}_{\PP^1})^{i}\big) \\
    &= \mcD \cdot \sum_{i=0}^{n}\sum_{\vec{k}} a_{\vec{k}} \frac{n!}{(n+k)!}  \big(  (\mcL^{[\vec{k}]})^{n+k-i}\cdot (L^{[\vec{k}]}_{\PP^1})^{i}\big)\\
    &= \sum_{i=0}^n \int_{X^\an} (\phi_{\mcL}-\phi_{\triv}) \; \MA^\NA_\rv (\phi_{\triv}^{[i]},\phi^{[n-i]}_\mcL),
\end{align*}
where $\MA^\NA_\rv (\phi_{\triv}^{[i]},\phi^{[n-i]}_\mcL) = \sum_{\vec{k}} a_{\vec{k}} \frac{n!}{(n+k)!}  \big(  (\mcL^{[\vec{k}]})^{n+k-i}\cdot (L^{[\vec{k}]}_{\PP^1})^{i}\big)$ is a weighted twisted Monge-Ampere measure. The second and the third equality above hold since the text configuration is $T_\CC$-equivariant.
\end{comment}

For any $\phi\in\mcE_\rv^1(X^\an,L^\an)$, one defines the weighted non-Archimedean entropy as follows
\begin{align}
    \bfH^\NA_\rv(\phi) \coloneqq \int_{X^\an} A_X(x) \; \MA^\NA_\rv(\phi).
\end{align}
Then the weighted non-Archimedean Mabuchi functional is defined as
\begin{align}
    \bfM_{\rv,\rw}^\NA(\phi) &
         = \bfH_\rv^\NA(\phi) + \bfR_\rv^\NA(\phi) 
           +c_{(\rv,\rw)}(L)  \bfE_{\rw}^\NA(\phi).
\end{align}
As proved in  \cite[Proposition~5.8, Proposition~B.8]{HL23}, we have the following slope formula for test configurations.
\begin{lemma}
    \label{lemma_slope_formula}
    Let $\phi\in\mcH_T(X^\an,L^\an)$, $\Phi$ be a $T$-invariant geodesic ray with $\Phi^\NA = \phi$ $($see \textup{Definition~\ref{Phi_NA_div}}$)$. Then
    \begin{align}
        \bfF^{\prime\infty}(\Phi) = \bfF^\NA(\phi),
    \end{align}
    where $\bfF$ can be $\bfE_\rv, \bfE^\mathcal{Q}_\rv, \bfJ_\rv,\bfH_\rv$.
\end{lemma}

When $\phi = \phi_\mcL$ is induced by a test configuration $(\mcX,\mcL)$,  then
\begin{align}
\label{M_v,w^NA_(mcL)}
    \bfM_{\rv,\rw}^\NA(\phi_\mcL) = (K^{\log}_{\mcX/\PP^1}\cdot \mcL^n)_\rv + {c_{\rv,\rw}(L)}(\mcL^{n+1})_{\rw},
\end{align}
where we denote
\begin{align}
    \label{E_L_v}
        (\mcL^{n+1})_{\rw} &\coloneqq \lim \sum_{\vec{k}}b_{\vec{k}}\frac{k_1!\cdots k_r!}{(n+k+1)!} (\mcL^{[\vec{k}]})^{n+k+1},\\
        (K^{\log}_{\mcX/\PP^1}\cdot\mcL^{n})_\rv
        %\coloneqq \sum_{\vec{k}}a_{\vec{k}}\frac{\vec{k}!}{n!}(E\cdot\mcL^n)_{\vec{k}}
        &\coloneqq \lim \sum_{\vec{k}}a_{\vec{k}}\frac{k_1!\cdots k_r!}{(n+k)!} (K^{\log}_{\mcX/\PP^1})^{[\vec{k}]}\cdot(\mcL^{[\vec{k}]})^{n+k}, 
\end{align}
where the existence of the right-hand-sides is by Lemma \ref{lem:E_v_i_convergence}.

When the test configuration $(X_{\PP^1},\mcL_\xi)$ is induced by a holomorphic vector field $\xi$, we define the weighted Futaki invariant as
\begin{align}
    \Fut_{\rv,\rw}(\xi) &
         = (K^{\log}_{X_{\PP^1}/\PP^1}\cdot \mcL_\xi^n)_\rv + {c_{\rv,\rw}(L)}(\mcL_\xi^{n+1})_{\rw}.
\end{align}

In \cite[Section~3]{Li22b}, the author defines a twist of non-Archimedean potential $\phi_\xi$ for any $\xi\in N_\RR$ and computes the difference $\bfE^\NA(\phi_\xi)-\bfE^\NA(\phi)$, which is the Futaki invariant of $\xi$.
By modifying the arguments in \cite[Lemma 2.24]{Li22a}, one has
\begin{align}
    \bfE^\NA_\rv(\phi_\xi) &= \bfE^\NA_\rv(\phi) + {\rm CW}_{L,\rv}(\xi), \\
     \bfM^\NA_{\rv,\rw}(\phi_\xi) &= \bfM^\NA_{\rv,\rw}(\phi) + \Fut_{\rv,\rw}(\xi),
\end{align}
where ${\rm CW}_{L,\rv}(\xi) = \frac{1}{\vol_\rv(\mathbb{O}) n!}\int_P \big(\int_{\mathbb{O}_{\underline{y}}} G_{\mathcal{F}_\xi} \; dy\big)\; \rv(\underline{y})\; d \underline{y}$.

For any $\phi,\psi\in\mcE_\rv^1(X^\an,L^\an)$, we set
\begin{align}
\bfJ_\rv^\NA(\phi) 
        \coloneqq& \VV_\rv\cdot \sup_{X^\an}(\phi-\phi_\triv) 
          - \bfE_\rv^\NA(\phi), \\
    \bfJ^\NA_{\rv,\TT}(\phi) \coloneqq& \inf_{\xi\in N_\RR} \bfJ^\NA_\rv(\phi_\xi),  \\
    \bfI_\rv^\NA(\phi,\psi)
    \coloneqq& \int_{X^\an}(\phi-\psi)(\MA_\rv^\NA(\psi)-\MA_\rv^\NA(\phi)),
\end{align}
where the Lie algebra of $\TT$ is identified with $N_\RR$.

By a direct modification of the proof of \cite[Lemma 6.4]{Li22a} to the weighted case, we have
\begin{lemma}
\label{reduced_J_v_convergence}
    Let $\phi_j,\phi\in\mcE_\rv^1(X^\an,L^\an)$ and $\phi_j\rightarrow\phi$ in the strong topology, then $\bfJ_{\rv,\TT}^\NA(\phi_j)\rightarrow\bfJ_{\rv,\TT}^\NA(\phi)$.
\end{lemma}

\begin{definition}
    Let $L$ be an ample $\QQ$-line bundle on $X$. We say $(X,L)$ is $\bG$-\textit{uniformly weighted K-stable $($for models$)$} if there exists $\gamma>0$ such that for any $\bG$-equivariant test configuration $(\mcX,\mcL)$ of $(X,L)$ ($\bG$-equivariant model $(\mcX,\mcL)$ of $(X,L)$)
    \begin{align*}
        \bfM_{\rv,\rw}^\NA(\phi_\mcL)
        \geq \gamma\  \bfJ^\NA_{\rv,\TT}(\phi_\mcL).
    \end{align*}
\end{definition}

\begin{lemma}
\label{I_v_I}
    There exists a constant $C>1$, such that for any $\phi_1,\phi_2\in \mcH(X^\an,L^\an)$,
    \begin{align}
        \frac{1}{C} \bfI^\NA(\phi_1,\phi_2) \leq \bfI^\NA_\rv(\phi_1,\phi_2) \leq C\bfI^\NA(\phi_1,\phi_2),
    \end{align}
    \begin{align}
        \frac{1}{C}\rmd_1(\phi_1,\phi_2)\leq \rmd_{\rv,1}(\phi_1,\phi_2)\leq C \rmd_1(\phi_1,\phi_2).
    \end{align}
\end{lemma}
\begin{proof}
    The first formula follows from \cite[(2.31)]{HL23} and the slope formula.   

    Meanwhile, since $\rv$ has positive upper and lower bounds, by formula \eqref{d_1_v_d_1}, $\rmd_{\rv,1}$ is comparable to $\rmd_1$. The second formula follows.
\end{proof}

Similarly, there also exists a constant $C>1$ such that for any $\phi\in\mcH(X^\an,L^\an)$, $\frac{1}{C}\bfJ^\NA(\phi)\leq \bfJ^\NA_\rv(\phi)\leq C\bfJ^\NA(\phi)$.
Therefore, we have the following corollary.

\begin{corollary}
\label{E_v_E1}
    The two spaces $\mcE^{1}(X^\an,L^\an)$, $ \mcE^{1}_\rv(X^\an,L^\an)$ are homeomorphic with each other with respect to strong topologies.
\end{corollary}
\begin{lemma}
    \label{M_v_M}
    Let $\phi\in\mcH(X^\an,L^\an)$, $\mu = \MA^\NA_\rv(\phi)$. Then
    \begin{align}
        \bfE_\rv^{*,\NA}(\mu) = \bfI^\NA_\rv(\phi)-\bfJ^\NA_\rv(\phi),
    \end{align}
    where $\bfI_\rv^\NA(\phi)\coloneqq\bfI_\rv^\NA(\phi,0)$ and
    \begin{align}
    \label{E_v^*}
        \bfE_\rv^{*,\NA}(\mu)
        \coloneqq\sup_{\psi\in\mcH(X^\an,L^\an)}\Big( \bfE^\NA_\rv(\psi)-\int_{X^\an}\psi\;\mu\Big)
    \end{align}
    And there exists a constant $C>0$, such that
    \begin{align}
        \frac{1}{C}\bfJ^\NA_\rv(\phi)\leq \bfE^{*,\NA}_\rv(\mu)\leq C\bfJ^\NA_\rv(\phi).
    \end{align}
\end{lemma}
\begin{proof}
    Let $\psi$ be any function in $\mcH(X^\an,L^\an)$.
    Let $\Phi, \Psi$ be geodesic rays with $\Phi^\NA=\phi, \Psi^\NA = \psi$.
    By \cite[(2.31)]{HL23}, 
    \begin{align*}
        \bfE_{\rv}(\Psi(s))-\bfE_{\rv}(\Phi(s))-\int_X (\Psi(s)-\Phi(s))\;\MA_\rv(\Phi(s)) 
        = -\bfJ_{\rv,\Phi(s))}(\Psi(s)) \leq 0.
    \end{align*}
    Apply Lemma~\ref{lemma_slope_formula} to the inequality above, we have
    \begin{align}
        \bfE^\NA_\rv(\phi)-\bfE^\NA_\rv(\psi) \geq \int_{X^\an} (\phi-\psi)\; \MA_\rv^\NA(\phi).
    \end{align}
    Then
    \begin{align*}
        \bfE^{*,\NA}(\mu)=\sup_{\psi\in\mcH(X^\an,L^\an)}\Big( \bfE^\NA_\rv(\psi)-\int_{X^\an}\psi\;\MA^\NA_\rv(\phi)\Big) \leq \bfE^\NA_\rv(\phi)-\int_{X^\an}\phi\;\mu.
    \end{align*}
    On the other hand, by definition,
    \begin{align*}
        \bfE^{*,\NA}_\rv(\mu) \geq \bfE^\NA_\rv(\phi)-\int_{X^\an}\phi\; \mu.
    \end{align*}
    Therefore
    \begin{align}
        \bfE^{*,\NA}_\rv(\mu) = \bfE^\NA_\rv(\phi)-\int_{X^\an}\phi\; \mu = \bfI^\NA_\rv(\phi)-\bfJ^\NA_\rv(\phi).
    \end{align}
    Apply \cite[(2.28), (2.29)]{HL23} and Lemma~\ref{lemma_slope_formula}, there exist constants $C,C'>1$, such that
    \begin{align}
        \frac{1}{C}\bfJ^\NA_\rv(\phi)\leq \frac{1}{C'}(\bfI^\NA(\phi)-\bfJ^\NA(\phi))
        &\leq (\bfI^\NA_\rv(\phi)-\bfJ^\NA_\rv(\phi)) \nonumber \\
        &\leq {C'}(\bfI^\NA(\phi)-\bfJ^\NA(\phi)) \leq {C}\bfJ^\NA_\rv(\phi).
    \end{align}
\end{proof}

\begin{proposition}
\label{E_v_E}
    There exists a bi-Lipschitz map:
    \begin{align}
        & (\mcE^{1,\NA}(X^\an,L^\an)/\RR, \bfI^\NA) \rightarrow (\mcE^{1,\NA}_\rv(X^\an,L^\an)/\RR, \bfI_\rv^\NA).
       % & (\mcM^{1,\NA}(X^\an), \bfE^{*,\NA}) \rightarrow (\mcM^{1,\NA}_\rv(X^\an), \bfE_\rv^{*,\NA}).
    \end{align}
\end{proposition}
\begin{proof}
    By Lemma~\ref{I_v_I}, the completions of $\mcH(X^\an,L^\an)$ with respect to $\bfI^\NA$ and $\bfI_\rv^\NA$ are the same. Therefore $\mcE^{1,\NA}(X^\an,L^\an)$ is homeomorphic to $\mcE^{1,\NA}_\rv(X^\an,L^\an)$. The bi-Lipschitz property is also induced by Lemma~\ref{I_v_I}. 
\end{proof}

\begin{remark}
    Darvas and Xia produce an alternative formula for $\bfE^\NA, (\bfE^{Q})^\NA, \bfJ^\NA$ by employing the method of the test curve. By a straightforward equivariant modification of their results (\cite[Theorem~3.7]{DX22}, \cite[Theorem~6.7]{Xia23}, \cite[Proposition~3.10]{DZ24}), we also have the following alternative definitions of the non-Archimedean functionals.

    Let $\phi$ be a $T_\CC$-equivariant test configuration. Let $\psi_\bullet$ be the corresponding test curve that emanates from $\varphi_0\in \mcH(X,L)$. Then by working on the resolution $Y$, we have
    \begin{align*}
        \bfE^\NA_\rv(\phi) &= \frac{1}{\mathbb{V}_\rv}\int_{-\infty}^0 \Big(\int_Y \rv(m_{\psi_\tau})\frac{\omega_{\psi_\tau}^n}{n!}-\int_Y \rv(m_{\varphi_0})\frac{\omega_{\varphi_0}^n}{n!} \Big) \; d\tau. \\
        (\bfE^Q_\rv)^\NA(\phi) 
        &= \frac{1}{\mathbb{V}_\rv}\int_{-\infty}^0 \Big(\int_Y \rv(m_{\psi_\tau})\frac{\alpha\wedge\omega_{\psi_\tau}^{n-1}}{(n-1)!}+\int_Y \langle d\rv(m_{\psi_\tau}), m_\alpha\rangle\frac{\omega_{\psi_\tau}^n}{n!} \Big) \\
        &\qquad-\Big( \int_Y \rv(m_{\varphi_0})\frac{\alpha\wedge\omega_{\varphi_0}^{n-1}}{(n-1)!}+\int_Y \langle d\rv(m_{\varphi_0}), m_\alpha\rangle\frac{\omega_{\varphi_0}^n}{n!} \Big) \; d\tau,
    \end{align*}
    where $\alpha\in c_1(Q)$ is a smooth closed $(1,1)$-form.
    If the conjecture 8.8 in \cite{Xia25} holds, then
    there is also a chance to express $(\bfE_\rv^Q)^\NA$ by using the weighted mixed volume of Okounkov body.
\end{remark}

\subsection{Weighted non-Archimedean Monge--Amp\`ere equations}

As in \cite[\S~8]{BJ22}, for any $\phi:X^\an\rightarrow\RR$, one defines
\begin{align*}
    \bfE^{\downarrow,\NA}_\rv(\phi)
    \coloneqq\inf\{\bfE_\rv^\NA(\psi):\psi\in C^0(X^\an), \psi\geq \phi\} 
\end{align*}

\begin{remark}
    By a weighted version of \cite[Proposition 8.3(iv)]{BJ22}, if the envelope property holds and $\phi$ is an u.s.c function that is bounded above, or $\phi\in C^0(X^\an)$, we have $\bfE^{\downarrow,\NA}_\rv(\phi) 
    = \bfE^\NA_\rv({\rm P}_L(\phi))$.
\end{remark}

\begin{lemma}
\label{difference_E_v_main}
    Let $f\in {\rm PL}(X^\an)$. For any $\phi\in\mcE^{1,\NA}_\rv(X^\an,L^\an)$,
\begin{align}
\label{difference_of_E_v_formula_1}     
\bfE^{\downarrow,\NA}_\rv(\phi+\epsilon f) 
= \bfE^\NA_\rv(\phi)+\epsilon\int_{X^\an}f\; \MA^\NA_\rv(\phi)+O(\epsilon^2),
    \end{align}
    where $O(\epsilon^2)$ is uniform with respect to $\phi$. In particular, $\bfE^\NA_\rv$ is differentiable,
    \begin{align}
        \left.\frac{d}{d\epsilon}\right|_{\epsilon=0}\bfE^{\downarrow,\NA}_\rv(\phi+\epsilon f) 
        = \int_{X^\an}f\;\MA^\NA_\rv(\phi).
    \end{align}
\end{lemma}

\begin{proof}
    For any $\phi\in\mcE^{1,\NA}_\rv(X^\an,L^\an)$, there exists a decreasing sequence $\phi_i\in\mcH_T(X^\an,L^\an)$ converging to $\phi$.
    Then $\bfE_\rv^\NA(\phi_i)\rightarrow\bfE_\rv^\NA(\phi)$ and $\MA_\rv^\NA(\phi_i)\rightarrow\MA_\rv^\NA(\phi_i)$ weakly. By a weighted version of \cite[Proposition 8.3(v)]{BJ22}, then $\bfE_\rv^\NA(\phi_i+\epsilon f)\rightarrow \bfE_\rv^{\downarrow,\NA}(\phi+\epsilon f)$.
    
    Without the loss of generality, we may assume $\phi\in\mcH_T(X^\an,L^\an)$.
There exists an ample line bundle $M$ over $X$ and  $\psi_1,\psi_2\in \mcH(X^\an,M^\an)$, such that $f = \psi_1-\psi_2$. 
    Let $C =  \vol_\rv(L+\epsilon M)- \vol_\rv( L)>0$, 
    where
    \begin{align*}
        \vol_\rv(L) = \lim\sum_{\vec{k}}a_{\vec{k}} \frac{k_1!\cdots k_r!}{(n+k)!} (L^{[\vec{k}]})^{n+k}.
    \end{align*}
    Note that $C = O(\epsilon)$.
    By Lemma~\ref{difference_of_E_v}, we have
\begin{align}
\label{difference_of_E_v_middle_step_1}
    \left|\bfE_\rv^\NA({\rm P}_L(\phi+\epsilon f))-\bfE^\NA_\rv(\phi)-\epsilon \int_{X^\an}f\;\MA^\NA_\rv(\phi+\epsilon\psi_1)\right| 
    \leq \epsilon \cdot C\cdot \sup_{X^\an} |f| \leq O(\epsilon^2).
\end{align}
Meanwhile,
\begin{align*}
    & \left|\int_{X^\an}f\; \MA^\NA_\rv(L+\epsilon M, \phi+\epsilon\psi_1) - \int_{X^\an} f\; \MA^\NA_\rv(L,\phi)\right| \\
    \leq & \sup_{X^\an}|f|\cdot \Big(\int_{X^\an} \MA^\NA_\rv(L+\epsilon M, \phi+\epsilon\psi_1) - \MA^\NA_\rv(L,\phi)\Big) = O(\epsilon).
\end{align*}
Plug into \eqref{difference_of_E_v_middle_step_1}, we have
\begin{align*}
    \left|\bfE^{\NA}_\rv({\rm P}_L(\phi+\epsilon f))-\bfE^\NA_\rv(\phi)-\epsilon\int_{X^\an}f\; \MA^\NA_\rv(\phi)\right| 
    \leq O(\epsilon^2),
\end{align*}
which concludes the inequality \eqref{difference_of_E_v_formula_1}.
\end{proof}

\begin{lemma}
\label{difference_of_E_v}
Let $L,M$ be ample line bundles on $X$, $\phi\in \mcH_T(X^\an,L^\an)$, $\psi_1,\psi_2\in \mcH(X^\an,M^\an)$, $f=\psi_1-\psi_2$.
Let $\rv$ be a smooth positive function on $P$. Then
     \begin{align}
     \label{bound_E_quantization}
        C \inf_{x\in X^\an} f(x) \leq \int_{X^\an}f\; \MA^\NA_\rv(\phi+\psi_1) 
        - {\rm vol}_\rv\big( L, (\phi+f), \phi \big) \leq C \sup_{x\in X^\an} f(x) ,
    \end{align}
    where $C = \vol_\rv(L+M)-\vol_\rv(L)$ and
    \begin{align*}
        \vol_\rv( L, (\phi+f), \phi )
        \coloneqq\vol_\rv(\IN(\phi+f))-\vol_\rv(\IN(\phi)).
    \end{align*}
\end{lemma}
\begin{proof}
    We will first assume that $\rv$ is a polynomial.
    Let $(\mcX,\mcL), (\mcX,\mcM_1),(\mcX,\mcM_2)$ be test configurations corresponding to $\phi,\psi_1,\psi_2$. Without the loss of generality, we may assume $\lambda_{\min}(\mcL)=0$.

    Replace $f$ by $f-\inf_{X^\an}f$, let $\phi_E = f \geq 0$ and $\mcM_1-\mcM_2=\oo_\mcX(E)$.
    For $1\leq j\leq m$, consider
    \begin{align*}
        0\rightarrow H^0(\mcX, m\mcL+(j-1)E)\rightarrow H^0(\mcX,m\mcL+jE)\rightarrow H^0(E,(m\mcL+jE)|_E).
    \end{align*}
    One observes
    \begin{align*}
        &\frac{1}{m^{n+1}}\sum_{s_i\in H^0(\mcX,m(\mcL+E))} \rv(y_1(s_i),\ldots,y_r(s_i))-\frac{1}{m^{n+1}}\sum_{s_i\in H^0(\mcX,m\mcL)} \rv(y_1(s_i),\ldots,y_r(s_i)) \\
        =&\frac{1}{m^{n+1}}\sum_{j=1}^m\left(\sum_{s_i\in H^0(\mcX,m\mcL+jE)} \rv(y_1(s_i),\ldots,y_r(s_i))-\sum_{s_i\in H^0(\mcX,m\mcL+(j-1)E)} \rv(y_1(s_i),\ldots,y_r(s_i)) \right) \\
        \leq&\frac{1}{m^{n+1}}\sum_{j=1}^m\left(\sum_{s_i\in H^0(E,(m\mcL+jE)|_E)} \rv(y_1(s_i),\ldots,y_r(s_i)) \right)   \\
        \leq&\frac{1}{m^{n+1}}\sum_{j=1}^m \left(\sum_{s_i\in H^0(E,m(\mcL+\mcM_1)|_E)} \rv(y_1(s_i),\ldots,y_r(s_i))\right)  \\
        \leq&\frac{1}{m^{n}}\sum_{s_i\in H^0(E,m(\mcL+\mcM_1)|_E)} \rv(y_1(s_i),\ldots,y_r(s_i)). 
    \end{align*}
    By formula \eqref{S_v=vol_v} and \cite[Theorem~2.4]{BGM22}, we have 
    \begin{align}
        \label{vol_v_upbd}
        &\vol_\rv(L,(\phi+f),\phi)  \\
       =& \lim_{m\to\infty}\left(  \frac{1}{m^{n+1}}\sum_{s_i\in H^0(\mcX,m(\mcL+E))} \rv(y_1(s_i),\ldots,y_r(s_i)) 
       -\frac{1}{m^{n+1}}\sum_{s_i\in H^0(\mcX,m\mcL)} \rv(y_1(s_i),\cdots,y_r(s_i))  \right) \nonumber\\
       \leq& \lim_{m\to\infty}\left(  \frac{1}{m^{n}}\sum_{s_i\in H^0(E,m(\mcL+\mcM_1)|_E)} \rv(y_1(s_i),\ldots,y_r(s_i)) \right) \nonumber \\
        =&\lim_{m\rightarrow\infty}\left( \frac{1}{m^n}\sum_{\vec{k}} a_{\vec{k}}\cdot\frac{k_1!\cdots k_r!}{m^k}\cdot \sum_{\eta_1,\ldots,\eta_r\in\ZZ}\sum_{s_i\in H^0(E,m(\mcL+\mcM_1)|_E)_{\eta_1,\ldots,\eta_r}} c^{[\vec{k}]}_{i,\eta_1,\ldots,\eta_r} \right) \nonumber \\
            =& \lim_{m\to\infty} \sum_{\vec{k}} a_{\vec{k}}\cdot \frac{k_1!\cdots k_r!}{m^{n+k}}\cdot h^0(E^{[\vec{k}]},m(\mcL+\mcM_1)^{[\vec{k}]}|_{E^{[\vec{k}]}})  \nonumber \\
            =&  \sum_{\vec{k}} a_{\vec{k}} \cdot \frac{k_1!\cdots k_r!}{(n+k)!} \int_{(X^{[\vec{k}]})^\an} \phi_E^{[\vec{k}]}\; {(\ddc\phi^{[\vec{k}]}_{\mcL+\mcM_1})^{n+k}} \nonumber\\
            =&  \int_{X^\an}\phi_E \; \MA^\NA_\rv(\phi+\psi_1). \nonumber
    \end{align}
     Hence
    \begin{align*}
        C\cdot\inf_{X^\an} f(x) \leq \int_{X^\an}f\; \MA^\NA_\rv(\phi+\psi_1) - {\rm vol}_\rv(L,(\phi+f),\phi),
    \end{align*}
    where $C =  \vol_\rv(L+M) - \vol_\rv(L)$.
    The second inequality of \eqref{bound_E_quantization} is proved similarly.
  
    This concludes the proof when $\rv$ is a polynomial. When $\rv$ is a smooth function, $\rv$ can be approximated by a sequence of polynomials. Since the constant $C$ in \eqref{bound_E_quantization} is uniformly bounded, by taking a limit, we obtain \eqref{bound_E_quantization} for smooth function $\rv$.

\end{proof}

The space $\mcM_\rv^{1,\NA}\subset\mcM$ of \textit{measures of finite energy} is defined as
\begin{align*}
    \mcM_\rv^{1,\NA}\coloneqq\{\mu\in\mcM:\bfE_\rv^{*,\NA}(\mu)<+\infty\}.
\end{align*}
A \textit{maximal sequence} for $\mu$ is a sequence $\phi_i\in\mcH(X^\an,L^\an)$ satisfying
\begin{align*}
    \bfE_\rv^{*,\NA}(\mu)-\bfE_\rv^\NA(\phi_i)+\int_{X^\an}\phi\mu\rightarrow0.
\end{align*}
We have the following weighted non-Archimedean Calabi--Yau theorem.
\begin{theorem}
    \label{weighted_NA_Calabi_problem}
    Assume $(X,L)$ satisfies the envelope property.
    Then the weighted Monge-Amp\`ere operator is a Bi-Lipschitz map %({\color{blue} use I as distance})
    \begin{align}
    \label{MA_v_formula}
        \MA^\NA_\rv: (\mcE^{1,\NA}_\rv(X^\an,L^\an)/\RR,\bfI^\NA_\rv) \rightarrow (\mathcal{M}^{1,\NA}_\rv, \bfE^{*,\NA}_\rv).
    \end{align}
    If in addition, ${\rm supp}(\mu)\subset \{T\leq C\}$ $($see \textup{\cite[Definition~1.26]{BJ22}} for the definition of $T$$)$, then the solution of $\MA^\NA_\rv(\phi) = \mu$ is continuous.
\end{theorem}
\begin{proof}
    The proof follows the same lines as the proofs \cite[Theorem 12.8, Theorem 12.12]{BJ22}. For readers' convenience, we provide a sketch the proof. The injectivity is implied by Lemma~\ref{M_v_M}. For any $\mu\in\mcM^{1,\NA}_\rv$, let $\phi_i\in\mcH_T(X^\an,L^\an)$ be a maximizing sequence for $\mu$, with $\sup_{X^\an}\phi_i=0$. The envelope property implies that $\phi_i$ converges to $\phi\in \PSH_T(X^\an,L^\an)$ with respect to the weak topology. 
    
    For any smooth $\omega$-psh functions $\varphi_1,\varphi_2$, by \cite[(2.31),(3.36),(2.38)]{HL23}, $\bfJ_{\rv,\varphi_1}(\varphi_2)\geq C\cdot\bfI_\rv(\varphi_2,\varphi_1)$. Then by slope formula Lemma~\ref{lemma_slope_formula} and \cite[Lemma 7.29]{BJ22}, for any $\phi_0,\phi_1\in\mcH_T(X^\an,L^\an)$, $t\in [0,1]$, we have
    \begin{align*}
        \bfE^\NA_\rv\big((1-t)\phi_1+t\phi_0\big)-\big((1-t)\bfE^\NA_\rv(\phi_1)+t\bfE^\NA_\rv(\phi_0)\big) \geq C\cdot t(1-t)\bfI^\NA_\rv(\phi_1,\phi_0).
    \end{align*}
    Then by the approximation of piecewise-linear psh functions, the inequality above holds for any $\phi_0,\phi_1\in\mcE^{1,\NA}_\rv(X^\an,L^\an)$.
    
    Let $t=\frac{1}{2}$, $\phi_1=\phi_i$, $\phi_0=0$, we have
    \begin{align*}
        \bfE^{*,\NA}_\rv(\mu)\geq \bfE^\NA_\rv(\frac{1}{2}\phi_i)-\int_{X^\an}\frac{1}{2}\phi_i\;\mu \geq \frac{1}{2}\Big(\bfE^\NA_\rv(\phi_i)-\int_{X^\an}\phi_i\;\mu\Big)+C\cdot \bfI^\NA_\rv(\phi_i),
    \end{align*}
    which implies that $\bfJ^\NA_\rv(\phi_i)$ is uniformly bounded from above by a constant $C$. Along the weakly convergent sequence $\phi_i$, $\bfE^\NA_\rv$ is u.s.c, therefore $\bfJ_\rv^\NA$ is l.s.c, and $\bfJ_\rv^\NA(\phi) \leq C$. Then $\phi\in \mcE^{1,\NA}_\rv(X^\an,L^\an)$. 

    By the same argument as the proof of \cite[Proposition 9.19]{BJ22}, $\lim_{i\to\infty}\int_{X^\an}\phi_i\;\mu = \int_{X^\an}\phi\;\mu$. Then, by the upper-semi continuity of $\bfE^\NA_\rv$, $\bfE^\NA_\rv(\phi)-\int_{X^\an} \phi\; \mu \geq \bfE^{*,\NA}_\rv(\mu)$, and consequently, the equality holds.

    Next, we show that $\bfE^\NA_\rv(\phi)-\int_{X^\an} \phi\; \mu = \bfE^{*,\NA}_\rv(\mu)$ implies that $\MA^\NA_\rv(\phi)=\mu$.
    
    For any $f\in {\rm PL}(X^\an)$, 
    \begin{align*}
        \bfE^\NA_{\rv}({\rm P}_L(\phi+\epsilon f)) - \int_{X^\an}(\phi+\epsilon f) \; \mu 
        \leq \bfE_\rv^{*,\NA}(\mu)
        =\bfE^\NA_\rv(\phi)-\int_{X^\an}\phi \; \mu.
    \end{align*}
    Together with Lemma~\ref{difference_E_v_main}, we have
    \begin{align*}
        \epsilon \int_{X^\an}f\; \MA_\rv^\NA(\phi) = \bfE^\NA_\rv({\rm P}_L(\phi+\epsilon f))-\bfE^\NA_\rv(\phi)+O(\epsilon^2) 
        \leq \epsilon\int_{X^\an}f\;\mu+O(\epsilon^2).
    \end{align*}
    As $\epsilon\to 0$, we have $\int_{X^\an}f\;\MA^\NA_\rv(\phi) \leq \int_{X^\an}f\;\mu$. Replace $f$ by $-f$, we obtain $\int_{X^\an}f\;\MA^\NA_\rv(\phi) = \int_{X^\an}f \; \mu$. This implies that $\MA^\NA_\rv(\phi)=\mu$ as measures. This concludes the surjectivity of \eqref{MA_v_formula}.

    Next we assume ${\rm supp}(\mu)\subset \{T\leq C\}$. As long as the domination principle \cite[Corollary~10.6]{BJ22} also holds for the weighted Monge--Amp\`ere measures, the same proof of \cite[Theorem~12.12]{BJ22} would imply that the solution of $\MA^\NA_\rv(\phi)=\mu$ is continuous.
    
    It suffices to show the weighted domination principle, that is, for $\phi\in\PSH_T(X^\an,L^\an)$, $\psi\in\mcE^{1,\NA}_\rv(X^\an,L^\an)$, if $\phi\leq\psi$ a.e. with respect to $\MA_\rv^\NA(\psi)$, then $\phi\leq \psi$ on $X^\an$.

    Replace $\phi$ by $\sup(\phi,\psi)$, then $\phi\in\mcE^{1,\NA}_\rv(X^\an,L^\an)$. Then
    \begin{align*}
        \bfI^\NA_\rv(\phi,\psi) = \int_{X^\an}(\phi-\psi)\big(\MA^\NA_\rv(\psi)-\MA^\NA_\rv(\phi)\big)\leq \int_{X^\an}(\phi-\psi)\MA^\NA_\rv(\psi)=0.
    \end{align*}
    Then $\bfI^\NA(\phi,\psi)\leq C\bfI^\NA_\rv(\phi,\psi)=0$. Then by \cite[Corollary 10.4]{BJ22}, $\phi = \psi+C$, and $C=0$ since $\psi=\psi$ a.e. with respect to $\MA^\NA_\rv(\psi)$.
\end{proof}
\begin{remark}
    Theorem~\ref{weighted_NA_Calabi_problem} induces a bi-Lipschitz map from 
    \begin{align*}
        (\mcM^{1,\NA}, \bfE^{*,\NA})\rightarrow (\mcM^{1,\NA}_\rv(X^\an),\bfE_\rv^{*,\NA}).
    \end{align*}
\end{remark}

\begin{theorem}
    \label{Thm_M_div_MA}
    There exists a bi-Lipschitz map
    \begin{align}
        \MA^\NA_\rv\circ {\rm FS}: (\mcN^\div_\RR/\RR, \underline{\rmd}_{\rv,1})\rightarrow (\mcM^{\div}_\rv, \rmd_{\rv,1}).
    \end{align}
\end{theorem}
\begin{proof}
    Let $\Sigma\subset X^\div$ be a finite subset and pick $\chi\in\mcN_\RR^\Sigma$. %Let $\phi\in C^0(X^\an)\cap \PSH^\NA_T(X^\an,L^\an)$.
    First we show that $\MA^\NA_\rv({\rm FS}(\chi)) \in \mcM^{\div}_\rv$, which is equivalent to show ${\rm supp}\big(\MA^\NA_\rv({\rm FS}(\chi))\big) \subset \Sigma$.
    Let  $f$ be any continuous function on $X^\an$ and $f|_{\Sigma} = 0$.
    Since ${\rm IN}\big({\rm FS}(\chi)+t f\big) \leq {\rm IN}_\Sigma\big({\rm FS}(\chi)+t f\big)$,
    then
    \begin{align*}
        \bfE^\NA_\rv\big({\rm FS}(\chi)+\epsilon f\big) &= \vol_\rv\big({\rm IN}({\rm FS}(\chi)+\epsilon f)\big) \\
            &\leq \vol_\rv\big({\rm IN}_{\Sigma}({\rm FS}(\chi)+\epsilon f)\big) = \vol_\rv\big({\rm IN}_{\Sigma}({\rm FS}(\chi))\big) =\vol_\rv(\chi),
    \end{align*}
    where we have used that $\chi=\IN_\Sigma(\FS(\chi))$ (see Lemma~\ref{IN_Sigma(FS)}) for the last equality.
    By the differentiability result of $\bfE^\NA_\rv$ as shown in Lemma~\ref{difference_E_v_main}, we have $\int_{X^\an}f\; \MA^\NA_\rv({\rm FS}(\chi)) = 0$. Therefore ${\rm supp}\big(\MA^\NA_\rv({\rm FS}(\chi))\big) \subset \Sigma$.

    Next we show the morphism in the statement is onto.
    Pick any $\mu\in \mcM^{\div}_\rv$ that is supported on a finite set $\Sigma\subset X^\div$. By Theorem~\ref{weighted_NA_Calabi_problem}, there exists $\phi\in C^0(X^\an)\cap \PSH_T(X^\an,L^\an)$, such that $\MA_\rv^\NA(\phi) = \mu$. Let $\chi = {\rm IN}_\Sigma(\phi)$. 
    %Then ${\rm FS}(\chi) \geq \phi$, and ${\rm FS}(\chi)|_\Sigma = \phi|_\Sigma$. 
    Since $\chi=\IN_\Sigma(\phi)\geq \IN(\phi)$, then $\vol_\rv(\chi)\geq \vol_\rv({\rm IN}(\phi))=\bfE_\rv^\NA(\phi)$. We claim $\MA^\NA_\rv({\rm FS}(\chi)) = \mu$. 
    Indeed, one has
    \begin{align*}
        \bfE_\rv^\NA({\rm FS}(\chi))-\int_{X^\an} {\rm FS}(\chi)\; \mu 
        &\geq \vol_\rv(\chi)-\int_{\Sigma} \phi\; \mu \\
        &\geq \bfE_\rv^\NA(\phi) - \int_{X^\an} \phi \; \mu = \bfE^{*,\NA}_\rv(\mu), 
    \end{align*}
    where the first inequality is due to $\FS(\chi)|_\Sigma\leq \phi$ (see Lemma~\ref{IN_Sigma(FS)}).
    Then ${\rm FS}(\chi)= \phi$ is the solution of the weighted Monge-Amp\`ere equation, and $\chi \in \mathcal{N}_\RR^\Sigma$.
\end{proof}

\subsection{Multiplier ideal approximations}

\begin{definition}
\label{Phi_NA_div}
    Let $\Phi$ be a psh ray of linear growth. The associated non-Archimedean psh function $\Phi^\NA: X^\div\rightarrow \RR\cup \{-\infty\}$ is defined by setting $\Phi^\NA(v) = -\sigma(v)(\Phi)$ for any divisorial valuation $v\in X^\div$.
\end{definition}

\begin{proposition}
    \label{approximated_geodesic}
    Assume $(X,L)$ satisfies the envelope property. For any psh ray $\Phi$ of linear growth, then $\Phi^\NA\in \PSH^\NA(X^\an,L^\an)$ and there exists a decreasing sequence of $\phi_m\in \mcH^{\NA}(X^\an,L^\an)$ that converges to $\Phi^\NA$. 
\end{proposition}
\begin{proof}
Let $\pi:Y\rightarrow X$ be a resolution and denote $\Psi\coloneqq \pi^*\Phi$ for simplicity and $\tilde{\fb}_m\coloneqq \mcJ(m\Psi)$.
Let $L_m\coloneqq \pi^*L+\frac{1}{\sqrt{m}}A$, where $A$ is an ample divisor on $Y$.

We claim: $\oo((m+\sqrt{m})p_1^*L_m)\otimes\tilde{\fb}_m$ is globally generated.

Indeed, take a very ample divisor $H$ on $Y$, then for $m\gg1$, 
\begin{align*}
    H_1\coloneqq \sqrt{m}L_m-K_Y-(n+1)H
    =\sqrt{m}\pi^*L+A-K_Y-(n+1)H
\end{align*}
is ample as long as we take $A$ very positive such that $A-(n+1)H$ is still ample.
By relative Nadel vanishing (see \cite[Theorem~1.7]{Mat22} or \cite[Theorem~B.8]{BFJ16}),
\begin{align*}
    &R^j(p_2)_*(\oo((m+\sqrt{m})p_1^*L_m-jp_1^*H)\otimes\tilde{\fb}_m) \\
    =&R^j(p_2)_*(p_1^*K_Y\otimes\oo(mp_1^*L_m+(n+1-j)p_1^*H+p_1^*H_1)\otimes\tilde{\fb}_m)=0.
\end{align*}
By the relative version of the Castelnuovo--Mumford criterion, then $\oo((m+\sqrt{m})p_1^*L_m)\otimes\tilde{\fb}_m$ is $p_2$-globally generated.

Thus, by Proposition~\ref{equiv_Fubibi-Study_func},   $\tilde\varphi_m\coloneqq \frac{1}{m+\sqrt{m}}\varphi_{\tilde{\fb}_m}\in\mcH^\NA(Y^\an,L_m^\an)$.
By sub-additivity of multiplier ideal sheaf on $Y_\CC$, i.e., $\tilde{\fb}_{2m}\subset \tilde{\fb}_m^2$, one has $\frac{1}{m}\varphi_{\tilde{\fb}_m}\geq\frac{1}{2m}\varphi_{\tilde{\fb}_{2m}}$.
Thus
\begin{align*}
    \tilde\varphi_{2m}
    =\frac{2m}{2m+\sqrt{2m}}\frac{1}{2m}\varphi_{\tilde{\fb}_{2m}}
    \leq \frac{m}{m+\sqrt{m/2}}\frac{1}{m}\varphi_{\tilde{\fb}_m}
    =\frac{m+\sqrt{m}}{m+\sqrt{m/2}}\tilde\varphi_m
    \leq\tilde\varphi_m,
\end{align*}
since $\tilde\varphi_m\leq0$.
By \cite[Lemma 5.7 (ii)]{BBJ21}, 
\begin{align*}
    \frac{m}{m+\sqrt{m}}\Psi^\NA
    \leq\tilde{\varphi}_m
    \leq \frac{m}{m+\sqrt{m}}\Psi^\NA
    +\frac{1}{m+\sqrt{m}}(A_Y+1).
\end{align*}
Thus
$\lim_m\tilde\varphi_m=\Psi^\NA$.

All in all, $\psi_m\coloneqq \tilde\varphi_{2^m}\in\mcH^\NA(Y^\an,L_{2^m}^\an)$ is a decreasing sequence, which converges pointwise to $\Psi^\NA$ on $Y^\div$.
Since $L_{2^m}\rightarrow\pi^*L$, then $\Psi^\NA=\pi^*\Phi^\NA\in\PSH(Y^\an,\pi^*L^\an)$.
Therefore, by Proposition~\ref{env_cor}, we have $\Phi^\NA\in\PSH(X^\an,L^\an)$.
By \cite[Corollary~12.18 (iii)]{BJ22}, there exists a decreasing sequence of $\phi_m\in \mcH^{\NA}(X^\an,L^\an)$ that converges to $\Phi^\NA$.
\end{proof}
\begin{corollary}
Assume $(X,L)$ satisfies the envelope property.
    For any psh ray $\Phi$ of linear growth,  then $\Phi^\NA\in\mcE^1(X^\an,L^\an)$ and
    \begin{align}
    \label{slope_ineq_E}
        \bfE^\NA(\Phi^\NA)\geq\bfE^{\prime\infty}(\Phi)>-\infty.
    \end{align}
\end{corollary}
\begin{proof}
    By Proposition~\ref{approximated_geodesic},  there exists a decreasing sequence $\phi_m\in\mcH^\NA(X^\an,L^\an)$ that converges to $\Phi^\NA$. By \cite[Proposition 2.7]{Ber16}, there exists a unique geodesic ray $\Phi_m$ with $\Phi_m^\NA = \phi_m$, which corresponds to a test configuration $(\mcX_m,\mcL_m)$.  
Since $\Phi^\NA\leq\phi_m$, by \cite[Lemma~4.6]{BBJ21}, there exists a constant $C>0$ such that $\Phi\leq\Phi_m+C$. Thus
$\bfE^{\prime\infty}(\Phi)\leq\bfE^{\prime\infty}(\Phi_m)=\bfE^\NA(\phi_m)$. Taking $m\rightarrow\infty$, we derive \eqref{slope_ineq_E}.
\end{proof}

\subsection{Maximal geodesic rays}

\begin{definition}
    Let $\Phi$ be a psh ray of linear growth emanating from $\varphi\in \mcE^1(X,L)$. Then $\Phi$ is called a \textit{maximal geodesic ray} if for any other psh ray of linear growth $\Psi$  such that $\lim_{s\rightarrow0}\Psi(s)\leq\varphi$ and $\Psi^\NA\leq \Phi^\NA$, we always have $\Psi\leq \Phi$.
\end{definition}

\begin{proposition}
    \label{max_geodesic_existence}
    Let $(X,L)$ be a polarized klt variety satisfying the envelope property. 
    \begin{enumerate}[$(1)$]
        \item For any $\varphi\in \mcE^1(X,L)$, $\phi\in \mcE^{1,\NA}(X^\an,L^\an)$, there exists a unique maximal geodesic ray $\Phi$ emanating from $\varphi$, such that $\phi = \Phi^\NA$.
        \item Let $\Phi$ be a psh ray of linear growth. Then $\Phi$ is a maximal geodesic ray if and only if $\bfE(\Phi(s)) = \bfE(\Phi_0)+s\bfE^\NA(\phi)$, for $t\geq 0$. In particular,
        \begin{align}
        \label{max_geodesic_slope}
            \bfE^{\prime\infty}(\Phi) = \bfE^{\NA}(\phi).
        \end{align}
    \end{enumerate}
\end{proposition}
\begin{proof}
    The proof follows the steps of the proofs of \cite[Theorem 6.6, Corollary 6.7]{BBJ21}.
    Let $\varphi_j$ be a decreasing sequence of psh functions that converge to $\varphi$, such that each $\varphi_j\in\mcH(X,L)$. By Proposition~\ref{approximated_geodesic}, there exists a decreasing sequence $\phi_m\in\mcH^\NA(X^\an,L^\an)$ that converges to $\phi$. By \cite[Proposition 2.7]{Ber16}, there exists a unique maximal geodesic ray $\Phi_j$ emanating from $\varphi_j$, which satisfies $\Phi_j^\NA = \phi_j$.  
By the slope formula for test configurations (see Lemma~\ref{lemma_slope_formula}), %$\Phi_j^{\prime\infty} = \phi_j$. Therefore, 
we have 
    \begin{align}
        \label{slope_test_configuration}
        \bfE(\Phi_j(s)) = \bfE(\varphi_j)+s\bfE^\NA(\phi_j).
    \end{align}
    By the maximum principle, $\Phi_{j+1}\leq \Phi_{j}$. Therefore the decreasing sequence $\Phi_j$ converges to a psh ray $\Phi$.
    Meanwhile, since the two sides of \eqref{slope_test_configuration} converge, we have
    $\bfE(\Phi(s)) = \bfE(\varphi)+s\bfE^\NA(\phi)$.
    This equality implies that $\Phi$ is a geodesic ray. 
    As $\Phi\leq \Phi_j$, $\Phi^\NA\leq \phi$. In addition, by \eqref{slope_ineq_E}, $\bfE^\NA(\Phi^\NA)\geq \bfE^{\prime\infty}(\Phi)=\bfE^\NA(\phi)$. Then we have $\Phi^\NA=\phi$.
    Let $\Psi$ be any other psh ray emanating from $\varphi$ with $\Psi^\NA = \phi$. By the maximum principle, $\Psi \leq \Phi_j$. Therefore, $\Psi\leq \Phi$. This shows that $\Phi$ is a maximal geodesic, which is naturally unique. 

    For (2), the if part has been proved by the argument above. We only need to show the only if part. Let $\Psi$ be the maximal geodesic ray emanating from $\varphi$ with $\Psi^\NA=\phi$. Then $\bfE(\Psi(s)) = \bfE(\varphi)+s\bfE^{\prime\infty}(\Psi) \leq \bfE(\varphi)+ s\bfE^\NA(\phi) \leq \bfE(\Phi(s))$. Together with $\Phi\leq \Psi$, we have $\Phi = \Psi$.
\end{proof}

\begin{proposition}
     Let $(X,L)$ be a klt variety satisfying the envelope property. Let $\Phi$ be a maximal geodesic ray emanating from $\varphi\in\mcE^{1}(X,L)$. Then
     \begin{align}
        \label{max_geodesic_slope_2}
         (\bfE^{Q})^\NA(\phi) = (\bfE^{\ddc\psi_Q})^{' \infty}(\Phi).
     \end{align}
\end{proposition}
\begin{proof}
    The proof basically follows from the proof of \cite[Theorem 4.1]{Li22a}. We will sketch the key points in the following.
 There exists a decreasing sequence $\varphi_m\in\mcH(X,L)$ that converges to $\varphi$. By Proposition~\ref{approximated_geodesic}, there exists a decreasing sequence $\phi_m\in\mcH^\NA(X^\an,L^\an)$ that converges to $\phi$. There exists a unique geodesic ray $\Phi_m$ emanating from $\varphi_m$ with $\Phi_m^\NA = \phi_m$, which corresponds to a test configuration $(\mcX_m,\mcL_m)$. By the maximum principle, $\Phi_{m}\searrow \Phi$. 

    Without the loss of generality, we may assume $0\leq \ddc\psi_Q\leq C\omega$. By the slope formula Lemma~\ref{lemma_slope_formula}, then $(\bfE^{\ddc\psi_Q})^{\prime\infty}(\Phi_m) = (\bfE^Q)^\NA(\phi_m)$. 
    To obtain \eqref{max_geodesic_slope_2}, it suffices to show that
    \begin{align*}
        \lim_{m\rightarrow\infty}(\bfE^{\ddc\psi_Q})^{\prime\infty}(\Phi_m)
        =(\bfE^{\ddc\psi_Q})^{\prime\infty}(\Phi).
    \end{align*}
By the co-cycle condition (see \cite[(100)]{Li22a}) and $\Phi_m\geq\Phi$, we have
    \begin{align*}
        0\leq (\bfE^{\ddc\psi_Q})^{\prime\infty}(\Phi_m)-(\bfE^{\ddc\psi_Q})^{\prime\infty}(\Phi)
        \leq C\Big(\bfE^{\omega})^{\prime\infty}(\Phi_m)-(\bfE^{\omega})^{\prime\infty}(\Phi) \Big).
    \end{align*}
    Then it suffices to show that $\bfE^\omega_{\Phi(s)}(\Phi_m(s)) = \bfE^\omega(\Phi_m(s))-\bfE^\omega(\Phi(s))\leq \delta_ms$, with $\delta_m\to 0$ as $m\to\infty$. Since $\Phi,\Phi_m$ are maximal geodesic rays, 
    \begin{align*}
        \bfE_{\Phi(s)}(\Phi_m(s)) = \bfE(\Phi_m(s))-\bfE(\Phi(s)) = (\bfE^{\prime\infty}(\phi_m)-\bfE^{\prime\infty}(\phi))s \leq \epsilon_m s,
    \end{align*}
    where $\epsilon_m\to 0$ as $m\to\infty$ by Proposition~\ref{approximated_geodesic}. By \cite[(127)--(131)]{Li22a}, which is based on \cite[Appendix A]{BBJ21}, we have
    \begin{align*}
        \bfE^\omega_{\Phi(s)}(\Phi_m(s)) \leq C \epsilon_m^{\frac{1}{2^n}}s.
    \end{align*}
    This concludes the proof.
\end{proof}

It is straightforward to generalize \eqref{max_geodesic_slope} and \eqref{max_geodesic_slope_2} to the equivariant case. 
\begin{corollary}
    \label{geodesic_approx_equiv}
    Let $(X,L)$ be a klt variety satisfying the envelope property. Let $\Phi$ be a maximal geodesic ray in $\mcE_T^1(X,L)$ and $\phi=\Phi^\NA$. Then
    \begin{align}
         \bfE_{\rm v}^\NA(\phi) &= \bfE_{\rm v}^{\prime\infty}(\Phi), \\
         (\bfE_{\rm v}^{Q})^\NA(\phi) &= (\bfE_{\rm v}^{\ddc\psi_Q})^{\prime\infty}(\Phi).
    \end{align}
\end{corollary}
%\begin{proof}
%    By Proposition \ref{approximated_geodesic}, there exists a decreasing sequence $\phi_m\in\mcH(X^\an,L^\an)$ converging to $\Phi^\NA$.
%    Then $\bfE_\rv^\NA(\phi_m)\rightarrow\bfE^\NA(\Phi^\NA)$.
%    There exists a unique maximal geodesic ray $\Phi_m$ emanating from $u_m\in\mcH(X,L)$, which decreasingly converges to $\Phi(0)$.
%    By the maximum principle, $\Phi_{m+1}\leq \Phi_{m}$, the decreasing sequence $\Phi_m$ converges to $\Phi$.
%    Thus, $\bfE_\rv^{\prime\infty}(\Phi_m)\rightarrow\bfE^{\prime\infty}_\rv(\Phi)$.
    
%\end{proof}
\begin{remark}
    \label{geodesic_approx_equiv_remark}
    By the same argument, for the smooth resolution $\pi:Y\rightarrow X$, we also have
    \begin{align}
        \bfE_{\rm v}^\NA(\pi^*\phi) = \bfE_{\rm v}^{\prime\infty}(\pi^*\Phi), \;
         (\bfE_{\rm v}^{Q})^\NA(\pi^*\phi) = (\bfE_{\rm v}^{\ddc\psi_Q})^{' \infty}(\pi^*\Phi).
    \end{align}
\end{remark}
\begin{corollary}
\label{slope_reduced_J_v}
Suppose that $(X,L)$ satisfies the envelope property.
    Let $\Phi$ be a maximal geodesic ray in $\mcE_T^1(X,L)$ and $\phi=\Phi^\NA$. Then
    \begin{align}
        \bfJ_{\rv,\TT}^\NA(\phi)
        =\bfJ_{\rv,\TT}^{\prime\infty}(\Phi)
        \coloneqq  \inf_{\xi\in N_\RR}\bfJ_\rv^{\prime\infty}(\Phi_\xi),
    \end{align}
    where $\Phi_\xi=(\sigma_\xi(s)^*\Phi(s))$.
\end{corollary}
\begin{proof}
     By Proposition~\ref{approximated_geodesic}, there exists a decreasing sequence $\phi_m\in\mcH(X^\an,L^\an)$ converging to $\Phi^\NA$.
     There exists a unique maximal geodesic ray $\Phi_m$ emanating from $\varphi_m\in\mcH(X,L)$, which decreasingly converges to $\varphi_0$.
     By the maximum principle, $\Phi_{m}\searrow \Phi$.
     By Lemma~\ref{reduced_J_v_convergence} and \cite[Proposition 5.15]{HL23}, it suffices to show that
     \begin{align}
     \label{converg_slope_reduced_J_v}
         \lim_{m\rightarrow\infty}
         \bfJ_{\rv,\TT}^{\prime\infty}(\Phi_j)
         =\bfJ_{\rv,\TT}^{\prime\infty}(\Phi).
     \end{align}
     Note that $\Phi_{j,\xi}$ decreasingly converges to $\Phi_\xi$, then $\lim_{j\rightarrow\infty}\bfJ_\rv^{\prime\infty}(\Phi_{j,\xi})=\bfJ_\rv^{\prime\infty}(\Phi_{\xi})$, which implies that $\limsup_{j\rightarrow\infty}\bfJ_{\rv,\TT}^{\prime\infty}(\Phi_j)
         \leq\bfJ_{\rv,\TT}^{\prime\infty}(\Phi)$.
    To show the converse direction, first, claim: there exists $C>0$ such that
    \begin{align*}
        \bfJ_{\rv,\TT}^{\prime\infty}(\Phi_j)
         =\inf_{|\xi|\leq C}\bfJ_\rv^{\prime\infty}(\Phi_{j,\xi}), \quad \bfJ_{\rv,\TT}^{\prime\infty}(\Phi)
         =\inf_{|\xi|\leq C}\bfJ_\rv^{\prime\infty}(\Phi_\xi).
    \end{align*}
    The slope formula and the quasi-triangle inequality for  $\bfI_\rv$ (see \cite[Lemma 2.7]{HL23}) implies
    \begin{align*}
        \bfI_\rv^{\prime\infty}(\Phi_\triv,\Phi_{\triv,\xi})
        \leq C_{n,\rv}(\bfI_\rv^{\prime\infty}         (\Phi_\triv,\Phi_{j,\xi})+\bfI_\rv^{\prime\infty}     (\Phi_{j,\xi},\Phi_{\triv,\xi})).
    \end{align*} 
Hence $\xi$ is bounded if $\bfJ_\rv^{\prime\infty}(\Phi_{j,\xi})$ is bounded.

   By the slope formula and Lemma~\ref{arch_J_v-relation}, on $\{|\xi|\leq C\}$,  we have 
    \begin{align*}
        |\bfJ_{\rv}^{\prime\infty}(\Phi_{j,\xi})-\bfJ_{\rv}^{\prime\infty}(\Phi_{\xi})|
        \leq C_n \bfI_\rv^{\prime\infty}(\Phi_{j,\xi},\Phi_\xi)^{\frac{1}{2}}
    \end{align*}
    
    The functions $\xi\mapsto\bfJ_{\rv}^{\prime\infty}(\Phi_{j,\xi})$ converges to $\xi\mapsto\bfJ_{\rv}^{\prime\infty}(\Phi_{\xi})$ uniformly over $\{|\xi|\leq C\}$. Hence we obtain \eqref{converg_slope_reduced_J_v}.
\end{proof}

\section{Weighted K-stability for models}
\label{w-K-stability}

In this section, we first prove that $\GG$-uniform weighted K-stability for models is invariant under birational resolution. Using this result, and adapting Chi Li's argument to the resolution, we then prove Theorem \ref{theorem_A}.

\subsection{Preliminaries}
Let $(X,L)$ be a polarized klt variety and the smooth projective variety $Y$ be a resolution of $X$.
For a model $(\mcY,\mcL')$, we recall the notation $\phi_{\mcL'} \coloneqq  {\rm P}_{\pi^*L}(\varphi_{\mcL'})$ which will be used in this subsection. And we will always assume $\mcY$ is an SNC model. Note that $\pi^*L$ is semi-ample and big, and $\MA_\rv^\NA$, $\bfE^\NA_\rv$, $\bfJ^\NA_\rv$, $(\bfE^\mathcal{Q})^\NA_\rv$, $\bfH^\NA_\rv$ can be defined over $\PSH(Y^\an,\pi^*L^\an)$ in the same way as in Section \ref{w-NA-MA}.

\begin{comment}
\begin{definition}
   %Let $X$ be a projective variety with klt singularities. 
   Let $M$ be a semi-ample and big $\QQ$-line bundle on $X$. For any continuous function $\phi\in \PSH(X^\an,M^\an)$, the non-Archimedean Mabuchi function is defined as
    \begin{align*}
        \bfM^\NA(\phi) = \bfH^\NA(\phi) + \bfR^\NA(\phi) + \underline{S} \; \bfE^\NA(\phi),
    \end{align*}
    where
    \begin{align*}
        \bfH_\rv^\NA(\phi) &= \int_{X^\an} A_X(x) \; \MA_\rv^\NA(\phi), \\
        \bfR^\NA(\phi) &= (\bfE^{K_X})^\NA = \sum_{i=0}^{n-1} \int_{X^\an} (\phi - \phi_{\triv}) \; \ddc \psi \wedge \MA^\NA(\phi_\triv^{[i]}, \phi^{[n-1-i]}), \\
        \bfE^\NA(\phi) &= \frac{1}{n+1}\sum_{i=0}^{n+1}\int_{X^\an} (\phi - \phi_\triv) \; \MA^\NA(\phi_\triv^{[i]}, \phi^{[n-i]}),
    \end{align*}
    where $\psi$ is a Hermitian metric on $K^\an_X$. We also define
    \begin{align*}
        \bfJ^\NA(\phi) = L^n \cdot \sup(\phi-\phi_\triv) - \bfE^\NA(\phi).
    \end{align*}

\end{definition}
\end{comment}

For any $\varphi_D\in\PL(Y^\an)_T$ and $\phi_\mcL\in\mcH_T(Y^\an,\pi^*L^\an)$, we may assume that $D$ and $\mcL$ are in a same model $\mcY$. Let $\mcY_0 = \sum_{E} b_E\cdot E$. One has
    \begin{align}
        \int_{Y^\an}\varphi_D\MA_\rv^\NA(\phi_\mcL)     
        %=&\sum_{\vec{k}}a_{\vec{k}}\frac{\vec{k}!}{(n+k)!}\int_{X^\an}\varphi_D\, (p_{\vec{k}})_*(\ddc\phi_{\mcL^{[\vec{k}]}})^{n+k}  \nonumber  \\
       % =&\sum_{\vec{k}}a_{\vec{k}}\frac{\vec{k}!}{(n+k)!} \int_{(X^{[\vec{k}]})^\an}\varphi_{D^{[\vec{k}]}}(\ddc\phi_{\mcL^{[\vec{k}]}})^{n+k} \nonumber  \\
        =&\lim \sum_{\vec{k}}a_{\vec{k}}\frac{k_1!\cdots k_r!}{(n+k)!}   D^{[\vec{k}]}\cdot(\mcL^{[\vec{k}]})^{n+k} \nonumber \\
        =& \lim \sum_{\vec{k}}a_{\vec{k}}\frac{k_1!\cdots k_r!}{(n+k)!}  \sum_E \ord_{E^{[\vec{k}]}}(D^{[\vec{k}]})E^{[\vec{k}]}\cdot(\mcL^{[\vec{k}]})^{n+k} \nonumber \\
        =& \lim \sum_{\vec{k}}a_{\vec{k}}\frac{k_1!\cdots k_r!}{(n+k)!} \sum_E b_E\; \varphi_D(v_E) \; E^{[\vec{k}]}\cdot(\mcL^{[\vec{k}]})^{n+k}\nonumber \\
        =&\sum_{E}b_E \; \varphi_D(v_E) \; (E\cdot\mcL^n)_{\rv},
        \label{}
    \end{align}
    where   
    \begin{align}
        (E\cdot\mcL^{n})_\rv
        %\coloneqq \sum_{\vec{k}}a_{\vec{k}}\frac{\vec{k}!}{n!}(E\cdot\mcL^n)_{\vec{k}}
        = \lim \sum_{\vec{k}}a_{\vec{k}}\frac{k_1!\cdots k_r!}{(n+k)!} E^{[\vec{k}]}\cdot(\mcL^{[\vec{k}]})^{n+k}. 
    \end{align}
    Hence we have
    \begin{align}
    \label{MA_v^NA(phi_mcL)smooth_pi^*L}
        \MA_\rv^\NA(\phi_\mcL)=\sum_E b_E(E\cdot\mcL^{n})_\rv\delta_{v_E}.
    \end{align}

%\begin{definition}
%    We say a polarized klt variety $(X,L)$ is K-stable for models if there exists $\gamma>0$, such that for any normal model $(\mcX, \mcL)$ of $(X,L)$, we have
%    \begin{align*}
%        \bfM^\NA(\phi_\mcL) \geq \gamma \cdot \bfJ^\NA(\phi_\mcL).
%    \end{align*}
%\end{definition}

For a big $\QQ$-line bundle $\mcL'$ on $\mcY$, let $\tilde{\fa}_m$ be the base ideal of $|m \mcL'|$, $\phi_m = \frac{1}{m}\varphi_{\tilde{\fa}_m} + \varphi_{\mcL'}$.
There exists a test configuration $(\mcY_m,\mcL'_m)$ and birational morphism $\mu_m: \mcY_m\rightarrow \mcY$, such that $\phi_m = \phi_{\mcL'_m}$.
\begin{proposition}
\label{functionals_as_positive_product}
    Let $(\mcY,\mcL')$ be a $T_\CC$-equivariant big SNC model of $(Y,\pi^*L)$.
    Let $\mcY_0 = \sum_{1\leq i\leq r'} b'_i E'_i$,  $v'_i = \frac{1}{b'_i} {\rm ord}_{E'_i}$. Then
        \begin{align}
             \MA_\rv^\NA(\phi_{\mcL'}) &= \sum_{1\leq i\leq r'} b'_i \cdot \langle {\mcL'}^n \rangle_\rv \cdot E'_i \cdot \delta_{v'_i}, \nonumber \\
            \bfH_\rv^\NA(\phi_{\mcL'}) &= \langle {\mcL'}^n \rangle_\rv \cdot K^{\log}_{\mcY/Y_{\PP^1}}, \nonumber \\
            \bfR_\rv^\NA(\phi_{\mcL'}) &= \langle {\mcL'}^n \rangle_\rv \cdot \rho^* K_{Y_\PP^1}, \\
            \bfE_\rv^\NA(\phi_{\mcL'}) &= \langle {\mcL'}^{n+1} \rangle_\rv =  \langle {\mcL'}^n \rangle_\rv \cdot \mcL', \nonumber
        \end{align}
        where we define
        \begin{align}
            \label{big_m_L}
            \langle\mcL'^n\rangle_\rv\cdot D    \coloneqq \sup_{m\to \infty} (\mu_m^* D\cdot \mcL_m'^n)_\rv =  \lim_{m\to \infty} (\mu_m^* D\cdot \mcL_m'^n)_\rv.
        \end{align}
\end{proposition}
\begin{remark}
\label{exist_lim}
    Since $\rv$ is a bounded positive function, there exists a constant $C>0$ independent of $m$ such that $(\mu_m^* D\cdot \mcL_m'^n)_\rv \leq C (\mu_m^* D\cdot \mcL_m'^n)$. Since $\sup_{m\to \infty} (\mu_m^* D\cdot \mcL_m'^n)<+\infty$ by \cite[Proposition 3.6]{DEL00}, the limit in \eqref{big_m_L} exists.
\end{remark}

\begin{proof}
 Denote $\tilde{E}'_{i,m}$ as the strict transform of $E_i'$ in $\mu:\mcY_m\rightarrow \mcY$.
By modifying Proposition~\ref{env_cor} to the weighted case, $\MA_\rv^\NA(\phi_{\mcL'_m})$ converges to $\MA_\rv^\NA(\phi_{\mcL'})$ weakly as Radon measures on $Y^\an$. In particular,
\begin{align*}
    \int_{Y^\an} \MA_\rv^\NA(\phi_{\mcL'}) = \VV_\rv.
\end{align*}

The SNC model $\mcY$ induces a retraction $r_\mcY: Y^\an\rightarrow \Delta_\mcY$. 
Let $\nu_\mcY \coloneqq  (r_\mcY)_*\MA_\rv^\NA(\phi_{\mcL'})$ and $\nu_{\mcY, m}\coloneqq (r_\mcY)_*\MA_\rv^\NA(\phi_{\mcL'_m})$.
Since $v'_i$ is a point in $\Delta_\mcY$, we have
\begin{align*}
    \nu_\mcY(\{v'_i\})=\MA_\rv^\NA(\phi_{\mcL'})(r_\mcY^{-1}\{v'_i\}) 
    \geq \MA_\rv^\NA(\phi_{\mcL'})(\{v'_i\}), \;
    \nu_{\mcY, m}(\{v'_i\}) \geq \MA_\rv^\NA(\phi_{\mcL'_m})(\{v'_i\}).
\end{align*}
By Portmanteau's theorem for weak convergence of measures, we have
\begin{align*}
    \limsup_{m\to\infty} \nu_{\mcY, m}(\{v'_i\}) \leq \nu_\mcY(\{v'_i\}) =: V_i.
\end{align*}
\begin{enumerate}[(i)]
\item If $E'_i\not\subset \mathbb{B}_+(\mcL')$, since $(\mcY,\mcL')$ is $T_\CC$-equivariant, $E_i'^{[\vec{k}]}\not\subset \mathbb{B}_+(\mcL'^{[\vec{k}]})$. Hence we have 
\begin{align*}
    \lim_{m\rightarrow\infty}({\mcL'_m}^n \cdot \tilde{E}'_{i,m})_\rv
    =&\lim_m\lim\sum_{\vec{k}}a_{\vec{k}}\frac{k_1!\cdots k_r!}{(n+k)!} (\mcL_m'^{[\vec{k}]})^{n+k}\cdot (\tilde{E}_{i,m}')^{[\vec{k}]} \\
    =& \lim_m\lim\sum_{\vec{k}}a_{\vec{k}}\frac{k_1!\cdots k_r!}{(n+k)!} (\mcL_m'^{[\vec{k}]})^{n+k}\cdot (\mu_m^*E_{i}')^{[\vec{k}]}  =\langle\mcL'^n\rangle_\rv\cdot E_i',
\end{align*}
where the second equality is by \cite[Theorem 2.7]{Li23}.

It follows that
\begin{align*}
    b'_i \langle {\mcL'}^n \rangle_\rv \cdot E'_i    
        = \lim_{m\to\infty} b'_i ({\mcL'_m}^n \cdot \tilde{E}'_{i,m})_\rv 
        &=  \lim_{m\to\infty} \MA_\rv^\NA(\phi_m)(\{v'_i\}) \\
        &\leq  \lim_{m\to\infty} \nu_{\mcY,m}(\{v'_i\}) \leq  \nu_\mcY(\{v'_i\}) = V_i.
\end{align*}

\item\label{E_i_subset_B_+} 
If $E'_i\subset \mathbb{B}_+(\mcL')$, then $E_i'^{[\vec{k}]}\subset\mathbb{B}_+(\mcL'^{[\vec{k}]})$. then 
\begin{align*}
&\vol_{\mcY^{\vec{k}}|E_i'^{[\vec{k}]}}(\mcL'^{[\vec{k}]}) \\
=&     \lim_{r\to\infty}\frac{\dim\big( {\rm Im}\big( H^0(\mcY^{[\vec{k}]},\oo_\mcY(r\mcL'^{[\vec{k}]}))
    \rightarrow H^0(E_i'^{[\vec{k}]}, \oo_{E_i'^{[\vec{k}]}}(r\mcL'^{[\vec{k}]}))\big)}{r^{(n+k)}/(n+k)!} = 0.
\end{align*}
Since the restricted volume equals to zero, by \cite[Theorem 2.7]{Li23}, $ \langle\mcL'^{[\vec{k}]}\rangle^{n+k}\cdot (E_{i}')^{[\vec{k}]} = 0$. Then by the definition of positive product intersection,
\begin{align*}
    0\leq  (\mcL_m'^{[\vec{k}]})^{n+k}\cdot (\mu_m^*E_{i}')^{[\vec{k}]} \leq \langle\mcL'^{[\vec{k}]}\rangle^{n+k}\cdot (E_{i}')^{[\vec{k}]} = 0.
\end{align*}
This implies that
$\langle\mcL'^n\rangle_\rv\cdot E_i'=0$. Therefore
\begin{align*}
       0 =b_i'\langle\mcL'^n\rangle_\rv\cdot E_i'= \lim_{m\to\infty} b'_i {(\mcL'_m}^n \cdot \tilde{E}'_{i,m})_\rv
       = \lim_{m\to\infty} \MA_\rv^\NA(\phi_m)(\{v'_i\}) \leq V_i.
\end{align*}
\end{enumerate}
Combining these, we obtain
\begin{align*}
    \VV_\rv = \sum_i b'_i \langle {\mcL'}^n \rangle_\rv \cdot E'_i    
        = \sum_i \lim_{m\to\infty} b'_i ({\mcL'_m}^n \cdot \tilde{E}'_i)_\rv 
        =& \sum_i \lim_{m\to\infty} \MA_\rv^\NA(\phi_m)(\{v'_i\}) \\
        \leq& \sum_i \nu_\mcY(\{v'_i\}) =\VV_\rv.
\end{align*}
Therefore, $\MA_\rv^\NA(\phi_{\mcL'}\circ r_\mcY)(\{v'_i\}) = V_i=b'_i \langle {\mcL'}^n \rangle_\rv \cdot E'_i$ for $1\leq i\leq r'$.
And this equality holds for any blowups of $\mcY$.
And since $Y$ is smooth, 
\begin{align*}
Y^\an = \lim_{\leftarrow}\Delta_\mcY    
\end{align*}
over SNC models $\mcY$, then $\MA_\rv^\NA(\phi_{\mcL'}\circ r_\mcY) = \sum_{1\leq i\leq r'} b'_i \cdot \langle {\mcL'}^n \rangle _\rv\cdot E'_i \cdot \delta_{v'_i}$ holds for all higher model $\mcY$. Therefore, we have
\begin{align*}
    \MA_\rv^\NA(\phi_{\mcL'}) &= \sum_{1\leq i\leq r'} b'_i \cdot \langle {\mcL'}^n \rangle_\rv \cdot E'_i \cdot \delta_{v'_i}.
\end{align*}

Since $\phi_m$ converges to $\phi_{\mcL'}$ increasingly, we have
\begin{align*}
    \bfE_\rv^\NA(\phi_{\mcL'}) &= \lim_{m\to\infty} \bfE_\rv^\NA(\phi_m) 
    = \lim_{m\to\infty}\left(\lim\sum_{\vec{k}}a_{\vec{k}} \frac{k_1!\cdots k_r! }{(n+k+1)!}  (\mcL_m'^{[\vec{k}] })^{n+1} \right)
    = \langle \mcL'^{n+1}\rangle_\rv ,
\end{align*}
where the last equality above is due to Remark \ref{exist_lim}.

The proof for $\bfR_\rv^\NA(\phi_{\mcL'})$ follows similarly. So we omit the proof.

For the entropy,  note that
\begin{align*}
    K^{\log}_{\mcY/Y_{\PP^1}} &= K_\mcY + \mcY_0^{\rm red} - (K_{Y_{\PP^1}}+\mcY_0) = \sum_i (A_{Y_{\PP^1}}-b_i') E_i' \\
    &= \sum_i b'_i (A_{Y_{\PP^1}}(v_i') - 1) E'_i = \sum_i b'_i A_Y(v'_i) E'_i.
\end{align*}
Hence, we have
\begin{align*}
   \bfH^\NA(\phi_{\mcL'}) 
   = \int_{Y^\an} A_Y(y)\; \MA_\rv^\NA(\phi_{\mcL'}) 
   =& \sum_i A_Y(v'_i)\cdot b'_i \cdot \langle \mcL'^n \rangle _\rv\cdot E'_i \\
  =& \langle \mcL'^n \rangle_\rv \cdot K^{\log}_{\mcY/Y_{\PP^1}}.
\end{align*}

\end{proof}

\begin{proposition}
\label{M_v,w_equal}
Suppose that $(X,L)$ satisfies the envelope property.
    Let $(\mcY,\mcL')$ be a big model of $(Y,\pi^*L)$. Then there exists a model $(\mcX,\mcL)$ of $(X,L)$ such that $\phi_{\mcL'} = \pi^* \phi_\mcL$, and 
    \begin{align*}
        \bfM_{\rv,\rw}^\NA(\phi_\mcL) = \bfM_{\rv,\rw}^\NA(\phi_{\mcL'}),\;
        \bfJ_{\rv,\TT}^\NA(\phi_\mcL) = \bfJ_{\rv,\TT}^\NA(\phi_{\mcL'}).
    \end{align*}
\end{proposition}
\begin{proof}
    By the envelope property and Proposition~\ref{env_cor}, there exists a model $(\mcX,\mcL)$ of $(X,L)$, such that $\phi_{\mcL'} = \pi^*\phi_{\mcL}$. 
    Let $\pi: (\mcY,\pi^*\mcL)\rightarrow (\mcX,\mcL)$ be the birational morphism induced by the pullback. Then $\phi_{\mcL'} = \phi_{\pi^*\mcL} = \pi^*\phi_{\mcL}$ and $\phi_{\mcL',\xi}=\pi^*\phi_{\mcL,\xi}$ for any $\xi\in N_\RR$. Without the loss of generality, we may assume $\mcY$ is a SNC model.
    
    Without the loss of generality, we can add $c\cdot\mcX_0$ to $\mcL$ for some $c\geq 0$ (which is equivalent to add a constant $c$ to $\phi_\mcL$), such that $(\mcX,\mcL)$ is a big model. Then $(\mcY,\pi^*\mcL)$ is also a big model.

    By Proposition~\ref{functionals_as_positive_product}, we have
    \begin{align*}
        \bfE_\rv^\NA(\phi_\mcL) =\langle\mcL^{n+1}\rangle_\rv
        = \langle \pi^*\mcL^{n+1}\rangle_\rv
        &= \bfE_\rv^\NA(\phi_{\mcL'}).
    \end{align*}
    Similarly, $\bfE_\rw^\NA(\phi_\mcL) = \bfE^\NA_\rw(\phi_{\mcL'})$.
    In addition, by definition, $\sup_{X^\an}\phi_\mcL = \sup_{Y^\an}\phi_{\pi^*\mcL}$. Therefore
    \begin{align*}
        \bfJ_\rv^\NA(\phi_\mcL) 
        = \VV_\rv\cdot \sup_{X^\an}(\phi_\mcL-\phi_\triv) 
          - \bfE_\rv^\NA(\phi_\mcL) =  \bfJ_\rv^\NA(\phi_{\mcL'}),
    \end{align*}
and
    \begin{align*}
        \bfJ^\NA_{\rv,\TT}(\phi_\mcL) 
        = \inf_{\xi\in N_\RR} \bfJ^\NA_\rv(\phi_{\mcL,\xi})
        =\inf_{\xi\in N_\RR} \bfJ^\NA_\rv(\pi^*\phi_{\mcL,\xi})
        =\inf_{\xi\in N_\RR} \bfJ^\NA_\rv(\phi_{\mcL',\xi})
        = \bfJ^\NA_{\rv,\TT}(\phi_\mcL'). 
    \end{align*}
    Let $\mcX_0 = \sum_{1\leq i\leq r} b_i E_i$, $\mcY_0 = \sum_{1\leq i\leq r} b'_i E'_i + \sum_{r+1\leq i\leq l} b'_i E'_i$, where for $1\leq j \leq r$, $E'_i$ is the strict transform of $E_i$, $b'_i = b_i$; for $r+1\leq i\leq l$, $E'_i$ is an exceptional divisor of $\pi: \mcY\rightarrow \mcX$. Let $v_i = \frac{1}{b_i}{\rm ord}_{E_i}$, $v'_i = \frac{1}{b'_i}{\rm ord}_{E'_i}$.
     By Proposition~\ref{functionals_as_positive_product}, we have 
     \begin{align*}
         \MA^\NA(\pi^*\phi_\mcL) 
         &= \sum_{1\leq i\leq r} \big( b_i \cdot \langle \pi^*\mcL^n \rangle_\rv \cdot E'_i \big) \cdot \delta_{v'_i} 
          + \sum_{r+1\leq i\leq l} \big( b'_i \cdot \langle \pi^*\mcL^n \rangle_\rv \cdot E'_i \big) \cdot \delta_{v''_i}\\
         &= \sum_{1\leq i\leq r} \big( b_i \cdot \langle \pi^*\mcL^n \rangle_\rv \cdot E'_i \big) \cdot \delta_{v'_i}.
     \end{align*}
     By \cite[Proposition 2.3]{BBG13}, for $r+1\leq i\leq l$, we have $E'_i\in {\rm Ex}(\pi)\subset \mathbb{B}_+(\pi^*\mcL)$. Then $E_i'^{[\vec{k}]}\subset\mathbb{B}_+((\pi^*\mcL)^{[\vec{k}]})$ for each $\vec{k}$. 
     By \eqref{E_i_subset_B_+} in the proof of Proposition~\ref{functionals_as_positive_product}, then $\langle \pi^*\mcL^n \rangle_\rv \cdot E'_j = 0$.
     This implies the second equality in the formula above. Recall that
    \begin{align*}
        & \sum_{1\leq i\leq l} A_Y(v'_i)\; b'_i\; E'_i - \sum_{1\leq i\leq l} \pi^*A_X(v'_i)\; b'_i\; E'_i \\
        &= \Big( K_\mcY - (\mcY_0 - \mcY_{0, {\rm red}}) - \pi^*(K_\mcX - (\mcX_0-\mcX_{0,{\rm red}})) \Big) - \Big( \rho^* K_{Y_{\PP^1}} - \rho^*\pi^* K_{X_{\PP^1}} \Big) \\
        &= \sum_{r+1\leq i\leq l} c_i E'_i - \rho^* K_{(Y/X)_{\PP^1}}.
    \end{align*}
%    And we denote $f(v) \coloneqq  \sum_{r+1\leq i\leq l} c_i v(s_{E'_i})$. And $\rho^*K_{(Y/X)_{\PP^1}}$ corresponds to the function $\psi_Y-\psi_X$. 
     Then we have
     \begin{align*}
         & \bfH_\rv^\NA(\phi_{\mcL'}) - \bfH_\rv^\NA(\phi_\mcL) \\
         &= \int_{Y^\an} A_Y(y) \; \MA_\rv^\NA(\phi_{\mcL'})
              - \int_{X^\an} A_X(x) \; \MA_\rv^\NA(\phi_\mcL) \\
         &= \int_{Y^\an} A_Y(y) \; \MA_\rv^\NA(\pi^*\phi_\mcL)
             - \int_{Y^\an} \pi^* A_X(y) \; \MA_\rv^\NA(\pi^*\phi_\mcL)\\
             &=\sum_{1\leq i\leq r}b_i\langle (\pi^*\mcL)^n\rangle_\rv\cdot E_i'(A_Y(E_i')-\pi^*A_X(E_i')) \\
        % &= \int_{Y^\an} f(y) \; \MA_\rv^\NA(\pi^*\phi_\mcL) 
           %   - \int_{Y^\an}( \psi_Y(y)-\psi_X(y)) \; \MA_\rv^\NA(\pi^*\phi_\mcL)\\
         &= \sum_{r+1\leq  i\leq l}  c_i\cdot \langle (\pi^*\mcL)^n\rangle_\rv\cdot E_i' 
         -\langle (\pi^*\mcL)^n \rangle _\rv\cdot\rho^*K_{(Y/X)_{\PP^1}}\\
         &=-\langle (\pi^*\mcL)^n \rangle _\rv\cdot\rho^*K_{Y_{\PP^1}}
           +\langle (\pi^*\mcL)^n \rangle _\rv\cdot\rho^*\pi^*K_{X_{\PP^1}}  \\
         %&= -\int_{Y^\an} \psi_Y(y) \; \MA^\NA(\pi^*\phi_\mcL) + \int_{Y^\an}\psi_X(y) \; \MA^\NA(\pi^*\phi_\mcL) \\
         &= -\bfR_\rv^\NA(\phi_{\pi^*\mcL}) + \bfR_\rv^\NA(\phi_\mcL).
     \end{align*}
     Hence,
     \begin{align}
        \label{H_X_H_Y}
         \bfH_\rv^\NA(\phi_{\mcL'}) + \bfR_\rv^\NA(\phi_{\mcL'}) 
         = \bfH_\rv^\NA(\phi_\mcL) + \bfR_\rv^\NA(\phi_\mcL).
     \end{align}
     %And $\underline{S} = \frac{K_X\cdot L^{n-1}}{L^n} = \frac{K_Y\cdot \pi^*L^{n-1}}{\pi^*L^{n}}$. 
     Therefore,
     \begin{align*}
         \bfM_{\rv,\rw}^\NA(\phi_\mcL) &
         = \bfH_\rv^\NA(\phi_\mcL) + \bfR_\rv^\NA(\phi_\mcL) 
           +c_{(\rv,\rw)}  \bfE_{\rw}^\NA(\phi_\mcL) \\
         &= \bfH_\rv^\NA(\phi_{\mcL'}) + \bfR_\rv^\NA(\phi_{\mcL'})  
          +c_{(\rv,\rw)}  \bfE_{\rw}^\NA(\phi_{\mcL'}) 
         = \bfM_{\rv,\rw}^\NA(\phi_{\mcL'}).
     \end{align*}
\end{proof}

\begin{corollary}
\label{birational_uKs}
Suppose that $(X,L)$ satisfies the envelope property.
    If $(X,L)$ is $\bG$-uniformly weighted K-stable for models. Then there exists $\gamma >0$, such that for any $\bG$-equivariant big model $(\mcY,\mcL')$ of $(Y,\pi^*L)$, we have
    \begin{align*}
        \bfM_{\rv,\rw}^\NA(\phi_{\mcL'}) \geq \gamma \cdot \bfJ_{\rv,\TT}^\NA(\phi_{\mcL'}).
    \end{align*}
\end{corollary}

\subsection{Uniform weighted K-stability for models}

\begin{theorem}
    Let $(X,L)$ be a polarized klt variety satisfying the envelope property. Assume $(X,L)$ is $\bG$-uniformly weighted K-stable for models. Then the weighted Mabuchi functional $\bfM_{\rv,\rw}$ is $\bG$-coercive over $\mcE^1_{K}$.
\end{theorem}
\begin{proof}
    The proof basically follows the argument of Chi Li for the cscK case. We will sketch Li's argument and emphasize the necessary modifications under the weighted and singular setting.

    {\bf Step 1: } Show that, for a geodesic ray $\Phi$, if $\bfM^{\prime\infty}_{\rv,\rw}(\Phi)<\infty$, then $\Phi$ is a maximal geodesic ray.

\begin{lemma}
\label{integrabiliy_Phihat-Phi}
    Fix $\varphi\in \mcE^1(X,L)$, $\phi\in\mcE^{1}(X^\an,L^\an)$. Let $\Phi$ be a geodesic ray, $\hat{\Phi}$ be a maximal geodesic ray, such that $\Phi(0)=\hat{\Phi}(0)=\varphi$, and $\Phi^\NA = \hat{\Phi}^\NA = \phi$. Then for any $\alpha>0$,
\begin{align}
    \int_{X\times \mathbb{D}} e^{\alpha(\hat{\Phi}-\Phi)} \; \Omega\wedge \sqrt{-1} dt \wedge \bar{dt} < \infty.
\end{align}
\end{lemma}
\begin{proof}
The proof is based on \cite[Section~4.2]{LTW22}.
Let $A$ be an ample line bundle on $Y$ and $\varphi_A\in c_1(A)$ be a smooth metric. 
	Then one sets
	\begin{align*}
		\Phi_\epsilon\coloneqq \pi^*\Phi+\epsilon q_1^*\varphi_A, \text{ and }
		\hat{\Phi}_\epsilon\coloneqq \pi^*\hat\Phi+\epsilon q_1^*\varphi_A
	\end{align*}
	where $\pi:Y_\CC\rightarrow X_\CC$ and $q_1:Y_\CC\rightarrow Y$ is the projection.
	Then $\Phi_\epsilon$ and $\hat{\Phi}_\epsilon$ are singular metrics on $q_1^*L_\epsilon$ over $Y_\CC$, where $L_\epsilon\coloneqq \pi^*L+\epsilon A$. 
	Note that 
	\begin{align*}
		\mcJ(m\Phi_\epsilon)=\mcJ(m\pi^*\Phi)=\mcJ(m\pi^*\hat\Phi)=\mcJ(m\hat\Phi_\epsilon).
	\end{align*}
	By the relative Castelnuovo–Mumford criterion, $\oo_{Y_\CC}(q_1^*((m+m_0)L_\epsilon)\otimes\mcJ(m\Phi_\epsilon))$ is $q_2$-globally generated for $m$ sufficiently large.
    By the multiplier ideal approximation of \cite{BBJ21}, we have a sequence of geodesic rays $\Phi_{\epsilon,m}$ associated to test configurations $(\mcY_{\epsilon,m},\mcL_{\epsilon,m})$.
	Locally the singularity of $\Phi_{\epsilon,m}$ is comparable to $\frac{1}{m+m_0}\log\Sigma_i|f_i|^2$, where $\{f_i\}$ are generators of $\mcJ(m\Phi_\epsilon)=\mcJ(m\Phi)$.
	By Demailly's regularization theorem which depends on the Ohsawa--Takegoshi extension theorem,  there exists $C = C_{\epsilon,m}>0$ such that
	\begin{align*}
		(m+m_0)\Phi_{\epsilon,m}\geq m\Phi_\epsilon + m_0 \Phi_{\triv,\epsilon}-{C},
	\end{align*}
    where $\Phi_{\triv,\epsilon} = \pi^*\Phi_\triv + \epsilon q_1^*\varphi_A$.
	
	A similar estimate also holds for $\hat\Phi$, that is, 
	\begin{align*}
		(m+m_0)\Phi_{\epsilon,m}\geq m\hat\Phi_\epsilon + m_0 \Phi_{\triv,\epsilon}-C.
	\end{align*}
	Then
 \begin{align*}
         &\int_{X\times \mathbb{D}} e^{\alpha(\hat{\Phi}-\Phi)} \; \Omega\wedge \sqrt{-1} dt \wedge \bar{dt} \\
         &=  \int_{Y\times \mathbb{D}} e^{\alpha(\hat{\Phi}_\epsilon+\frac{m_0}{m}\Phi_{\triv ,\epsilon}- \frac{m+m_0}{m}\Phi_{\epsilon,m})} e^{\alpha(\frac{m+m_0}{m}\Phi_{\epsilon,m}-\Phi_\epsilon-\frac{m_0}{m}\Phi_{\triv,\epsilon}) } \; \Omega\wedge \sqrt{-1} dt \wedge \bar{dt}\\
         &\leq e^{\frac{\alpha C}{m}} \int_{Y\times \mathbb{D}} e^{\alpha(\frac{m+m_0}{m}\Phi_{\epsilon, m}-\Phi_\epsilon-\frac{m_0}{m}\Phi_{\triv,\epsilon})} \; \Omega\wedge \sqrt{-1} dt \wedge \bar{dt}.
    \end{align*}
    Take $\alpha=m$, integral is bounded by the definition of multiplier ideal sheaf.

\end{proof}

We also need a generalization of \cite[Lemma 3.1]{BDL17} in the setting of klt varieties.
\begin{lemma}
\label{BDL_klt}
    Let $(X,L)$ be a polarized klt variety. Let $\omega \in c_1(L)$ be a smooth K\"ahler metric on $X$, i.e, it can be extended to a smooth K\"ahler metric on the ambient projective space. For a given $\varphi\in \mcE^1(X,L)$, there exists $\varphi_j\in C^0(X)\cap C^\infty(X^{\rm reg})$, such that $d_1(\varphi_j,\varphi)\to 0$, $\bfH_\rv(\varphi_j)\to \bfH_\rv(\varphi)$.
\end{lemma}
\begin{proof}
    Let $g = \frac{\MA_\rv(\varphi)}{\omega^n}$, $g\in L^1(X,\omega^n)$. Let $h_j = \min\{g, j\}$, $j\in \NN^*$. As in the proof of \cite[Lemma 3.1]{BDL17}, we have $\int_X|h_j-g|\omega^n \to 0$, and $\int_X h_j\log(h_j)\omega^n \to \int_X g \log(g)\omega^n$. And there exists $0< g_j\in C^\infty(X)$, i.e, $g_j$ is the restriction of a smooth function defined on the ambient projective space, such that $\int_X|h_j-g_j|\omega^n \leq \frac{1}{j}$, $|\int_X \big(h_j\log(h_j)-g_j\log(g_j)\big)\omega^n|\leq \frac{1}{j}$. 
    By \cite[Proposition~3.8]{HL23}, there exists $\varphi_j\in C^0(X)\cap C^\infty(X^{\rm reg})$, such that $\MA_\rv(\varphi_j) = \frac{V_\rv}{\int_X g_j\omega^n}g_j\omega^n$. Without the loss of generality, we may assume $\sup_X g_j = \sup_X\varphi = 0$. Since the entropy is uniformly bounded, by the weak compactness, up to  picking a subsequence, $\varphi_j$ converges to a limit $\varphi'$ in $d_1$-topology, and $\omega_{\varphi_j}^n$ converges to $\omega_{\varphi'}^n$ weakly. Then by the uniqueness of the solution of the complex Monge-Amp\`ere equation, we have $\varphi = \varphi'$.
\end{proof}

Let $U = \hat{\Phi}-\Phi$.
Note that $\bfM_{\rv,\rw}^{\prime\infty}(\Phi)<+\infty$ implies that $\bfH^{\prime\infty}(\phi)<\infty$. By the proof of \cite[Theorem 1.2]{Li22a} and Lemma~\ref{integrabiliy_Phihat-Phi}, for any $\alpha>0$, there exists $s_j\to +\infty$, such that
\begin{align}
    \int_X e^{\alpha U(s_j)} \; \Omega \leq e^{s_j}.
\end{align}
By Lemma~\ref{BDL_klt}, there exists $\varphi_k\in C^0(X)\cap C^\infty(X^{\rm reg})$ that converges to $\Phi(s_j)$.
Apply Jensen's inequality to $\int_X e^{\alpha U(s_j)-\log(\frac{(\ddc \varphi_k)^n}{\Omega})} \; (\ddc \varphi_k)^n \leq e^{s_j}$,
we have
\begin{align*}
    \int_X \log(\frac{(\ddc \varphi_k)^n}{\Omega})(\ddc \varphi_k)^n \geq \alpha\Big(\int_X U(s_j) \; (\ddc \varphi_k)^n \Big) - s_j V.
\end{align*}
Let $k\to\infty$, we then have
\begin{align*}
    \bfH_\Omega(\Phi(s_j)) \geq \alpha (\bfE(\hat{\Phi}(s_j))-\bfE(\Phi(s_j)) - s_j V,
\end{align*}
which furthermore implies
\begin{align}
    \bfH^{\prime\infty}(\Phi) \geq \alpha (\bfE^{\prime\infty}(\hat{\Phi})-\bfE^{\prime\infty}(\Phi)) - V.
\end{align}
If $\Phi$ is not a maximal geodesic ray, then $\bfE^{\prime\infty}(\hat{\Phi})-\bfE^{\prime\infty}(\Phi)>0$. Since $\alpha$ can be chosen arbitrarily large, this contradicts with the fact that $\bfH^{\prime\infty}(\Phi)<\infty$.

\bigskip

     {\bf Step 2: } Show that for a maximal geodesic ray $\Phi$ with the associated $\phi=\Phi^\NA\in\mcE^{1}_\rv(X^\an,L^\an)$, 
     \begin{align}
        \label{Entropy_inequality}
            \bfH^{\prime\infty}_{\rv}(\pi^*\Phi) \geq \bfH^\NA_{\rv}(\pi^*\phi),
    \end{align}
    where we fix a smooth positive measure $\Omega_Y$ on $Y$ and define
    \begin{align*}
        \bfH_{\rv}(\pi^*\Phi(s))
        \coloneqq
        \int_Y \log(\frac{\pi^*\MA_\rv(\Phi(s))}{\Omega_Y})\; \pi^*\MA_\rv(\Phi(s)), 
    \end{align*}
    and $\pi^*\MA_\rv(\Phi(s))$ is viewed as the pullback of $(n,n)$-form.
    
By Proposition~\ref{approximated_geodesic}, there exists a sequence $\phi_m\in \mcH^\NA(X^\an,L^\an)$ that converges to $\phi$,  and each $\phi_m$ corresponds to a test configuration $(\mcX_m,\mcL_m)$. 
%Let $\pi: Y\rightarrow X$ be a smooth resolution of $X$. 
%Then by Corollary~\ref{geodesic_approx_equiv} and Remark~\ref{geodesic_approx_equiv_remark}, we have
%\begin{align}
 %   \bfR_\rv^\NA(\phi) = \bfR_\rv^{\prime\infty}(\Phi),  \;  \bfR_\rv^\NA(\pi^*\phi) = %\bfR_\rv^{\prime\infty}(\pi^*\Phi).
%\end{align}
%By formula \eqref{H_X_H_Y},
%in order to show $\bfH^{\prime\infty}_{\rv}(\Phi) \geq \bfH^\NA_{\rv}(\phi)$, it suffices to show $\bfH^{\prime\infty}_{\rv}(\pi^*\Phi) \geq \bfH^\NA_{\rv}(\pi^*\phi)$.
Let $\mcY$ be an snc model of $Y$, denote $\mcY_m = \pi^*\mcX_m$ and $\mcL_m'=\pi^*\mcL_m$. Let $\mcZ_m$ be a common resolution of $\mcY,\mcY_m$:

\begin{equation}
  \begin{tikzcd}
       & \mcZ_m \arrow[ld, "p_1"'] \arrow[d, "p_0"] \arrow[rd, "p_2"]\\
        \mcY_m \arrow[r,dashed] & Y_\CC & \mcY. \arrow[l, dashed]
    \end{tikzcd}
\end{equation}

As in \cite[Section~5]{Li22a}, one defines
\begin{align*}
    \bfH_\rv^\NA(\pi^*\phi,\mcY)
    \coloneqq\lim_{m\rightarrow\infty}\bfH_\rv^\NA(\pi^*\phi_m,\mcY),
\end{align*}
where 
\begin{align*}
    \bfH_\rv^\NA(\pi^*\phi_m,\mcY)
    \coloneqq& \lim_i(\mcL_m'^{n}\cdot(p_2^*K_{\bar{\mcY}/\PP^1}^{\log}-p_0^*K_{Y_{\PP^1}/\PP^1}^{\log} ))_{\rv_i} \\
    =&\lim\sum_{\vec{k}}a_{\vec{k}}\frac{k_1!\cdots k_r!}{(n+k)!}(p_2^*K_{\bar{\mcY}/\PP^1}^{\log}-p_0^*K_{Y_{\PP^1}/\PP^1}^{\log} )^{[\vec{k}]}\cdot(\mcL_m'^{[\vec{k}]})^{n+k}.
\end{align*}
Recall \cite[(144)]{Li22a},
\begin{align*}
    A_Y(r_\mcY(v))=\sigma(v)(D) \eqqcolon f_D(v),
\end{align*}
where $D\coloneqq K_{\mcY/Y_{\PP^1}}^{\log}$.
By \eqref{MA_v^NA(phi_mcL)smooth_pi^*L}, one has
\begin{align*}
    \int_{Y^\an}A_Y(r_\mcY(v))\MA_\rv^\NA(\pi^*\phi_m)(v)
    =&\sum_E b_E(\mcL_m'^n\cdot E)_\rv \sigma(r(b_E^{-1}\ord_{E})) (D) \\
    =&\sum_E(\mcL_m'^n\cdot E)_\rv  \ord_E(D)  \\
    =&\lim_i\sum_{\vec{k}}a_{\vec{k}}\frac{k_1!\cdots k_r!}{(n+k)!}(K_{\mcY/Y_{\PP^1}}^{\log})^{[\vec{k}]}\cdot(\mcL_m'^{[\vec{k}]})^{n+k} \\
    =&\bfH_\rv^\NA(\pi^*\phi_m,\mcY).
\end{align*}
Taking the limit, we obtain
\begin{align*}
     \bfH^\NA_{\rm v}(\pi^*\phi;\mcY) 
     = \int_{Y^\an} A_Y(r_{\mcY}(v)) \; \MA^\NA_{\rm v}(\pi^*\phi).
\end{align*}
Since $A_Y=\lim_\mcY A_Y\circ r_\mcY$ (see \cite[Theorem A.10]{BJ23}), where $\mcY$ runs all snc models, then we have
\begin{align}
    \label{entropy_model_sup}
    \bfH^\NA_{\rm v}(\pi^*\phi) 
    = \sup_{\mcY} \int_{Y^\an} A_Y(r_{\mcY}(v)) \; \MA^\NA_{\rm v}(\pi^*\phi)
    = \sup_{\mcY} \bfH^\NA_{\rm v}(\pi^*\phi;\mcY) .
\end{align}
To show \eqref{Entropy_inequality}, it suffices to show that
\begin{align}
    \bfH_\rv^{\prime\infty}(\pi^*\Phi) \geq  \bfH^\NA_{\rm v}(\pi^*\phi;\mcY)
\end{align}
for any model $\mcY$.

\begin{lemma} 
\label{lem:holder_estimate}
For any $\varphi_0\in\mcE^1(X,L)$ and $\varphi_1\in\mcH(X,L)$, $f\in \mcH(Y,\omega_Y)$, then
	\begin{align}
		&\left|\int_Y f\,\Big(\MA_\rv(\pi^*\varphi_0)- \MA_\rv(\pi^*\varphi_1)\Big)\right| \nonumber\\
		\leq & C \left(\sup_Y|f|\int_Y\omega_Y\wedge(\pi^*\omega)^{n-1}\right)^{\frac{1}{2}}\cdot\bfI(\varphi_0,\varphi_1)^{\alpha_n/2}\max_i(1+\bfI(\varphi_i))^{\frac{1}{2}-\frac{\alpha_n}{2}},
        \label{eq:holder_estimate_f()}
	\end{align}
    where $\alpha_n\coloneqq 2^{-n}$.
\end{lemma}
\begin{proof}
	First, we assume $\varphi_0\in\mcH(X,L)$, and denote $\varphi_t=t\varphi_0+(1-t)\varphi_1$ for $t\in [0,1]$.
	By the integration by part (see \cite[Lemma~3.10]{BJT24}), then
	\begin{align*}
		&\int_Y f\,\Big(\MA_\rv(\pi^*\varphi_0)- \MA_\rv(\pi^*\varphi_1)\Big) \\
		=&\int_Yf\,(\int_0^1\frac{d}{dt}\MA_\rv(\pi^*\varphi_t)dt) \\
		=&\int_0^1dt\int_Yf\Big(n\pi^*\rv(m_{\varphi_t})\ddc\pi^*(\varphi_0-\varphi_1)\wedge(\ddc\pi^*\varphi_t)^{n-1}+\pi^*\langle\rv'(m_{\varphi_t}),m_{\ddc(\varphi_0-\varphi_1)}\rangle\MA_\rv(\pi^*\varphi_t) \Big)  \\
		=&-n\int_0^1dt\int_X\pi^*\rv(m_{\varphi_t})d f\wedge d^c\pi^*(\varphi_0-\varphi_1)\wedge(\ddc\pi^*\varphi_t)^{n-1}.
	\end{align*}
	By the Cauchy-Schwarz inequality, one has
	\begin{align*}
		&\left|\int_Y f\,\Big(\MA_\rv(\pi^*\varphi_0)- \MA_\rv(\pi^*\varphi_1)\Big)\right|
		\leq C\sup_{t\in[0,1]}\left(\int_Yd f\wedge d^cf\wedge(\ddc\pi^*\varphi_t)^{n-1} \right)^{\frac{1}{2}} \\
		 &\qquad\qquad\qquad\qquad\qquad\qquad\ \cdot \left(\int_Yd \pi^*(\varphi_0-\varphi_1)\wedge d^c\pi^*(\varphi_0-\varphi_1)\wedge(\ddc\pi^*\varphi_t)^{n-1} \right)^{\frac{1}{2}}.
	\end{align*}
	Note that
	\begin{align*}
		&\int_Yd \pi^*(\varphi_0-\varphi_1)\wedge d^c\pi^*(\varphi_0-\varphi_1)\wedge(\ddc\pi^*\varphi_t)^{n-1} \\
		=&\int_X d (\varphi_0-\varphi_1)\wedge d^c(\varphi_0-\varphi_1)\wedge(\ddc\varphi_t)^{n-1} \\
		\leq& C\bfI(\varphi_0,\varphi_1)^{\alpha_n}(1+\bfI(\varphi_t))^{1-\alpha_n} 
		\leq  C\bfI(\varphi_0,\varphi_1)^{\alpha_n}\max_i(1+\bfI(\varphi_i))^{1-\alpha_n}. 
	\end{align*}
	On the other hand, since $\omega_Y+\ddc f>0$, then we have
	\begin{align}
		\int_Yd f\wedge d^cf\wedge(\ddc\pi^*\varphi_t)^{n-1}
		=&-\int_Y f\ddc f\wedge(\ddc\pi^*\varphi_t)^{n-1} \nonumber \\
		=&\int_Yf\omega_Y\wedge(\ddc\pi^*\varphi_t)^{n-1}-\int_Yf(\omega_Y+\ddc f)\wedge(\ddc\pi^*\varphi_t)^{n-1} \nonumber \\
		\leq&2\sup|f|\cdot\int_Y\omega_Y\wedge(\pi^*\omega)^{n-1}.
	\end{align}
    Combining the above inequalities, we obtain \eqref{eq:holder_estimate_f()}.

    For any $\varphi_0\in\mcE^1(X,L)$, one can choose a sequence $\varphi_j\in\mcH(X,L)$ such that $\varphi_j$ strongly converges to $\varphi_0$.
    One has
    \begin{align*}
        &\left|\int_Y f\,\Big(\MA_\rv(\pi^*\varphi_j)- \MA_\rv(\pi^*\varphi_1)\Big)\right| \nonumber\\
		\leq & C \left(\sup_Y|f|\int_Y\omega_Y\wedge(\pi^*\omega)^{n-1}\right)^{\frac{1}{2}}\cdot\bfI(\varphi_j,\varphi_1)^{\alpha_n/2}\max(1+\bfI(\varphi_0), 1+\bfI(\varphi_1))^{\frac{1}{2}-\frac{\alpha_n}{2}}.
    \end{align*}
    Taking $i\rightarrow\infty$ in the both side of the above inequality implies \eqref{eq:holder_estimate_f()}.

\end{proof}

Let $\mcY$ be a snc model of $Y$.
\begin{proposition}
\label{prop:slope_int_MA_v()}
	Let $\Phi$ be a maximal geodesic in $\mcE^1(X,L)$ with $\phi=\Phi^\NA\in\mcE^1_\rv(X^\an,L^\an)$, and $\nu$ be a smooth Hermitian metric on $-K_{\mcY/\PP^1}^{\log}$. Then
	\begin{align}
	\label{eq:slope_ent_mcY}
		\lim_{s\rightarrow\infty}s^{-1}\int_Y\log(\frac{\nu_s}{\Omega_Y})\MA_\rv(\pi^*\Phi(s))
		= \bfH^\NA_{\rm v}(\pi^*\phi;\mcY) 
     = \int_{Y^\an} A_Y(r_{\mcY}(v)) \; \MA^\NA_{\rm v}(\pi^*\phi).
	\end{align}
\end{proposition}
\begin{proof}
	For simplicity, we denote by $f_s\coloneqq\log(\frac{\nu_s}{\Omega_Y})$, which defines a ray with analytic singularity such that the non-Archimedean limit (\cite{BHJ19}) is $\varphi_D= A_Y\circ r_\mcY$, where $D\coloneqq K_{\mcY/Y_{\PP^1}}^{\log}$.
	We can find ample line bundles $F_i$ on $Y$, and relative ample line bundles $\mcF_i$ on $\mcY$, such that such that $F_1-F_2=\mcO_Y$, $\mcF_1-\mcF_2 = D$. Correspondingly, we also have psh rays with analytic singularities $\{\varphi_{i,s}\}$ in $\mcH(Y,F_i)$  for $i=1,2$, such that $f_s=\varphi_{1,s}-\varphi_{2,s}$. We denote by $\phi_i$ the non-Archimedean limit of $\{\varphi_{i,s}\}$ for $i=1,2$.  Then, $\varphi_D=\phi_1-\phi_2$.
	Therefore, to show \eqref{eq:slope_ent_mcY}, it suffices to prove 
	\begin{align}
	\label{eq:slope_formula_phi_i}
		\lim_{s\rightarrow\infty}s^{-1}\int_Y\varphi_{i,s}
        \MA_\rv(\pi^*\Phi(s))
		=\int_{Y^\an} \phi_i \; \MA^\NA_{\rm v}(\pi^*\phi).
	\end{align}
	
Proposition~\ref{approximated_geodesic}, there exists a decreasing sequence $\phi_m\in\mcH^\NA(X^\an,L^\an)$ that converges to $\phi$. There exists a unique geodesic ray $\Phi_m$ emanating from $\varphi_m$ with $\Phi_m^\NA = \phi_m$, which corresponds to a test configuration $(\mcX_m,\mcL_m)$. By the maximum principle, $\Phi_{m}\searrow \Phi$. 

{\bf Claim:} for any $m$,
\begin{align*}
	\lim_{s\rightarrow\infty}s^{-1}\int_Y\varphi_{i,s}\MA_\rv(\pi^*\Phi_m(s))
		=\int_{Y^\an} \phi_i \; \MA^\NA_{\rm v}(\pi^*\phi_m).
\end{align*}
{\bf Proof of Claim:} Let $\Psi_m$ be the smooth sub-geodesic induced by $\CC^*$-action of the test configuration, which emanates from $\varphi_m$ such that $\Psi_m^\NA = \phi_m$. We have
\begin{align}
	&\left|\int_Y \varphi_{i,s}\,\MA_\rv(\pi^*\Psi_m(s))- \int_Y \varphi_{i,s}\,\MA_\rv(\pi^*\Phi_m(s))\right| \nonumber \\
	=&\left|\int_Y \varphi_{i,s}\,\Big(\MA_\rv(\pi^*\Psi_m(s))- \MA_\rv(\pi^*\Phi_m(s))\Big)\right| \nonumber\\
	\leq&C \left(\sup_Y|\varphi_{i,s}|\cdot(F_i\cdot(\pi^*L)^{n-1})\right)^{\frac{1}{2}}\cdot\bfI(\Psi_m(s),\Phi_m(s))^{\alpha_n/2}\Big(1+\max\{\bfI(\Psi_m(s)), \bfI(\Phi_m(s)\}\Big)^{\frac{1}{2}-\frac{\alpha_n}{2}}.
	\label{eq:holder_estimate_along_ray}
\end{align}
Since $\lim_{s\to \infty}\frac{1}{s}\bfI(\Psi_m(s),\Phi_m(s)) = 0$, we have
\begin{align*}
    \lim_{s\rightarrow\infty}s^{-1}\int_Y\varphi_{i,s}\MA_\rv(\pi^*\Psi_m(s)) = \lim_{s\rightarrow\infty}s^{-1}\int_Y\varphi_{i,s}\MA_\rv(\pi^*\Phi_m(s)).
\end{align*}
Therefore to show the Claim, it suffices to show that
\begin{align}
	\label{claim_aid}\lim_{s\rightarrow\infty}s^{-1}\int_Y\varphi_{i,s}\MA_\rv(\pi^*\Psi_m(s))
		=\int_{Y^\an} \phi_i \; \MA^\NA_{\rm v}(\pi^*\phi_m).
\end{align}

As $\pi^*\Psi_m$ is a smooth psh metric on $\mcY_m$, where $\mcY_m$ is a snc model associated with $\pi^*\phi_m$, we have
$\MA_\rv(\pi^*\Psi_m(s))$ converges to $\MA_\rv(\pi^*\Psi_m|_{\mcY_{m,0}})$ weakly as measures. Meanwhile, $\varphi_{i,s}$ is a continuous function on the Hybrid space $\mcY_m^{\bf hyb}$. Then by \cite[Theorem B]{PS22}, the equality \eqref{claim_aid} holds. This concludes the Claim.

\vspace{0.15in}

On the other hand, by Lemma~\ref{lem:holder_estimate}, one has
\begin{align}
	&\left|\int_Y \varphi_{i,s}\,\MA_\rv(\pi^*\Phi(s))- \int_Y \varphi_{i,s}\,\MA_\rv(\pi^*\Phi_m(s))\right| \nonumber \\
	=&\left|\int_Y \varphi_{i,s}\,\Big(\MA_\rv(\pi^*\Phi(s))- \MA_\rv(\pi^*\Phi_m(s))\Big)\right| \nonumber\\
	\leq&C \left(\sup_Y|\varphi_{i,s}|\cdot(F_i\cdot(\pi^*L)^{n-1})\right)^{\frac{1}{2}}\cdot\bfI(\Phi(s),\Phi_m(s))^{\alpha_n/2}\Big(1+\max\{\bfI(\Phi(s)), \bfI(\Phi_m(s)\}\Big)^{\frac{1}{2}-\frac{\alpha_n}{2}}.
	\label{eq:holder_estimate_along_ray}
\end{align}
Note that
\begin{align*}
	\lim_{s\rightarrow\infty}s^{-1}\bfI(\Phi(s),\Phi_m(s))
	=&\bfI^\NA(\phi,\phi_m), \\
	\lim_{s\rightarrow\infty}s^{-1}\max\{\bfI(\Phi(s)), \bfI(\Phi_m(s)\}
	=&\max\{\bfI^\NA(\phi),\bfI^\NA(\phi_m)\}.
\end{align*}
And $\varphi_{i,s}$ extends to a quasi-psh function with analytic singularity on a simple-normal-crossing model $\mcY_i$, we also have
\begin{align*}
	\lim_{s\rightarrow\infty}s^{-1}\sup_Y|\varphi_{i,s}|
	=\sup_{Y^\an}|\phi_i|. 
\end{align*}
By \eqref{eq:holder_estimate_along_ray}, we obtain
\begin{align}
	&\int_{Y^\an} \phi_i \; \MA^\NA_{\rm v}(\pi^*\phi_m)
		-C_{F_i}(\sup_{Y^\an}|\phi_i|)^{\frac{1}{2}}\cdot\bfI^\NA(\phi,\phi_m)^{\frac{\alpha_n}{2}}\cdot\Big(\max\{\bfI^\NA(\phi),\bfI^\NA(\phi_m)\}\Big)^{\frac{1}{2}-\frac{\alpha_n}{2}}  \nonumber\\
		\leq&\liminf_{s\rightarrow\infty}s^{-1}\int_Y\varphi_{i,s}\,\MA_\rv(\pi^*\Phi(s))
		\leq \limsup_{s\rightarrow\infty}s^{-1}\int_Y\varphi_{i,s}\,\MA_\rv(\pi^*\Phi(s)) \nonumber  \\
		\leq& \int_{Y^\an} \phi_i \; \MA^\NA_{\rm v}(\pi^*\phi_m)
		+C_{F_i}(\sup_{Y^\an}|\phi_i|)^{\frac{1}{2}}\cdot\bfI^\NA(\phi,\phi_m)^{\frac{\alpha_n}{2}}\cdot\Big(\max\{\bfI^\NA(\phi),\bfI^\NA(\phi_m)\}\Big)^{\frac{1}{2}-\frac{\alpha_n}{2}} . 
		\label{eq:upperlower_bound_slope_formula}
\end{align}
Since
\begin{align*}
	\lim_{m\rightarrow\infty}\int_{Y^\an} \phi_i \; \MA^\NA_{\rm v}(\pi^*\phi_m)
	=\int_{Y^\an} \phi_i \; \MA^\NA_{\rm v}(\pi^*\phi),
\end{align*}
by letting $m\rightarrow\infty$ in \eqref{eq:upperlower_bound_slope_formula}, we obtain \eqref{eq:slope_formula_phi_i}.
Therefore, we finish the proof.

\end{proof}

Reformulate the entropy into
\begin{align}
\label{eq:entropy_reformula}
    \bfH_\rv(\pi^*\Phi(s)) 
    =& 
     \int_Y \log\left(\frac{\MA_\rv(\pi^*\Phi(s))}{\nu}\right)\; \MA_\rv(\pi^*\Phi(s)) 
     +\int_Y\log\left(\frac{\nu}{\Omega_Y}\right)\,\MA_\rv(\pi^*\Phi(s)).
\end{align}
By Jenssen's inequality and limit of Monge--Amp\`ere measures in \cite[Lemma 3.11]{BHJ19},
\begin{align}
\label{eq:1st_term_ent}
    \int_Y \log\left(\frac{\MA_\rv(\pi^*\Phi(s))}{\nu}\right)\;\MA_\rv(\pi^*\Phi(s)) \geq -\VV_\rv\log\int_{Y_\tau}\frac{\nu}{\VV_\rv}=-O(\log s^d),
\end{align}
where $s=-\log|\tau|^2$ and $d$ is the dimension of the dual complex of $\mcY_0$.

By \eqref{eq:entropy_reformula}, \eqref{eq:1st_term_ent} and Proposition \ref{prop:slope_int_MA_v()}, we obtain
\begin{align*}
	\bfH_\rv^{\prime\infty}(\pi^*\Phi)
	\geq  \bfH^\NA_{\rm v}(\pi^*\phi;\mcY). 
\end{align*}
Therefore, we have
\begin{align*}
    \bfH_\rv^{\prime\infty}(\pi^*\Phi)
	\geq \sup_\mcY \bfH^\NA_{\rm v}(\pi^*\phi;\mcY)
    =\bfH_\rv^\NA(\pi^*\phi).
\end{align*}

%%%%%%%%%%%%%%%%%%%%%%%%%%%%%%%%%%%%%%%%%%%%%%%%%%%%%%%%%%%%%%%%%%%%%%%%%%%%%%%%%%%%%%%
When the class $c_1(\pi^*L)$ on $Y$ is semipositive, the Ricci energy $\bfE^{-\Ric(\Omega_Y)}(\varphi)$ is not necessarily well-defined for an arbitrary psh metric $\varphi\in\mcE^1(Y,\pi^*L)$. 
This motivates restricting its definition to a suitable subset of $\mcE^1(Y,\pi^*L)$.
Denote $Ent(X,L)\coloneqq\{\varphi\in\mcE^1(X,L):\bfH_\rv(\varphi)<\infty\}$.

\begin{prop-def}
\label{prop-def:Ricci_energy_on_Y}
	The Ricci energy $\bfE^{-\Ric(\Omega_Y)}_\rv:\pi^*{\rm Psh}(X,L)\cap L^\infty\rightarrow\RR$ admits a unique extension $\bfE^{-\Ric(\Omega_Y)}_\rv:\pi^*Ent(X,L)\rightarrow\RR$ satisfying 
	\begin{align*}
		\bfH_\rv(\pi^*\varphi)-\bfH_\rv(\varphi)
		=\bfE_\rv^{-\Ric(\pi^*\Omega)}(\pi^*\varphi)
     -\bfE_\rv^{-\Ric(\Omega_Y)}(\pi^*\varphi)
     +\int_Y\log(\frac{\pi^*\Omega}{\Omega_Y})\; \MA_\rv(\pi^*\varphi_0),
	\end{align*}
	for any $\varphi\in Ent(X,L)$.
\end{prop-def}
\begin{proof}
	For any $\varphi\in\mcE^1(X,L)$, by Lemma~\ref{BDL_klt}, then there exists $\varphi_j\in C^0(X)\cap C^\infty(X^{\rm reg})$, such that $d_1(\varphi_i,\varphi)\to 0$, $\bfH_\rv(\varphi_i)\to \bfH_\rv(\varphi)=\int_X\frac{\MA_\rv(\varphi)}{\Omega}\MA_\rv(\varphi)$, where $\Omega=\omega^n$.
	For each $\varphi_i$, integration by parts implies 
\begin{align*}
    &\bfH_\rv(\pi^*\varphi_i)-\bfH_\rv(\varphi_i) \\
    =&\int_Y \log(\frac{\MA_\rv(\pi^*\varphi_i)}{\Omega_Y})\; \MA_\rv(\pi^*\varphi_i)
     -\int_Y \log(\frac{\MA_\rv(\pi^*\varphi_i)}{\pi^*\Omega})\; \MA_\rv(\pi^*\varphi_i) \\
    =&\int_Y\log(\frac{\pi^*\Omega}{\Omega_Y})\; \MA_\rv(\pi^*\varphi_i) \\
    =&\bfE_\rv^{-\Ric(\pi^*\Omega)}(\pi^*\varphi_i)
     -\bfE_\rv^{-\Ric(\Omega_Y)}(\pi^*\varphi_i)
     +\int_Y\log(\frac{\pi^*\Omega}{\Omega_Y})\; \MA_\rv(\pi^*\varphi_0).
\end{align*}

It suffices to show that 
\begin{align*}
	\int_Y\log(\frac{\pi^*\Omega}{\Omega_Y})\; \MA_\rv(\pi^*\varphi_i)
	\rightarrow
	\int_Y\log(\frac{\pi^*\Omega}{\Omega_Y})\; \MA_\rv(\pi^*\varphi)
\end{align*}
For simplicity, we
denote $\Theta\coloneqq\pi^*\Omega$, $f_i\cdot\Theta\coloneqq\MA_\rv(\pi^*\varphi_i)$,
$f\cdot\Theta\coloneqq\MA_\rv(\pi^*\varphi)$ and $g\coloneqq \log(\frac{\pi^*\Omega}{\Omega_Y})$.
Next, we prove the following desired convergence
\begin{align}
\label{eq:L1-convergence}
	\int_Yg\cdot f_i\,\Theta
	\rightarrow \int_Yg\cdot f\,\Theta.
\end{align}

By \cite[Theorem~4.5.8, Example~4.5.10]{Bog07}, the convergence of entropy and $\varphi\in Ent(X,L)$ imply that $\{f_i\}$ is uniformly integrable (see \cite[Definition~4.5.1]{Bog07}).
Note that $f_i\rightarrow f$ in measure.  By Lebesgue--Vitali theorem (see \textup{\cite[Theorem~4.5.4]{Bog07}}), then $f_i\rightarrow f$ in $L^1(\Theta)$.
It follows from the definition that for any $\epsilon>0$,
\begin{align*}
\lim_{i\rightarrow\infty}\Theta(\{y\in Y:|f_i(y)-f(y)|>\epsilon\})=0.
\end{align*}
One denotes by $E_\epsilon\coloneqq\{y\in Y:|g(f_i-f)(y)|>\epsilon\}$. For any sufficient large $M\gg1$, one splits
\begin{align*}
	\Theta(E_\epsilon)
	=\Theta(E_\epsilon\cap\{|g|\leq M\})
	 +\Theta(E_\epsilon\cap\{|g|> M\}).
\end{align*}
On the one hand, 
\begin{align*}
	\Theta(E_\epsilon\cap\{|g|> M\})
	\leq \Theta(E_\epsilon\cap\{|g|> M\}).
\end{align*}
Since $g$ is finite almost everywhere, then for any $\delta>0$, we can choose $M\gg1$ sufficiently large such that $\Theta(E_\epsilon\cap\{|g|> M\})<\frac{\delta}{2}$.
On the other hand, for any $y\in E_\epsilon\cap\{|g|\leq M\}$, then $|(f_i-f)(y)|>\frac{\epsilon}{M}$. Thus
\begin{align*}
	\Theta(E_\epsilon\cap\{|g|\leq M\})
	\leq
	\Theta(\{(f_i-f)(y)|>\frac{\epsilon}{M} \}).
\end{align*}
Since $f_i\rightarrow f$ in measure, then one can choose $i\gg1$ sufficiently large, such that
\begin{align*}
	\Theta(\{(f_i-f)(y)|>\frac{\epsilon}{M} \})
	\leq\frac{\delta}{2}.
\end{align*}
Therefore, for any $\delta>0$, there exists $N=N(M)\gg1$ such that for all $i>N$,
\begin{align*}
	\Theta(\{y\in Y:|g(f_i-f)(y)|>\epsilon\})
	<\frac{\delta}{2}+\frac{\delta}{2}=\delta.
\end{align*}
We obtain that $g\cdot f_i\rightarrow g\cdot f$ in measure $\Theta$.

%Note that $\ddc g=\Ric(\Omega_Y)-\pi^*\Ric(\Omega)\geq-C\omega_Y$. 
Note that
\begin{align*}
    \Omega=\omega^n
    =g_0\left(\sum_j|g_i|^2 \right)^{\frac{1}{m}}\cdot\mu_h,
\end{align*}
where $g_0$ is a positive smooth function, $g_i$ are holomorphic functions such that $V((g_i)_i)=X^{\rm sing}$ and $\mu_h$ is the adapted measure (see \cite[\S 1.2.5]{PTT23}).
By \cite[Lemma~3.2]{BBEGZ19}, then
\begin{align*}
    g=\log(\frac{\pi^*\Omega}{\Omega_Y})
    =\log(g_0)+\frac{1}{m}\log\sum_j|g_j|^2+\psi^+-\psi^-.
\end{align*}
Thus,
\begin{align*}
    |g|\leq |\log(g_0)|-\frac{1}{m}\log\sum_j|g_j|^2-\psi^+-\psi^-+ C.
\end{align*}
It follows that there exists a sufficiently small $\lambda>0$ such that
\begin{align}
\label{eq:integrablity_volume_ratio}
	\int_Ye^{\lambda|g|}\,\Theta<\infty.
\end{align}

Recall Young inequality with respect to weights $\chi(t)=e^{t}$ and $\chi^*(s)=s\log s-s$, 
\begin{align*}
	st\leq (s\log s-s)+e^t.
\end{align*}
Applying the above inequality with $s=f_i>0$ and $t=\lambda |g|\geq0$, then
\begin{align*}
	|g\cdot f_i|
	\leq\frac{1}{\lambda}f_i\log f_i+\frac{1}{\lambda}e^{\lambda |g|}
\end{align*}
For $\epsilon>0$ and any measurable set $B$ with $\Theta(B)<\delta$ (determined later), 
\begin{align*}
	\int_{B}|g\cdot f_i|\, \Theta
	\leq \frac{1}{\lambda}\int_Bf_i\log f_i\,\Theta+\frac{1}{\lambda}\int_Be^{\lambda |g|}\,\Theta.
\end{align*}
By \eqref{eq:integrablity_volume_ratio}, then we can choose $\delta>0$ sufficiently small, such that
\begin{align*}
	\int_Be^{\lambda |g|}\,\Theta\leq \frac{\epsilon}{2}.
\end{align*}
Since $\int_Xf_i\log f_i\,\Theta<C_0$, then we can choose $T>0$ such that
\begin{align*}
	\int_Bf_i\log f_i\,\Theta
	=&\int_{B\cap\{f_i\leq T\}}f_i\log f_i\,\Theta+\int_{B\cap\{f_i>T\}}f_i\log f_i\,\Theta \\
	\leq&T\log T\cdot\Theta(B)+\int_{B\cap\{f_i>T\}}f_i\log f_i\, \Theta
	\leq (T\log T)\delta+ \frac{\epsilon}{4}
	\leq \frac{\epsilon}{2}.
\end{align*}
Thus we obtain
\begin{align*}
	\int_{B}|g\cdot f_i|\, \Theta<\epsilon.
\end{align*}
By \cite[Proposition~4.5.3]{Bog07},
we derive that $\{g\cdot f_i\}$ is uniformly integrable.

Again by Lebesgue--Vitali theorem (see \textup{\cite[Theorem~4.5.4]{Bog07}}), together with the convergence $g\cdot f_i\rightarrow g\cdot f$ in measure, we obtain $g\cdot f_i\rightarrow g\cdot f$ in $L^1(\Theta)$, that is, \eqref{eq:L1-convergence}.
Moreover, $\int_Yg\cdot f\, \Theta<\infty$.

\end{proof}

If $\Phi\in\mcE^1(X,L)$ is a maximal geodesic ray with $\bfH_\rv^{\prime\infty}(\Phi)<\infty$.
By Proposition-Definition~\ref{prop-def:Ricci_energy_on_Y}, one has
\begin{align*}
		\bfH_\rv(\pi^*\Phi(s))-\bfH_\rv(\Phi(s))
		=&\bfE_\rv^{-\Ric(\pi^*\Omega)}(\pi^*\Phi(s))
     -\bfE_\rv^{-\Ric(\Omega_Y)}(\pi^*\Phi(s)) \\
     &+\int_Y\log(\frac{\pi^*\Omega}{\Omega_Y})\; \MA_\rv(\pi^*\varphi_0).
	\end{align*}
Taking the slope at infinity, we have
\begin{align}
\label{H+R=H+R_slope}
	\bfH_\rv^{\prime\infty}(\pi^*\Phi)-\bfH_\rv^{\prime\infty}(\Phi)
	=\bfR_\rv^{\prime\infty}(\Phi)-\bfR_\rv^{\prime\infty}(\pi^*\Phi).
\end{align}
In addition, by Proposition~\ref{prop:slope_int_MA_v()} and \eqref{H+R=H+R_slope}, we obtain the following.
\begin{align}
    \label{eq:R_slope=RNA}
    \bfR_\rv^{\prime\infty}(\pi^*\Phi)
    =\bfR_\rv^\NA(\pi^*\phi).
\end{align}

\bigskip
    
    {\bf Step 3: } Show that for any $\phi\in\mcE^{1}_\rv(X^\an,L^\an)$, there exist model metrics $\phi_j$, $j\in\NN$, such that $\phi_j\to\phi$ in strong topology, and $\bfH^\NA_{\rv}(\pi^*\phi_j)\to \bfH^\NA_{\rv}(\pi^*\phi)$.

Denote $\nu = \MA^\NA_\rv(\pi^*\phi)$, $\mu = \MA^\NA_\rv(\phi)$. Then $\mu = \pi_*\nu$.
By \eqref{entropy_model_sup}, there exists a sequence of s.n.c models $\mcY_j$, such that $\nu_j \coloneqq (r_{\mcY_j})_*\nu$, and since $A_Y$ is continuous on $Y^{\rm val}$,
\begin{align}
    \bfH_\rv^\NA(\pi^*\phi) = \lim_{j\to\infty}\int_{Y^\an} A_Y(v)\; \nu_j.
\end{align}
Then $\nu_j$ is supported on ${\rm Sk}(\mcY_j)\subset Y^{\rm lin}$, where $Y^{\rm lin}$ is the space of valuations of linear growth, see \cite[Section 11]{BJ22}. By \cite[Proposition 11.10]{BJ22}, $\pi:Y^{\rm lin}\rightarrow X^{\rm lin}$ is a homeomorphism.
Let $\mu_j = \pi_*\nu_j$, the support of which is in $X^{\rm lin}$.
By Theorem~\ref{weighted_NA_Calabi_problem}, there exists $\phi_j\in C^0(X^\an)$ such that $\MA_\rv^\NA(\phi_j)=\mu_j$.
In addition, $\MA_\rv^\NA(\pi^*\phi_j)=\nu_j$.

Let $\phi'_{j,k}\in\mcH(X^\an,L^\an)$ such that $\phi'_{j,k}$ strongly converges to $\phi_j$ and $\MA_\rv^\NA(\phi'_{j,k})$ converges to $\MA_\rv^\NA(\phi_j)=\mu_j$. 
Let $\nu_{j,k} = (r_{\mcY_j})_*\MA^\NA_\rv(\pi^*\phi_{j,k}')$,which is a divisorial type measure on $Y^\an$.
Let $\mu_{j,k}=\pi_*\nu_{j,k}$, which is a divisorial type measure on $X^\an$.
Let $\phi_{j,k}$ be the solution of $\MA_\rv^\NA(\phi_{j,k}) = \mu_{j,k}$. 
Then $\MA_\rv^\NA(\pi^*\phi_{j,k}) = \nu_{j,k}$,  which is a divisorial measure supported on ${\rm Sk}(\mcY_j)$. By Theorem \ref{Thm_M_div_MA}, $\phi_{j,k} = \FS(\chi_{j,k})$. Up to making a small perturbation, we may assume $\chi_{j,k}\in \mcN^\div_\QQ$. Then $\chi_{j,k}$ is induced by a model.

Since $\pi^*(\tilde{\phi}-\phi_{\triv})\leq \pi^*(\tilde{\phi}-\phi_\triv)\circ r_{\mcY_j}$ for any psh function $\tilde{\phi}$, following the definition of $\bfE^*_\rv$, we have $(\bfE^*_\rv)^\NA(\mu_{j,k}) \leq (\bfE^*_\rv)^\NA(\MA_\rv^\NA(\phi'_{j,k}))$. Then $\limsup_{k\to\infty}(\bfE^*_\rv)^\NA(\mu_{j,k})\leq \limsup_{k\to\infty}(\bfE^*_\rv)^\NA(\MA_\rv^\NA(\phi'_{j,k})) = (\bfE^*_\rv)^\NA(\mu_j)$.
By the same argument as in the proof of \cite[Proposition 6.3]{Li22a}, we also have $\liminf_{k\to\infty}(\bfE^*_\rv)^\NA(\mu_{j,k}) \geq (\bfE^*_\rv)^\NA(\mu_j)$. Therefore $\mu_{j,k}$ converges to $\mu_{j}$ in strong topology.
Similarly, we have $(\bfE^*_\rv)^\NA(\mu_j)\leq (\bfE^*_\rv)^\NA(\mu)$. On the other hand, 
\begin{align*}
\liminf_{j\to\infty}(\bfE^*_\rv)^\NA(\mu_j) 
\geq& \liminf_{j\to\infty} \Big(\bfE_\rv^\NA(\phi)-\int_{X^\an}(\phi-\phi_\triv)\mu_j\Big) \\
=&\bfE_\rv^\NA(\phi)-\int_{X^\an}(\phi-\phi_\triv)\mu = \bfE_\rv^*(\mu).\end{align*}
This implies that $\mu_j$ converges to $\mu$ in strong topology. Therefore the divisorial measure $\mu_{j,k}$ converges to $\mu$ in strong topology.

Since $(\bfE^*_\rv)^\NA(\nu_{j,k}) = (\bfE^*_\rv)^\NA(\mu_{j,k})$, $\nu_{j,k}$ also converges to $\nu$ in strong topology. Since $A_Y$ is a continuous function on $Y^{\rm val}$, we have
\begin{align}
    \bfH^\NA_\rv(\pi^*\phi) = \int_{Y^\an} A_Y(v)\; \nu = \lim_{j,k\to\infty}\int_{Y^\an} A_Y(v)\; \nu_{j,k} = \lim_{j,k\to\infty} \bfH_\rv^\NA(\pi^*\phi_{j,k}).
\end{align}

    \bigskip

    {\bf Step 4: } Show that if $(X,L)$ is $\bG$-uniformly weighted K-stable for models, then the weighted Mabuchi functional $\bfM_{\rv,\rw}$ is $\bG$-coercive over $\mcE^1_{ K}$.

Assume that $\bfM_{\rv,\rw}$ is not $\bG$-coercive. 
Using the same proof as in \cite[Lemma 6.2]{HL23}, there exists a psh geodesic ray $\Phi=(\Phi(t))\in\mcE_K^1(X,L)$ emanating from $\varphi_0$ satisfying
\begin{align*}
    \sup(\Phi(t)-\varphi_0)=&0, \quad \bfE_\rv(\Phi(t))=-t, \\
    \bfM_{\rv,\rw}^{\prime\infty}(\Phi)\leq&0, \\
    \inf_{\xi\in N_\RR}\bfJ_\rv^{\prime\infty}(\Phi_\xi)=&1.
\end{align*}
Hence, by {\bf Step 1}, $\Phi$ is a maximal geodesic.
By {\bf Step 3}, there exist model metrics $\phi_{\mcL_j}$ such that $\phi_{\mcL_j}\to\Phi^\NA$ in strong topology, and $\bfH^\NA_{\rv}(\pi^*\phi_{\mcL_j})\to \bfH^\NA_{\rv}(\pi^*\Phi^\NA)$.
Then, we obtain
\begin{align*}
    0\geq \bfM_{\rv,\rw}^{\prime\infty}(\Phi)
    = \bfM_{\rv,\rw}^{\prime\infty}(\pi^*\Phi)
    \geq& \bfM_{\rv,\rw}^{\NA}(\pi^*\Phi^\NA) \quad ( \eqref{H+R=H+R_slope}, {\bf Step\ 2}, \eqref{eq:R_slope=RNA}) \\
    =&\lim_{j\rightarrow\infty} \bfM_{\rv,\rw}^{\NA}(\pi^*\phi_{\mcL_j}) \quad ({\bf Step\ 3}) \\
    \geq&\lim_{j\rightarrow\infty} \gamma\cdot\bfJ_{\rv,\TT}^{\NA}(\pi^*\phi_{\mcL_j}) \quad(\textup{Corollary~\ref{birational_uKs}}) \\
    =&\lim_{j\rightarrow\infty} \gamma\cdot\bfJ_{\rv,\TT}^{\NA}(\phi_{\mcL_j}) \quad(\textup{Proposition~\ref{M_v,w_equal}})  \\
    =&\gamma\cdot\bfJ_{\rv,\TT}^{\NA }(\Phi^\NA) \quad(\textup{Lemma~\ref{reduced_J_v_convergence}}) \\
    =&\gamma\cdot\inf_{\xi\in N_\RR}\bfJ_\rv^{\prime\infty}(\Phi_\xi)\quad(\textup{Corollary~\ref{slope_reduced_J_v}}) \\
    =&\gamma>0,
\end{align*}
which leads to the contradiction.

\end{proof}

\section{Openness of \texorpdfstring{$\GG$}{G}-uniform weighted K-stability in toric case}
\label{toric}

In this final section, we focus on toric varieties. We give an equivalent characterization of $\GG$-uniform weighted K-stability, then prove Theorem~\ref{theorem_B}. 
\vskip1em

Let $(X,L)$ be a toric polarized klt variety with the associated polytope $P$.
\begin{lemma}
    The \textit{envelope property} holds for $\PSH_T(X^\an,L^\an)$, i.e, for any increasing sequence $\phi_j$ of toric invariant psh functions that is uniformly  bounded from above, $(\lim_{j\to\infty}\varphi_j)^*$ is in $\PSH_T(X^\an,L^\an)$.
\end{lemma}
\begin{proof}
    Let $\pi:Y\rightarrow X$ be a smooth toric resolution.
    By the description of the toric invariant psh functions in \cite[Theorem 4.8.1]{BPS14} and \cite{BGJK25}, we have $\PSH_T(X^\an,L^\an) = \PSH_T(Y^\an,\pi^*L^\an)$. By \cite[Theorem A]{BJ24a}, the envelope property holds for $\PSH_T(Y^\an,\pi^*L^\an)$. This concludes the lemma.
\end{proof}

Each rational piecewise linear convex function $f$ on $P$ gives rise to a toric test configuration (see \cite{Do02}), denoted by $(\mcX_f,\mcL_f)$. Up to taking base changes, it suffices to consider test configurations with reduced central fibers.
\begin{proposition}
    As the above notations, for any weight $\rw\in C^\infty(P,\RR)$, then
    \begin{align}
    \label{sing_toric_M_v,w_formula}
        \bfM_{\rv,\rw}^\NA(\phi_{\mcL_f})
    =&2\int_{\partial P}f \rv\, d\sigma
    -c_{\rv,\rw}(L)\int_P f\rw\, dy
    \eqqcolon \mcL_{\rv,\rw}(f).
    \end{align}
\end{proposition}
\begin{proof}
    Let $\pi:Y\rightarrow X$ be a smooth toric resolution of $X$ and $A$ be an ample toric line bundle on $Y$. Denote $L^\epsilon\coloneqq\pi^*L+\epsilon A$ and $P^\epsilon$ the associated convex polytope.
    As constructed in \eqref{pull-back}, one denotes $\mcY$ the fiber product of $\mcX_f$ and $Y_{\PP^1}$. Then $(\mcY,\pi^*\mcL_f)$ is a semiample test configuration of $(Y,\pi^*L)$.
    By the calculation of Proposition~\ref{M_v,w_equal}, \eqref{M_v,w^NA_(mcL)}, and Proposition~\ref{functionals_as_positive_product},
    then
    \begin{align*}
        \bfM_{\rv,\rw}^\NA(\phi_{\mcL_f})
        =\bfM_{\rv,\rw}^\NA(\phi_{\pi^*\mcL_f})
        =(K^{\log}_{\mcY/\PP^1}\cdot (\pi^*\mcL_f)^n)_\rv + {c_{\rv,\rw}(L)}((\pi^*\mcL_f)^{n+1})_{\rw}
    \end{align*}
    
    One writes $\pi^*\mcL_f=\rho^*(\pi^*L)_{\PP^1}+D$ for some vertical divisor $D$. Then one defines $\mcL^\epsilon\coloneqq\rho^*(L^\epsilon)_{\PP^1}+D
    =\pi^*\mcL_f+\epsilon\rho^*A_{\PP^1}$, which is still semiample.
    Thus $(\mcY,\mcL^\epsilon)$ is a toric test configuration of $(Y,L^\epsilon)$.
    By \eqref{M_v,w^NA_(mcL)}, one has
    \begin{align}
    \label{M_v,w_inv_pullback}
        \bfM_{\rv,\rw}^\NA(\phi_{\mcL^\epsilon})
        =(K^{\log}_{\mcY/\PP^1}\cdot (\mcL^\epsilon)^n)_\rv + {c_{\rv,\rw}(L_\epsilon)}((\mcL^\epsilon)^{n+1})_{\rw}.
    \end{align}
    By the definition of the weighted intersection product \eqref{E_L_v}, then as $\epsilon\rightarrow0$,
    \begin{align}
        \label{converg_M_v,w_eps}
 \bfM_{\rv,\rw}^\NA(\phi_{\mcL^\epsilon})
 \rightarrow\bfM_{\rv,\rw}^\NA(\phi_{\pi^*\mcL_f})=\bfM_{\rv,\rw}^\NA(\phi_{\mcL_f}).
    \end{align}

    Note that as in \cite[Appendix B]{BJ22}, the toric test configuration $(\mcY,\mcL^\epsilon)$ induces a rational piecewise linear convex function $f_\epsilon$ on $P^\epsilon$ as follows,
    \begin{align*}
        f_\epsilon(y)\coloneqq\sup_{x\in\RR^n}(\langle x,y\rangle-(\varphi_D+\psi_{P^\epsilon})(x)),
    \end{align*}
    where $\psi_{P^\epsilon}(x)\coloneqq\sup_{y\in P^\epsilon}\langle x,y\rangle$ is the support function of $P^\epsilon$.

    Apply \cite[Proposition 4]{Lah19} to the toric test configuration $(\mcY,\mcL^\epsilon)$, which is an orbifold, we have
\begin{align}
\label{slope_M_v,w_eps}
    \bfM_{\rv,\rw}^\NA(\phi_{\mcL^\epsilon})
    =2\int_{\partial P^\epsilon}f_\epsilon \rv\, d\sigma-c_{\rv,\rw}(L^\epsilon)\int_{P^\epsilon} f_\epsilon\rw\, dy
    \eqqcolon \mcL_{\epsilon;\rv,\rw}(f_\epsilon).
\end{align}
On the other hand, since $f=\sup_{x\in\RR^n}(\langle x,\cdot\rangle-(\varphi_D+\psi_{P})(x))$ and $\psi_{P^\epsilon}\searrow\psi_P$ as $\epsilon\rightarrow0$, then $f_\epsilon\nearrow f$.
Thus $\mcL_{\epsilon;\rv,\rw}(f_\epsilon)\rightarrow\mcL_{\rv,\rw}(f)$.
Together with \eqref{converg_M_v,w_eps} and \eqref{slope_M_v,w_eps}, we obtain the formula \eqref{sing_toric_M_v,w_formula}.
\end{proof}

Denote $P^*$ as the union of $P^0$ with all relative interiors of maximal faces and $\cc_*(P)$ as the set of all lower semicontinuous convex functions which are continuous on $P^*$ and integrable on $\partial P$.
And set
\begin{align*}
    \tilde{\cc_*}(P)
    \coloneqq\{f\in\cc_*(P):\inf_Pf=f(0)=0 \}.
\end{align*}

Modifying the argument of \cite[Theorem~1.0.1]{NS21}, we have the following weighted version.
\begin{proposition}
\label{weighted_NittaSaito}
    Let $(X,L,P)$ be a toric polarized variety. For weight functions $\rv\in C^\infty(P,\RR_{>0})$, $\rw\in C^\infty(P,\RR)$, the followings are equivalent:
    \begin{enumerate}[$(\rm i)$]
        \item $(X,L,P)$ is $\bG$-uniformly $(\rv,\rw)$-weighted K-stable;
        \item There exists  $\delta>0$ such that $\mcL_{\rv,\rw}(f)\geq\delta\int_{\partial P}fd\sigma$ for any $f\in\tilde{\cc_*}(P)$;
        \item $\mcL_{\rv,\rw}(f)\geq0$ for any $f\in\cc_*(P)$ and $\mcL_{\rv,\rw}(f)=0$ iff $f$ is affine.
    \end{enumerate}
\end{proposition}

Suppose $\rw\in C^\infty(P,\RR_{>0})$, then there exists a unique affine function $\ell_\ext$ on $P$, called \textit{weighted extremal function}, satisfying for any $\xi$ affine linear on $P$, 
\begin{align}
\label{extremal_fun_condition}
    2\int_{\partial P}\xi\rv\, d\sigma- \int_P\xi\rw\cdot\ell_\ext\, dx=0.
\end{align}

The following Lemma is essentially \cite[Theorem 1.12]{Li21}.
\begin{lemma}
    Let $(X,L)$ be a polarized projective toric variety with at most klt singularities and $\rv,\rw\in C^\infty(P,\RR_{>0})$. Then $(X,L)$ is $(\CC^{\times})^n$-uniformly $(\rv,\rw\cdot\ell_\ext)$-weighted K-stable is equivalent to $(X,L)$ is $(\CC^{\times})^n$-uniformly $(\rv,\rw\cdot\ell_\ext)$-weighted K-stable for models.
\end{lemma}
\begin{proof}
    We sketch the proof. It suffices to show that $(X,L)$ is $(\CC^{\times})^n$-uniformly $(\rv,\rw\cdot\ell_\ext)$-weighted K-stable implies that $(X,L)$ is $(\CC^{\times})^n$-uniformly $(\rv,\rw\cdot\ell_\ext)$-weighted K-stable for models. It suffices to show that for any toric invariant measure $\mu = \sum_{i}c_i\cdot \delta_{v_i} \in \mcM^\div_\rv$, the solution of $\MA^\NA_\rv(\phi_\mcL) = \mu$ corresponds to a $(\CC^\times)^n$-equivariant test configuration $(\mcX,\mcL)$. A modification of \cite{BPS14, BGJK25} to the weighted Monge-Amp\`ere operator shows that, the solution $\phi_\mcL$ is a piecewise-linear psh function. Therefore it corresponds to a test configuration.
\end{proof}

Let $Y$ be a smooth toric resolution of $X$. Let $L_\epsilon =\pi^* L-\epsilon E$ be an ample line bundle over $Y$, where $E$ is an exceptional divisor.
Denote $P_\epsilon$ as the polytope of $(Y,L_\epsilon)$, $F_\epsilon$ as the maximal faces that converge to lower dimensional strata. A fact is that $P_\epsilon \subset P$.

The bi-linear form $\int_{P_\epsilon}\langle \cdot , \cdot \rangle\; \rw \; dy$ is strictly positive definite for $\epsilon\in [0,\epsilon_0]$. Therefore, there exists a unique affine function $\ell_{\ext,\epsilon}$ that satisfies
\begin{align}
\label{extremal_fun_condition_epsilon}
    2\int_{\partial P_\epsilon}\xi\rv\, d\sigma- \int_{P_\epsilon}\xi\rw\cdot\ell_{\ext,\epsilon}\, dx=0
\end{align}
for any affine functions $\xi$.
Since the coefficients in equation \eqref{extremal_fun_condition_epsilon} are smooth for $\epsilon\in [0,\epsilon_0]$, the affine functions  $\ell_{\ext,\epsilon}$ converge to $\ell_\ext$ smoothly. 

\begin{theorem}
    \label{openness_toric}
     Assume $(X,L)$ is $\bG$-uniformly $(\rv,\rw\cdot\ell_{\ext})$-weighted $K$-stable. Then there exists $\epsilon_0>0$, such that for any $\epsilon<\epsilon_0$, $(Y,L_\epsilon)$ is $\bG$-uniformly $(\rv,\rw\cdot\ell_{\ext,\epsilon})$-weighted $K$-stable.
\end{theorem}

\begin{proof}
    Assume that no such $\epsilon_0$ in the statement of the theorem exists. By applying Proposition~\ref{weighted_NittaSaito} to $\rw\cdot\ell_\ext$, we can pick a sequence $\epsilon_j\to 0$, such that there exists a sequence $f_j\in \tilde{\cc_*}(P_{\epsilon_j})$, ${\mathcal{L}_{\epsilon_j;\rv,\rw\cdot\ell_{\ext,\epsilon_j}}}(f_j)\leq 0$. We can multiply $f_j$ by a positive constant such that
    \begin{align}
        \label{boundary_normal}
        \int_{\partial P_{\epsilon_j}} f_j = 1.
    \end{align}
    Then by
    \begin{align}
        \label{instability_assumption}
         \mathcal{L}_{\epsilon_j;\rv,\rw\cdot\ell_{\ext,\epsilon_j}}(f_j) 
         = 2\int_{\partial P_{\epsilon_j}} f_j \rv\, d\sigma 
          - \int_{P_{\epsilon_j}} f_j \rw\cdot\ell_{\ext,\epsilon_j}\, dy \leq 0,
    \end{align}
    it follows that
    \begin{align*}
        \int_{\partial P_{\epsilon_j}}f_j\, d\sigma
        \leq C_{\rv, \rw}\int_{P_{\epsilon_j}}f_j\,dy,
    \end{align*}
    where $C(\rv, \rw)>0$ is independent of $\epsilon_j$.
    By \cite[Lemma~5.1.3]{Do02} or \cite[Proposition~5.1.2]{NS21}, there exists a constant $C>0$ independent of $\epsilon_j$ such that
    \begin{align*}
        \frac{1}{C}\leq \int_{P_{\epsilon_j}} f_j \; dy \leq C.
    \end{align*}
    On any compact subset $K$ of $P^0$,
    by \cite[Lemma~5.2.3]{Do02}, $f_j$ is uniformly Lipschitz for $j$ sufficiently large. Fix an $0<\epsilon\ll1$, and let $j\gg 1$ such that $\epsilon_j<\epsilon$.
    On $\partial P_\epsilon$,
    \begin{align*}
        \int_{\partial P_\epsilon}f_j d\sigma
        =&\int_{\partial P_\epsilon\cap\partial P_{\epsilon_j}}f_j\, d\sigma +\int_{\partial P_\epsilon\setminus\partial P_{\epsilon_j}}f_j\, d\sigma. 
    \end{align*}
    Let $F\subset \partial P_\epsilon \setminus \partial P_{\epsilon_j}$ be the union of maximal faces. Consider the projection $\tau_j: F\rightarrow \partial P_{\epsilon_j}$, by sending each $y\in F$ to the intersection of the ray $\overrightarrow{0y}$ with $\partial P_{\epsilon_j}$. Since $f_j$ is normalized, $f(\tau_j(y))\geq f(y)$. Then there exists a constant $C>0$ such that $\int_{F}f_j \; d\sigma \leq C\int_{\tau_j(F)} f_j \; d\sigma$. Then
    \begin{align*}
        \int_{\partial P_\epsilon} f_j \; d\sigma \leq C \int_{\partial P_{\epsilon_j}} f_j \; d\sigma = C,
    \end{align*}
    where $C>0$ is a constant independent of the choices of $\epsilon,\epsilon_j, f_j$.
    
   Then by \cite[Proposition~5.2.1]{NS21}, for each fixed $\epsilon>0$, on $P_\epsilon$, up to picking a subsequence, $f_j$ converges to a normalized convex function $f_\epsilon$ in $L^1$. The convergence is also locally uniform in $P^0_\epsilon$, and $0\leq \int_{P_{\epsilon}} f_\epsilon \; dy \leq C$. By letting $\epsilon\to 0$, up to picking a subsequence once more, 
    $f_\epsilon$ converges to $f$ locally uniformly on $P^0$, which extends to a convex function on $P$.

    We claim that $f\in \tilde{C_*}(P)$.
    Since $f_{\epsilon_j}$ converges on $P^*\cap P_{\epsilon}$ and $P^* = \cup_{\epsilon}(P^*\cap P_{\epsilon})$, we have $f\in C^0(P^*)$.
    Since $\int_{\partial P\cap \partial P_{\epsilon_j}} f_j \; d\sigma \leq 1$, we have
    \begin{align}
    \label{boundary_inequality}
        \int_{\partial P\cap \partial P_{\epsilon_j}} f \; d\sigma
        \leq \int_{\partial P\cap \partial P_{\epsilon_j}}\lim f_j \; d\sigma
        \leq \liminf \int_{\partial P\cap \partial P_{\epsilon_j}} f_j \; d\sigma \leq 1.
    \end{align}
    Therefore 
    \begin{align*}
         \int_{\partial P} f \; d\sigma \leq 1.
    \end{align*}
    And since each $f_{\epsilon_j}$ is normalized, $f$ is also normalized. Then $f\in \tilde{C}_*(P)$. 

    Meanwhile, by replacing $f$ by $f\cdot\rv$ in \eqref{boundary_inequality}, we have
    \begin{align}
    \label{boundary_inequality_2}
        \int_{\partial P} f\rv\, d\sigma \leq \liminf_{j\to\infty}\int_{\partial P_{\epsilon_j}} f_j\rv\, d\sigma.
    \end{align}
    As $\ell_{\ext,\epsilon_j}$ converges to $\ell_{\ext}$ smoothly, we do not need to worry about the regularity issue caused by $\rw\cdot \ell_{\ext,\epsilon_j}$.
    Let $\eta>0$ be a small number, and $(\partial P)_\eta$ be the $\eta$-neighborhood of $\partial P$. Since for $j$ sufficiently large, 
    \begin{align*}
       \left |\int_{P_{\epsilon_j}\cap (\partial P)_\eta} f_j \cdot \rw\ell_{\ext,\epsilon_j}\; dy\right|
       \leq C\eta\cdot\int_{\partial P_{\epsilon_j}} f_j\;d\sigma \leq C\eta
    \end{align*}
    for some uniform constant $C>0$, we have
    \begin{align*}
       \left |\int_{P\setminus (\partial P)_\eta} f_j \cdot \rw\ell_{\ext,\epsilon_j}\; dy 
       - \int_{P_{\epsilon_j}} f_j \cdot \rw\ell_{\ext,\epsilon_j}\; dy\right| \leq C\eta.
    \end{align*}
    Since $f_j$ converges to $f$ uniformly on $P\setminus (\partial P)_\eta$, we have
    \begin{align}
    \label{polytope_inequality}
        \int_{P} f\rw\cdot\ell_{\ext}\, dy 
        = \lim_{j\to\infty} \int_{P_{\epsilon_j}} f_j \rw\cdot\ell_{\ext,\epsilon_j}\, dy \geq 2C>0,
    \end{align}
    where $C(\rv) = \inf_{P}\rv >0$. The last inequality is by \eqref{instability_assumption} and \eqref{boundary_normal}. Then
    \begin{align}
        \label{interior_lower_bound}
        \int_P f \; dy \geq C \int_P f \rw\cdot \ell_\ext \; dy >0,
    \end{align}
    where $C = \sup_P(\rw\cdot\ell_\ext)^{-1}>0$.
    The normalization condition and $\int_{P} f \; dy >0$ imply that $f$ is not an affine function.
    
    Then by \eqref{boundary_inequality_2} and \eqref{polytope_inequality}, we have
    \begin{align*}
        \mathcal{L}_{\rv,\rw\cdot\ell_{\ext}}(f) \leq \liminf_{j\to\infty} \mathcal{L}_{\epsilon_j;\rv,\rw\cdot\ell_{\ext,\epsilon_j}}(f_{j}) \leq 0.
    \end{align*}
    However, since $f$ is not affine, the corresponding filtration is not induced by a holomorphic vector field. This contradicts with the fact that $(X,L,P)$ is $\bG$-uniformly $(\rv,\rw\cdot\ell_{\ext})$-weighted K-stable.
\end{proof}

%references

\bibliography{Singular-cscK_ref}
	
\end{document}